\documentclass{amsart}

\usepackage{amssymb,amsmath,graphicx}
\usepackage{xcolor}
\usepackage{enumerate}
\newtoks\prt

\newtheorem{thm}{Theorem}[section]

\newtheorem{lemma}[thm]{Lemma} 
\newtheorem{prop}[thm]{Proposition} 
\newtheorem{cor}[thm]{Corollary} 

\newtheorem{obs}[thm]{Observation}
\newtheorem{example}[thm]{Example}

\theoremstyle{definition} 
 
\newtheorem{remark}[thm]{Remark}

\def\eqn#1$$#2$${\begin{equation}\label#1#2\end{equation}} 
 \def\J#1#2#3{ \left\{ #1,#2,#3 \right\} } 
 
\def\A{\mathcal A} 
 
\def\C{\mathcal C}

\def\ce{\mathbb C}

\def\en{\mathbb N} 
\def\er{\mathbb R} 
 
\def\O{\mathbb O}

\def \f{\boldsymbol{f}} 
\def \p{\boldsymbol{p}} 
\def \g{\boldsymbol{g}}
\def \x{\boldsymbol{x}} 
\def \y{\boldsymbol{y}}
\def \z{\boldsymbol{z}}
\def \h{\boldsymbol{h}} 
\def \uu{\boldsymbol{u}} 
\def \vv{\boldsymbol{v}}

\def \ran {\operatorname{ran}}

\def\span{\operatorname{span}}

\def \reg {\partial _{\kern1pt\text{reg}}}

\def\inv{\diamond}

\def\ip#1#2{\left\langle#1,#2\right\rangle}

\def\di{\,\mbox{\rm d}}

\newcommand{\norm}[1]{\left\|#1\right\|}

\newcommand{\wscl}[1]{\overline{#1}^{w^*}}

\newcommand{\abs}[1]{\left|#1\right|}
\newcommand{\setsep}{;\,}

\title{Finite tripotents and finite JBW$^*$-triples}

\author[J. Hamhalter]{Jan Hamhalter}

\author[O.F.K. Kalenda]{Ond\v{r}ej F.K. Kalenda}

\author[A.M. Peralta]{Antonio M. Peralta}

\address{Czech Technical University in Prague, Faculty of Electrical Engineering, Department of Mathematics, Technicka 2, 166 27, Prague 6,
Czech Republic}
\email{hamhalte@math.feld.cvut.cz}
\address{Charles University, Faculty of Mathematics and Physics, Department of
Mathematical Analysis, Sokolovsk{\'a} 86, 186 75 Praha 8, Czech Republic}
\email{kalenda@karlin.mff.cuni.cz}
\address{Departamento de An{\'a}lisis Matem{\'a}tico, Facultad de
Ciencias, Universidad de Gra\-na\-da, 18071 Granada, Spain.}
\email{aperalta@ugr.es}

\thanks{The first two authors were in part supported by the Research Grant GA\v{C}R 17-00941S. The first author was partly supported further by the project OP VVV Center for Advanced Applied Science CZ.02.1.01/0.0/0.0/16\_019/000077.  The third author was partially supported by the Spanish Ministry of Science, Innovation and Universities (MICINN) and European Regional Development Fund project no. PGC2018-093332-B-I00 and Junta de Andaluc\'{\i}a grant FQM375.}



\begin{document}

\begin{abstract}
We study two natural preorders on the set of tripotents in a JB$^*$-triple defined in terms of their Peirce decomposition and weaker than the standard partial order. We further introduce and investigate the notion of finiteness for tripotents in JBW$^*$-triples which is a natural generalization of finiteness for projections in von Neumann algebras. We analyze the preorders in detail using the standard representation of JBW$^*$-triples. We also provide a refined version of this representation -- in particular a decomposition of any JBW$^*$-triple into its finite and properly infinite parts. Since a JBW$^*$-algebra is finite if and only if the extreme points  of its unit ball are just unitaries, our notion of finiteness differs  from the concept of modularity widely used in Jordan structures so far.  The exact relationship of these two notions is clarified in the last section. 
\end{abstract}

\maketitle

\section{Introduction}

A key role in the classification and representation of von Neumann algebras is played by the comparison theory of projections  (i.e., self-adjoint idempotents) introduced by Murray and von Neumann. Recall that two projections $p$ and $q$ in a von Neumann algebra $W$ are called \emph{equivalent} (we write $p\sim q$) if there is $u\in W$ with $p=u^*u$ and $q=uu^*$ (see, e.g., \cite[Definition V.1.2]{Tak}). In this case $u$ is a partial isometry with initial projection $p$ and final projection $q$.

Further, a projection $p\in W$ is called \emph{finite} if the only projection $q\le p$ with $q\sim p$ is $q=p$ (see, e.g., \cite[Definition V.1.15]{Tak}). This notion is then used to define types of von Neumann algebras \cite[Definition V.1.17]{Tak} and to formulate and prove a decomposition theorem for von Neumann algebras \cite[Theorem V.1.19]{Tak}. The von Neumann algebra $W$ is said to be finite, infinite, properly infinite, or purely infinite according to the property of the identity \cite[Definition V.1.16]{Tak}. 

A celebrated result due to R.V. Kadison (see \cite{KadisonAnn51}) states that a norm one element $e$ in a C$^*$-algebra $A$ is an extreme point of the closed unit ball of $A$ if and only if $e$ is a maximal partial isometry (i.e. a complete tripotent in the terminology of this note). In a general C$^*$-algebra there might exist extreme points of its closed unit ball which are not unitary elements. 

Finite von Neumann algebras were geometrically characterized by H. Choda, Y. Kijima, and Y. Nakagami, in the following terms: A von Neumann algebra $W$ is finite if and only if all the extreme points of its closed unit ball are unitary (see  \cite[Theorem 2]{ChodaKijimaNak69} or \cite[Proof of Theorem 4]{mil} in the case of AW$^*$-algebras). Consequently, the following statements are equivalent:\begin{enumerate}[$(a)$]\label{eq geometric characterization of finite projections}\item 
A projection $p$ in $W$ is finite; \item Every extreme point of the closed unit ball of $pWp$ is a unitary in the latter von Neumann algebra.
\end{enumerate}

The aim of this paper is to investigate the phenomenon of finiteness in the wider setting of JBW$^*$-triples. In JBW$^*$-triples there is no natural notion of  projection, so we will deal with tripotents which generalize partial isometries in von Neumann algebras. The equivalence between statements $(a)$ and $(b)$ above is the motivation for our definition of finite tripotents. 

G. Horn and E. Neher established in the eighties a milestone result in the structure theory of JBW$^*$-triples. These results offer a concrete clasification of JBW$^*$-triples as $\ell_{\infty}$-sums of type $I$  and continuous JBW$^*$-triples. Making use of the Murray-von Neumann classification of  W$^*$-algebras and Jordan JBW$^*$-algebras the continuous part is further classified into JBW$^*$-triples of types $II_1$, $II_{\infty}$, and $III$ (see \cite{horn1987classification,horn1988classification}  and the concrete description below). By using the notion of finiteness we modify and refine this representation result.

The paper is organized as follows. 

In the remaining part of the introduction we provide a basic background on JB$^*$-triples, JBW$^*$-triples and tripotents, including the  representation theory of JBW$^*$-triples due to Horn and Neher.
 
In Section~\ref{sec:2} we collect properties of three natural preorders on the set of all tripotents of a JB$^*$-triple. They include the standard partial order on tripotents and two more weaker preorders denoted by $\le_2$
and $\le_0$. Note that the preorder $\le_2$ was used in \cite{hamhalter2019mwnc,HKPP-BF} without introducing a notation.

In Section~\ref{sec:3} we introduce the notion of finiteness of tripotents and JBW$^*$-triples and establish basic properties and characterizations.

Sections~\ref{sec:vN and pV},~\ref{sec:symmetric and antisymmetric} and~\ref{sec:6 - spin and exceptional} are 
devoted to the study of the three preorders and finiteness in the individual summands from the representation recalled in the introductory section.

Section~\ref{sec:final} contains a synthesis of the results, and a refined representation of JBW$^*$-triples. We also include explanation of the relationship of finiteness to modularity which turns out to be its stronger version.

\subsection{JB$^*$-triples}

A {\em JB$^*$-triple} is a complex Banach space $E$ equipped with a continuous mapping $\J{\cdot}{\cdot}{\cdot}:E^3\to E$ (called {\em triple product}) which is symmetric and bilinear in the outer variables and conjugate linear in the second variable and satisfies, moreover, the following properties:
\begin{enumerate}[$(a)$]
    \item $\J xy{\J abc}=\J{\J xya}bc-\J a{\J yxb}c+\J ab{\J xyc}$ for any $x,y,a,b,c\in E$ \hfill ({\em Jordan identity});
    \item For any $a\in E$ the operator $L(a,a):x\mapsto \J aax$ is hermitian with non-negative spectrum;
    \item $\norm{\J xxx}=\norm{x}^3$ for $x\in E$.
\end{enumerate}
Let us recall that an operator $T$ on a Banach space is {\em hermitian} if $\norm{e^{i\alpha T}}=1$ for each $\alpha\in\er$.

By \cite[Fact 4.1.41]{Cabrera-Rodriguez-vol1} any C$^*$-algebra is a JB$^*$-triple when equipped with the triple product \begin{equation}\label{eq JC triple product}\J abc=\frac12(ab^*c+cb^*a).
 \end{equation} The same formula defines a triple product on $B(H,K)$, the space of bounded linear operators between two complex Hilbert spaces. More generally, any closed subspace of a C$^*$-algebra which is stable under the above-defined triple product, is a JB$^*$-triple (cf. \cite[Fact 4.1.40]{Cabrera-Rodriguez-vol1}). However, there are some JB$^*$-triples which are not of this form ({\em exceptional JB$^*$-triples}, see, e.g., \cite[Theorem 7.1.11]{Cabrera-Rodriguez-vol2}).

JB$^*$-triples are introduced and employed to formulate and prove a Riemann mapping theorem for infinite-dimensional complex Banach spaces (see \cite[Theorem 5.4]{kaup1983riemann}). Moreover, in JB$^*$-triples the metric structure and the algebraic structure determine each other, i.e., a linear bijection between two JB$^*$-triples is an isometry if and only if it is a triple-isomorphism (i.e., it preserves the triple product) -- see \cite[Proposition 5.4]{kaup1977-maan}.

A JB$^*$-triple which is a dual Banach space is called a {\em JBW$^*$-triple}. Any JBW$^*$-triple has a unique predual and the triple product is weak$^*$-to-weak$^*$ separately continuous (cf. \cite{BaTi}). Preduals of JBW$^*$-triples share many properties of von Neumann algebra preduals, the latter are often called non-commutative $L^1$ spaces. This is illustrated, among others, by recent structure results \cite{BHK-vN,BHK-JBW,BHKPP-triples,hamhalter2019mwnc}.

In the investigation of von Neumann algebras, one of the key tools is the study of projections. Projections in a von Neumann algebra form a complete lattice and are used, among others, to classify von Neumann algebras, cf. \cite[Chapter V]{Tak}. In a JB$^*$-triple there is no natural order structure and no natural notion of projection. Instead, an important role is played by {\em tripotents}, i.e., elements $u$ satisfying $u=\J uuu$. In a C$^*$-algebra a tripotent is just a partial isometry.

\subsection{Basic facts on tripotents}

Let $E$ be a JB$^*$-triple and let $u\in E$ be a tripotent. Then the operator $L(u,u)$ (recall that it is defined by $L(u,u)x=\J uux$ for $x\in E$) has eigenvalues contained in the set $\{0,\frac12,1\}$ (see \cite[Fact 4.2.14]{Cabrera-Rodriguez-vol1}). This inspires the definition of the {\em Peirce subspaces}
$$E_j(u)=\left\{x\in E\setsep \J uux=\frac j2 x\right\}\mbox{ for }j=0,1,2.$$
It is known that $E=E_2(u)\oplus E_1(u)\oplus E_0(u)$ and that the canonical projections (called {\em Peirce projections} and denoted by $P_j(u)$, $j=0,1,2$) have norm one at most one (see \cite[Corollary 1.2]{Friedman-Russo}).

Moreover, it is easy to check, that
\begin{equation}
\J{E_j(u)}{E_k(u)}{E_l(u)}\subset E_{j-k+l}(u),
\end{equation}
where the right-hand side is defined to be $\{0\}$ if $j-k+l\notin\{0,1,2\}$.
Moreover, it is known (but not obvious) that
\begin{equation}
\J{E_2(u)}{E_0(u)}E=\J{E_0(u)}{E_2(u)}E=\{0\}\ \ \hbox{(cf. \cite{loos1977bounded}).}     
\end{equation}
The two above rules are known, and will be referred to, as the {\em Peirce arithmetics} or the {\em Peirce calculus}. It easily follows that $E_j(u)$ is a JB$^*$-subtriple of $E$ for $j=0,1,2$. 

The following formulas for the Peirce projections may be easily deduced from the definitions.
\begin{equation}
\begin{aligned}
P_2(u)x&=2L(u,u)^2 x-L(u,u)x,\\
P_1(u)x&=4(L(u,u)x-L(u,u)^2 x),\\
P_0(u)x&=x-3L(u,u)x+2L(u,u)^2 x.
\end{aligned}    
\end{equation}
Another useful formula for $P_2(u)$ is
\begin{equation}
    P_2(u)x=Q(u)^2 x \mbox{ where }Q(u)x=\J uxu\mbox{ for }x\in E,
\end{equation}
see \cite[Lemma 4.2.20]{Cabrera-Rodriguez-vol1}

A tripotent $u$ is called {\em complete} (or {\em maximal}) if $E_0(u)=\{0\}$ and it is called  {\em unitary} if $E=E_2(u)$. Further, a nonzero tripotent $e$ is called {\em minimal} if $E_2(e)$ has dimension one, i.e., if $E_2(e)=\ce e$.

In a JB$^*$-triple there need not be any complete tripotents (in fact, there need not be any nonzero tripotent, take for example the non-unital C$^*$-algebra $\C_0(\er)$); but in a JBW$^*$-triple there is an abundance of complete tripotents, as they are exactly the extreme points of the unit ball (see \cite[Lemma 4.1]{braun1978holomorphic} and \cite[Proposition 3.5]{kaup1977jordan}).

On the other hand, JBW$^*$-triples need not contain any unitary element. For example, the space of $1\times 2$ complex matrices
(with the structure of the space of linear functionals on the two-dimensional Hilbert space equipped with the triple product given in \eqref{eq JC triple product}) is a JBW$^*$-triple without unitary elements. In fact, JB$^*$-triples with a unitary element are just the triples coming from unital JB$^*$-algebras \cite[Theorem~4.1.55]{Cabrera-Rodriguez-vol1}. This relationship is explained in the next subsection.

\subsection{JB$^*$-triples and JB$^*$-algebras}

Recall that a {\em Jordan Banach algebra} is a (real or complex) Banach space $B$ equipped with a product $\circ$ satisfying the following properties.
\begin{enumerate}[$(a)$]
    \item $(B,\circ)$ is a (possibly) non-associative real or complex algebra;
    \item $x\circ y=y\circ x$ for $x,y\in B$;
    \item $x^2\circ (y\circ x)=(x^2\circ y)\circ x$ for $x,y\in B$, where $x^2= x\circ x$;
    \item $\norm{x\circ y}\le\norm{x}\norm{y}$ for $x,y\in B$.
\end{enumerate}

Further, a {\em JB$^*$-algebra} is a complex Jordan Banach algebra $(B,\circ)$ equipped moreover with an involution $*$ satisfying the condition    
    \begin{enumerate}[$(e_*)$]
       \item $\norm{2(x\circ x^*)\circ x-x^2\circ x^*}=\norm{x}^3$ for $x\in B$.
\end{enumerate}
 It is shown in \cite[Lemma 4]{youngson1978vidav} (see also \cite[Proposition 3.3.13]{Cabrera-Rodriguez-vol1}) that the involution of every JB$^*$-algebra is an isometry. 
 
 A real counterpart is formed by JB-algebras. We recall that a JB-algebra is a real Jordan Banach algebra  $(B,\circ)$ satisfying
 \begin{enumerate}[$(a)$]\setcounter{enumi}{4}
          \item $\norm{x}^2\le \norm{x^2 +y^2}$ for $x,y\in B$.
 \end{enumerate}
 It is known that a (unital) real Jordan Banach algebra is a JB-algebra if and only if the norm closed (Jordan) subalgebra generated by a single element (and the unit) is isometrically isomorphic to the self-adjoint part of a commutative C$^*$-algebra (cf. \cite[Proposition 3.2.6]{hanche1984jordan}). 
 
JB-algebras and JB$^*$-algebras are closely related -- the self-adjoint part of any JB$^*$-algebra is a JB-algebra (cf. \cite[Proposition 3.8.2]{hanche1984jordan} or \cite[Corollary 3.4.3]{Cabrera-Rodriguez-vol1}) and, conversely, any JB-algebra is of this form (see \cite[Theorem 2.8]{Wright1977} or \cite[Theorem 3.4.8]{Cabrera-Rodriguez-vol1}). 
 
 Any C$^*$-algebra becomes a JB$^*$-algebra if equipped with the Jordan product $x\circ y=\frac12(xy+yx)$.
 More generally, any closed subspace of a C$^*$-algebra which is stable under involution and the Jordan product is a JB$^*$-algebra. A JB$^*$-algebra of this form is called a {\em JC$^*$-algebra}. 
 Similarly, a {\em JC-algebra} is the self-adjoint part of a JC$^*$-algebra, or, equivalently, a closed Jordan subalgebra of the self-adjoint part of a C$^*$-algebra.
 
 There are some JB$^*$-algebras which are not JC$^*$-algebras (the so-called {\em exceptional} JB$^*$-algebras).

JB$^*$-algebras are closely related to JB$^*$-triples. On one hand, any JB$^*$-algebra becomes a JB$^*$-triple when equipped with the triple product
\begin{equation}
    \J xyz=(x\circ y^*)\circ z+x\circ(y^*\circ z)-(x\circ z)\circ y^*, 
\end{equation}
cf. \cite[Theorem 3.3]{braun1978holomorphic} or \cite[Theorem 4.1.45]{Cabrera-Rodriguez-vol1}.
Note that condition $(e_*)$ from the definition of a JB$^*$-algebra yields condition $(c)$ from the definition of a JB$^*$-triple.

Conversely, if $E$ is a JB$^*$-triple with a unitary element $u$, it becomes a unital JB$^*$-algebra when equipped with the operations
\begin{equation}\label{eq circ-star}
x\circ y=\J xuy\mbox{ and }x^*=\J uxu\qquad\mbox{for }x,y\in E,
\end{equation}
see \cite[Theorem 4.1.55]{Cabrera-Rodriguez-vol1}.In this case $u$ is the unit of this JB$^*$-algebra (i.e., $u\circ x=x$ for $x\in E$).

Note that while the triple product is uniquely determined by the structure of a JB$^*$-algebra, the converse is not true. The Jordan product and involution depend on the choice of the unitary element $u$. Therefore, we denote these operations by $\circ_u$ and $*_u$ (instead of $\circ$ and $*$ as in \eqref{eq circ-star}) whenever we need to stress this dependence.

In particular, if $u$ is any (nonzero) tripotent in a JB$^*$-triple $E$, then $u$ is unitary in the subtriple $E_2(u)$. Thus $E_2(u)$ may be viewed as a unital JB$^*$-algebra with the above operations. The properties of the JB$^*$-algebra $E_2(u)$ somehow reflect the properties of $u$. In particular, a tripotent $u$ is called \emph{abelian} if the JB$^*$-algebra $E_2(u)$ is associative. In this case it is even an abelian C$^*$-algebra (this easily follows from the axioms). Obviously, any minimal tripotent is abelian.

Similarly as for triples, 
 a \emph{JBW$^*$-algebra} ({\em JBW-algebra}) is a JB$^*$-algebra (JB-algebra, respectively) which is a dual Banach space. The predual is again unique and, moreover, JBW-algebras are precisely self-adjoint parts of JBW$^*$-algebras. 
 
 Finally, a {\em JW$^*$-algebra} is a weak$^*$-closed Jordan-$*$ subalgebra of a von Neumann algebra, and a {\em JW-algebra} is a weak$^*$-closed Jordan subalgebras of the self-adjoint part of a von Neumann algebra (or, equivalently, the self-adjoint part of a JW$^*$-algebra).
 
\subsection{Ideals  and direct summands of JB$^*$-triples} 

A linear subspace $I$ of a JB$^*$-triple $E$ is said to be an
\emph{ideal} if $\J IEE\subset I$ and $\J EIE\subset I$ (see \cite{horn1987ideal}).

Further, $I$ is said to be a
\emph{direct summand} of $E$ if there is a linear subspace $J\subset E$
such that
\begin{enumerate}[$(i)$]
    \item $I\cap J=\{0\}$ and $E=I+J$;
    \item $\J{a_1+b_1}{a_2+b_2}{a_3+b_3}=\J{a_1}{a_2}{a_3}+\J{b_1}{b_2}{b_3}$ whenever $a_1,a_2,a_3\in I$ and $b_1,b_2,b_3\in J$.
\end{enumerate}

It is clear that in this case both $I$ and $J$ are ideals. Conversely, if $I$ and $J$ are ideals satisfying $(i)$, then $(ii)$ is also satisfied. Hence, it follows from \cite[Lemma 4.3]{horn1987ideal} that any direct summand of a JB$^*$-triple is closed and any direct summand of a JBW$^*$-triple is weak$^*$-closed. Moreover, \cite[Lemma 4.3 and 4.4]{horn1987ideal} shows that in this case we have $E=I\oplus^{\ell_\infty} J$ (see also \cite[Facts 5.7.23 and 5.7.24]{Cabrera-Rodriguez-vol2}).

Further observe, that if $I$ is a direct summand of $E$, the subspace $J$ from the definition is uniquely determined. Indeed, this is an easy consequence of \cite[Lemma 5.7.22]{Cabrera-Rodriguez-vol2} using the previous paragraph. Hence, to any direct summand $I$ we may associate the canonical projection of $E$ onto $I$ (along $J$). Let us denote this projection by $P_I$.

The following easy proposition provides an alternative characterization of direct summands. Its proof follows from the preceding comments and from the fact that every direct summand is weak$^*$-closed and contains a complete tripotent. 

\begin{prop}\label{P:direct summand}
Let $M$ be a JBW$^*$-triple and let $I\subset M$ be a linear subspace.
\begin{enumerate}[$(a)$]
    \item $I$ is a direct summand of $M$ if and only if there is a tripotent $u\in M$ such that
    $$I=M_2(u)+M_1(u) \mbox{ and } \J{M_1(u)}{M_2(u)}{M_1(u)}=\J{M_0(u)}{M_1(u)}{M}=\{0\};$$
    \item $I$ is a direct summand of $M$ isomorphic to a JBW$^*$-algebra if and only if there is a tripotent $u\in M$ such that
    $$I=M_2(u)\mbox{ and }M_1(u)=\{0\}.$$
\end{enumerate}
\end{prop}

\begin{proof} See \cite[Lemma 3.12]{BattagliaOrder1991}.
\end{proof}

\subsection{Representation of JBW$^*$-triples}\label{subsec:representation}

In this subsection we recall the representation of JBW$^*$-triples provided by \cite{horn1987classification,horn1988classification}. We start by recalling some terminology from von Neumann algebras.

A projection $p$ in a von Neumann algebra $W$ is said to be \emph{abelian} if $p W p$ is commutative \cite[Definition 2.2.6]{Sak71}. Observe that a projection is abelian if and only if it is abelian as a tripotent. A von Neumann algebra is said to be of \emph{type I} or \emph{discrete} if every nonzero (central)
projection contains a nonzero abelian subprojection. 
(Equivalence of the two definitions follows from \cite[Corollary 4.20]{stra-zsi}.)  A von Neumann algebra having no nontrivial discrete von Neumann algebra as a direct summand (or, equivalently, having no nonzero abelian projection) is said to be \emph{continuous} (see \cite[Definition 2.2.9]{Sak71}).

Motivated by the notions in von Neumann algebra theory, a JBW$^*$-triple $M$ is said to be of \emph{type $I$} (respectively, \emph{continuous}) if it coincides with the weak$^*$ closure of the span of all its abelian tripotents (respectively, it contains no non-zero abelian tripotents) (cf. \cite{horn1987classification,horn1988classification}).

Due to the results of \cite{horn1987classification,horn1988classification} any JBW$^*$-triple $M$ may be represented in the form
\begin{equation}\label{eq:representation of JBW* triples}
 \left(\bigoplus_{j\in J}^{\ell_\infty}A_j\overline{\otimes}C_j\right)\oplus^{\ell_\infty}H(W,\alpha)\oplus^{\ell_\infty}pV,    
\end{equation}
where the $A_j$'s are abelian von Neumann algebras, the $C_j$'s are Cartan factors, $W$ and $V$ are continuous von Neumann algebras, $p\in V$ is a projection, $\alpha$ is a linear involution on $W$ commuting with $*$ and $H(W,\alpha)=\{x\in W\setsep\alpha(x)=x\}$. The first $\ell_{\infty}$-sum between brackets in \eqref{eq:representation of JBW* triples} identifies with the type $I$ part of $M$, while the last two summands give a concrete representation of the continuous part of $M.$

Let us explain this representation a bit. There are six types of Cartan factors, details on them can be found in \cite{loos1977bounded, kaup1981klassifikation, kaup1997real} and the precise definitions will be recalled below in Section~\ref{sec:vN and pV} (type 1), Subsection~\ref{subsec:involutions from conjugations} (type 2 and 3) and Section~\ref{sec:6 - spin and exceptional} (type 4, 5, 6). 
If $C_j$ is a Cartan factor of type 1--4, it can be represented as a weak$^*$-closed JB$^*$-subtriple of a von Neumann algebra $M_j$. In this case $A_j\overline{\otimes}C_j$ is defined as the weak$^*$-closure of the algebraic tensor product $A_j\otimes C_j$ in the von Neumann tensor product $A_j\overline{\otimes}M_j$ (see \cite[Section IV.1]{Tak}).
If $C_j$ is a Cartan factor of type 5 or 6, it has finite dimension and $A_j\overline{\otimes}C_j$ is defined as the completed injective tensor product (see \cite[Chapter 3]{ryan}).

Further, by \cite[Theorem 6.4.1]{DualC(K)} any abelian von Neumann algebra $A$ can be represented as
$$A=\bigoplus_{\theta\in\Lambda} L^\infty(\mu_\theta),$$
where each $\mu_\theta$ is a 
Radon probability measure on a (hyperstonean) compact space $X_\theta$ such that the algebra  $C(X_\theta)$ of continuous functions on $L$ is canonically isometric to $L^\infty(\mu_\theta)$. Hence we may and shall assume that each $A_j$ is of the form of an individual summand, i.e., 
$$A_j=C(X_j)\cong L^\infty(\mu_j),$$
where $X_j$ is a compact space, $\mu_j$ a Radon probability on $X_j$ and $C(X_j)\cong L^\infty(\mu_j)$.

Hence, if $C_j$ is a Cartan factor of type 5 or 6, by \cite[Section 3.2]{ryan} we have
$$A_j\overline{\otimes}C_j=C(X_j,C_j)\cong L^\infty(\mu_j,C_j),$$
in particular the predual can be expressed as
$$(A_j\overline{\otimes}C_j)_*=L^1(\mu_j,(C_j)_*)=L^1(\mu_j,C_j^*).$$
Indeed, since $C_j$ is finite-dimensional, by \cite[Theorem IV.1.1]{Diestel-Uhl} we deduce that $L^1(\mu_j,C_j^*)^*=L^\infty(\mu_j,C_j)$ and we conclude by uniqueness of the predual of a JBW$^*$-triple.

If $C_j$ is a Cartan factor of type 1--4, we also have
$$(A_j\overline{\otimes}C_j)_*=L^1(\mu_j,(C_j)_*).$$
This follows from \cite[Proposition 4.8]{BHKPP-triples} and \cite[Example 2.19]{ryan}. Hence, in case $C_j$ is reflexive, it follows from \cite[Theorem IV.1.1]{Diestel-Uhl} that $A_j\overline{\otimes}C_j=L^\infty(\mu_j,C_j)$.
If $C_j$ is not reflexive, the representation is more complicated.

The previous two paragraphs yield in particular the following observation.

\begin{lemma}\label{L:tensor=Bochner}
Let $C$ be a reflexive Cartan factor. Then $L^\infty(\mu)\overline{\otimes}C$ is canonically isometric to $L^\infty(\mu,C)$ for any probability measure $\mu$.
\end{lemma}

Let us remark that some modified versions of the representation \eqref{eq:representation of JBW* triples} were used in \cite[Remark 5.4]{BHKPP-triples}, \cite[Sections 9 and 10]{hamhalter2019mwnc}
or in \cite{HKPP-BF}. In Theorem~\ref{T:synthesis} below we give a refined version of this representation.

\section{Three preorders on the set of tripotents}\label{sec:2}

In this section we define three basic preorders on the set of tripotents and collect their properties. We start by the definition of orthogonality of tripotents, so we recall that two tripotents $e,u$ in a JB$^*$-triple $E$ are orthogonal (denoted by $e\perp u$) if $u\in E_0(e)$
(i.e., $\{e,e,u\}=0$). The following lemma gathers some equivalent reformulations.

\begin{lemma}\label{L:OG trip}
Let $E$ be a JB$^*$-triple, and let $e,u$ be two tripotents in $M$.
The following assertions are equivalent:
\begin{enumerate}[$(1)$]
\item $e\perp u$;
\item $e\in E_0(u)$;
\item $E_2(u)\subset E_0(e)$ and $E_2(e)\subset E_0(u)$;
\item $P_2(u)P_0(e)=P_0(e)P_2(u)=P_2(u)$ and $P_0(u)P_2(e)=P_2(e)P_0(u)=P_2(e)$;
\item $L(e,u)=0$;
\item $L(u,e)=0$;
\item Both $u+e$ and $u-e$ are tripotents.
\end{enumerate}
\end{lemma}

\begin{proof} The equivalences are essentially known in the literature. More precisely, the equivalence of the assertions $(1)$--$(6)$ follows by the Peirce calculus (see, e.g.,  \cite[Lemma 3.9]{loos1977bounded} for the equivalence of $(1)$, $(2)$, $(5)$, $(6)$
and  \cite[Proposition 6.7]{hamhalter2019mwnc} for the equivalence of $(1)$--$(4)$).

The equivalence between $(1)$ and $(7)$ is proved in \cite[Lemma 3.6]{isidro1995real} even in the more general setting of real JB$^*$-triples. (Note that $(1)\Rightarrow(7)$ follows already from 
\cite[Lemma 3.9]{loos1977bounded} applied to the pairs $u$, $e$ and $u$, $-e$; and that every JB$^*$-triple is a real JB$^*$-triple with the same set of tripotents.)
\end{proof}

We present next the definition of three basic preorders on the set of tripotents. Let $E$ be a JB$^*$-triple and let $e,u\in E$ be two tripotents.
We say that
\begin{itemize}
    \item $u\le e$ if $e-u$ is a tripotent orthogonal to $u$;
    \item $u\le_2 e$ if $u\in E_2(e)$;
    \item $u\le_0 e$ if $E_0(e)\subset E_0(u)$.
\end{itemize}

Here $\le$ is the standard partial order on tripotents used in \cite{Friedman-Russo} (and references therein), $\le_2$ is the preorder used in \cite[Sections 6 and 7]{hamhalter2019mwnc}  and \cite{HKPP-BF}.
The relation $\le_0$ is newly defined. It is related to the ordering of the seminorms defining the strong$^*$ topology (cf. \cite[\S 3]{HKPP-BF} and \cite[\S 6.3 and 6.4]{hamhalter2019mwnc}).

The relation $\le$ is a partial order -- it is well known and it will be recalled below. Relations $\le_2$ and $\le_0$ are preorders -- reflexive and transitive, but not antisymmetric. The reflexivity is obvious, the transitivity of $\le_0$ is clear, that of $\le_2$ is proved below. If $u\le_2 e$ and $e\le_2 u$, we will write $u\sim_2 e$. In \cite[Remark 1.3]{dangfriedmanclassification87} tripotents $e,u$ satisfying $u\sim_2 e$ are called ``equivalent''. If $u\le_2 e$ and $e\not\le_2 u$, we write $u<_2 e$. The relations 
$\sim_0$ and $<_0$ have the analogous meaning.

The following proposition summarizes characterizations of the relation $\le_2$ proved in \cite[Proposition 6.4]{hamhalter2019mwnc}. (Note that the implication $(1)\Rightarrow(3)$ is proved already in \cite[Lemma 1.14(1)]{horn1987ideal}.)

\begin{prop}\label{P:M2 inclusion} Let $E$ be a JB$^*$-triple and $e,u$ be two tripotents in $E$.
The following assertions are equivalent:
\begin{enumerate}[$(1)$]
\item $u\le_2 e$;
\item $P_2(u)P_2(e)=P_2(e)P_2(u)=P_2(u)$, $P_1(u)P_1(e)=P_1(e)P_1(u)$ and
$P_0(u)P_0(e)=P_0(e)P_0(u)=P_0(e)$;
\item $E_2(u)\subset E_2(e)$ and $E_0(e)\subset E_0(u)$;
\item $E_2(u)\subset E_2(e)$.
\end{enumerate}
\end{prop}

The next proposition provides a characterization of the equivalence relation $\sim_2$. 

\begin{prop}\label{P:M2 equality}
Let $E$ be a JB$^*$-triple. Then $\le_2$ is a preorder on the set of all tripotents in $E$. Moreover, given tripotents $e,u\in E$ the following assertions are equivalent:
\begin{enumerate}[$(1)$]
\item $u\sim_2 e$;
\item $E_2(e)=E_2(u)$;
\item The Peirce decompositions induced by $e$ and $u$ coincide (i.e., $P_j (e) = P_j(u)$ for all $j=0,1,2$);
\item $L(e,e)=L(u,u)$.
\end{enumerate}
\end{prop}

\begin{proof} Using the condition $(4)$ of Proposition~\ref{P:M2 inclusion} it follows that $\le_2$ is a preorder.

The equivalence of (1), (2) and (3) is proved in \cite[Proposition 6.5]{hamhalter2019mwnc} (it follows already from  \cite[Lemma 1.14(2)]{horn1987ideal}).

$(3)\Rightarrow(4)$ Observe that
$$L(e,e)x=P_2(e)x+\frac12P_1(e)x,\qquad x\in E,$$
and a similar formula holds for $L(u,u)$.

$(4)\Rightarrow(3)$ This is obvious.
\end{proof}

Now we collect characterizations of the standard order on tripotents.

\begin{prop}\label{P:order char} Let $u,e$ be two tripotents in a JB$^*$-triple $E$. The following assertions are equivalent. 
\begin{enumerate}[$(i)$]
    \item $u\le e$;
    \item $u=\J ueu$;
    \item $u=\J uue$;
    \item $u=P_2(u)e$;
    \item $L(e-u,u)=0$;
    \item $L(u,e-u)=0$;
    \item $u$ is a projection in the JB$^*$-algebra $E_2(e)$;
    \item $E_2(u)$ is a JB$^*$-subalgebra of $E_2(e)$.
\end{enumerate}
\end{prop}

\begin{proof} The equivalence between $(i)$ and $(iv)$ and $(i)$ and $(vii)$ are established in \cite[Corollary 1.7]{Friedman-Russo} and \cite[Lemma 3.5$(i)$]{BattagliaOrder1991}, respectively. The other equivalences are part of the folklore in JB$^*$-triples theory. For the sake of self-containing, we shall revisit the remaining implications.

The implications 
$(i)\Rightarrow(v)\&(vi)$
 follow from the definition (using Lemma~\ref{L:OG trip}). The implication $(viii)\Rightarrow(vii)$ is obvious.

$(vi)\Rightarrow(ii)$ If $L(u,e-u)=0$, then in particular
$$0=\J u{e-u}u=\J ueu-\J uuu =\J ueu -u.$$

$(v)\Rightarrow(iii)$ If $L(e-u,u)$, then in particular
$$0=\J {e-u}uu=\J euu-\J uuu =\J uue -u.$$

$(ii)\Rightarrow(iv)$ Note that $(ii)$ means that $u=Q(u)e$ and $P_2(u)=Q(u)^2$, so, assuming $(ii)$ we have
$$P_2(u)e=Q(u)^2e=Q(u)u=u.$$

$(iii)\Rightarrow(iv)$ Note that $(iii)$ means that
$u=L(u,u)e$ and $P_2(u)=2L(u,u)^2-L(u,u)$, this assuming $(iii)$ we have
$$P_2(u)e=2L(u,u)^2e-L(u,u)e=2L(u,u)u-u=2u-u=u.$$

$(i)\Rightarrow(viii)$ It follows from Peirce arithmetic that $E_2(u)\subset E_2(e)$ and $e-u \perp E_2(u)$. Given $x,y\in E_2(u)$ we have
$$x\circ_e y=\J xey=\J x{e-u +u}y=\J xuy=x\circ_u y,$$ and
$$x^{*_e}=\J exe=\J{e-u +u }x{e- u + u}=\J uxu=x^{*_u}.$$
\end{proof}

The following proposition collects some more properties of the relation $\le$.

\begin{prop}\label{P:order properties}
Let $E$ be a JB$^*$-triple. The relation $\le$ is a partial order on the set of all tripotents in $E$. Moreover, given tripotents $u,v,e\in E$ the following holds.
\begin{enumerate}[$(a)$]
    \item If $u\le e$, then $\alpha u\le \alpha e$ for any complex unit $\alpha$;
    \item If $u\le e$, $v\le e$ and $u,v$ are orthogonal, then $u+v\le e$.
 \end{enumerate}
\end{prop}

\begin{proof} All the properties follow using proper conditions from Proposition~\ref{P:order char}.
More precisely, using the condition $(viii)$  we see that $\le$ is indeed a partial order. Assertion $(a)$ follows for example using  condition $(ii)$. 

Let us prove  assertion $(b)$. If $u,v$ are orthogonal, $u+v$ is a tripotent by Lemma~\ref{L:OG trip}. Moreover, 
$$\J{u+v}{u+v}e=\J uue + \J uve+\J vue + \J vve =u+v,$$
where we used  condition $(iii)$ of Proposition~\ref{P:order char} and Lemma~\ref{L:OG trip}.
\end{proof}

If $E$ is a JB$^*$-triple and $B\subset E$ is a JB$^*$-subtriple, it does not matter whether the relations $\le$ and $\le_2$ for tripotents in $B$ are considered in $B$ or in $E$. For the relation $\le_0$ this is more complicated as witnessed by the following proposition.

\begin{prop}\label{P:le0 in a subtriple} Let $E$ be a JB$^*$-triple, $B\subset E$ its JB$^*$-subtriple and $u,e\in B$ two tripotents. Then the following holds.
\begin{enumerate}[$(a)$]
    \item If $u\le_0 e$ in $E$, then $u\le_0 e$ in $B$;
    \item It may happen that $u\sim_0 e$ in $B$ and $u$ and $e$ are incomparable for $\le_0$ in $E$;
    \item It may happen that $u<_0 e$ in $B$ and $u$ and $e$ are incomparable for $\le_0$ in $E$.
    \item It may happen that $u\sim_0 e$ in $B$ and $u<_0 e$ in $E$. 
\end{enumerate}
\end{prop}

\begin{proof}
$(a)$ This is obvious, as $B_0(e)=B\cap E_0(e)$ and $B_0(u)=B\cap E_0(u)$.

$(b)$ Let $E=M_2$ and let $B$ be formed by the matrices with zero second row. Set 
$$e=\begin{pmatrix}1&0\\0&0\end{pmatrix}, u=\begin{pmatrix}0&1\\0&0\end{pmatrix}.$$
Then both $e$ and $u$ are complete tripotents in $B$, thus $B_0(e)=B_0(u)=\{0\}$.
However, $E_0(e)$ and $E_0(u)$ are clearly incomparable.

$(c)$ Let $E=M_3$ and let $B$ be formed by the matrices with zero third row. Set
$$e=\begin{pmatrix}1&0&0\\0&1&0\\0&0&0\end{pmatrix}, u=\begin{pmatrix}0&0&1\\0&0&0\\0&0&0\end{pmatrix}.$$
Then $e,u\in B$ are tripotents, $e$ is complete and $u$ is not. Thus $u<_0 e$ in $B$.
However, $E_0(e)$ and $E_0(u)$ are clearly incomparable.

$(d)$ Let $H=\ell^2$ and let $(\xi_n)$ denote the canonical orthonormal basis. Let $E=B(H)$. Define $e$ to be the forward shift and $B=E_2(e)$. Let $u$ be the projection onto $\overline{\span}\{\xi_n\setsep n\ge2\}$. Then $e,u\in B$, $e$ is unitary in $B$ and $u$ is complete in $B$. Thus $u\sim_0 e$ in $B$.

On the other hand, $e$ is complete in $E$, but $u$ is not complete in $E$. Thus $u<_0 e$ in $E$.
\end{proof}

\begin{prop}\label{P:impl le20}
Let $E$ be a JB$^*$-triple and $u,e\in E$ be two tripotents. Then
$$u\le e\Rightarrow u\le_2 e\Rightarrow u\le_0 e.$$
If $E$ is a JB$^*$-algebra and $u,e\in E$ are projections, then all three relations are equivalent.

In general, the converse implications fail. The converse of the first implication fails even in $\ce$, the converse to the second one fails for example in the triple of $1\times 2$ complex matrices or in $B(H)$ if $\dim H=\infty$.
\end{prop}

\begin{proof} The implication $u\le e\Rightarrow u\le_2 e$ is obvious. The implication $u\le_2 e\Rightarrow u\le_0 e$ follows using Proposition~\ref{P:M2 inclusion}$(3)$.

Let us continue by showing some counterexamples to converse implications.
In $\ce$ we have $-1\sim_2 1$, but $-1\not\le 1$. In the space of $1\times 2$ matrices the tripotents $e=(1,0)$ and $u=(0,1)$ are both complete, hence $e\sim_0 u$. However, their Peirce-2 subspaces are incomparable. Similarly, if $\dim H=\infty$, then there is an isometry $u:H\to H$ which is not onto. Then $u$ is a complete tripotent, hence $u\sim_0 1$ (where $1$ is the identity, i.e., the unit of $B(H)$). However, since $1$ is unitary and $u$ not, we have
$1\not\le_2 u$ (in fact $u<_2 1$).

Finally, assume that $E$ is a JB$^*$-algebra and $u,e\in E$ are projections such that $u\le_0 e$.

Assume first that $E$ is unital. Then $1-e\in E_0(e)\subset E_0(u)$, so $L(1-e,u)=0$ by Lemma~\ref{L:OG trip}. Further, as $1-u\in E_0(u)$, another use of Lemma~\ref{L:OG trip} yields $L(1-u,u)=0$. By subtracting we get $L(e-u,u)=0$, so $u\le e$ by Proposition~\ref{P:order char}.

If $E$ is not unital, we pass to the bidual. Recall that $E^{**}$ is a JBW$^*$-algebra, hence it is unital. Moreover, $u$ and $e$ are projections in $E^{**}$. 

Observe $E_0(u)$ is weak$^*$ dense in $E^{**}_0(u)$. Indeed, fix any $x^{**}\in E^{**}_0(u)$. By Goldstine theorem there is a (bounded) net $x_\nu$ in $E$ weak$^*$-converging to $x^{**}$. Since $P_0(u)$ (acting on $E^{**}$) is weak$^*$-to-weak$^*$ continuous, we deduce that $P_0(u)x_\nu$ weak$^*$ converges to $P_0(u)x^{**}=x^{**}$. Since $P_0(u)x_\nu\in E_0(u)$ for each $\nu$ the density follows. 

The same holds for $e$. Since the Peirce-$0$ subspaces in $E^{**}$ are weak$^*$-closed, the inclusion $E_0(e)\subset E_0(u)$ implies $E^{**}_0(e)\subset E^{**}_0(u)$, i.e., $u\le_0 e$ also in $E^{**}$. Since $E^{**}$ is unital, we deduce that $u\le e$.
\end{proof}

The following proposition on tripotents in a direct sum of JB$^*$-triples is obvious.

\begin{prop}\label{P:preorders in direct sum}
Let $(E_j)_{j\in J}$ be a family of JB$^*$-triples. Set
$\displaystyle E=\bigoplus_{j\in J}^{\ell_\infty} E_j.$
Then the following holds.
\begin{enumerate}[$(1)$]
    \item $E$ is a JB$^*$-triple with the triple product defined coordinatewise.
    \item $(u_j)_{j\in J}$ is a tripotent in $E$ if and only if $u_j$ is a tripotent in $E_j$ for each $j\in J$.
    \item Let $(u_j)_{j\in J}$ and $(e_j)_{j\in J}$ be two tripotents in $E$ and let $R$ be any of the relations $\perp,\le,\le_2,\le_0, \sim_2,\sim_0$. Then
    $$(u_j)_{j\in J}R(e_j)_{j\in J} \Leftrightarrow \forall j\in J: u_j R e_j.$$
\end{enumerate}
\end{prop}

 We finish this section by the following proposition on Lebesgue-Bochner $L^\infty$ spaces.
 
\begin{prop}\label{P:Linfty refelexive}
Let $E$ be a reflexive JB$^*$-triple and let $\mu$ be a probability measure. Let $M=L^\infty(\mu,E)$.
\begin{enumerate}[$(1)$]
    \item $M$ is a JBW$^*$-triple with the triple product defined pointwise ($\mu$-a.e.).
    \item $\uu\in M$ is a tripotent if and only if $\uu(\omega)$ is a tripotent in $E$ for $\mu$-a.a. $\omega$.
    \item Let $\uu,\vv\in M$ be two tripotents  and let $R$ be any of the relations $\perp,\le,$ $\le_2,$ $\le_0,$ $\sim_2,$ $\sim_0$. Then
    $$\uu R\vv \Leftrightarrow \uu(\omega) R \vv(\omega)\ \mu\mbox{-a.e.}$$
\end{enumerate}
\end{prop}

\begin{proof}
$(1)$ It is easy to observe that $M$ is indeed a JB$^*$-triple. Moreover, since $E$ is reflexive, $L^\infty(\mu,E)$ is the dual of $L^1(\mu,E^*)$ (by \cite[Theorem IV.1.1]{Diestel-Uhl}), hence $M$ is a JBW$^*$-triple.

Assertions $(2)$ and $(3)$ now easily follow.
\end{proof}

\section{Notion of finiteness in JBW$^*$-triples}\label{sec:3}

In this section we define and analyze the notion of finite tripotent. Let us start by recalling similar notions for projections (cf. \cite[Definition V.1.15]{Tak} or \cite[Definition 6.3.1]{KR2}; the first notion has  already been mentioned in the introduction).

Let $M$ be a von Neumann algebra. A projection $p\in M$ is 
\begin{itemize}
    \item {\em finite} if any projection $q\le p$ such that $p\sim q$ is equal to $p$;
    \item {\em infinite} if it is not finite;
    \item {\em properly infinite} if $p\ne 0$ and $zp$ is infinite whenever $z$ is a central projection such that $zp\ne0$.
\end{itemize}

In the following lemma we collect some known results on finite and properly infinite projections.

\begin{lemma}\label{L:finite projections basic facts}
Let $M$ be a von Neumann algebra.
\begin{enumerate}[$(a)$]
    \item If $p\in M$ is a finite projection, then any projection $q\le p$ is again finite.
    \item If $p\in M$ is an infinite projection, then there is a central projection $z\in M$ such that $zp$ is properly infinite and $(1-z)p$ is finite.
    \item If $p\in M$ is a properly infinite projection, then there is a projection $q\le p$ with $q\sim p-q\sim p$.
    \item If $p\in M$ is a properly infinite projection, then $p=\sum_n p_n$, where $(p_n)$ is a sequence of mutually orthogonal projections equivalent to $p$.
    \item If $(q_n)$ is an orthogonal sequence of mutually equivalent nonzero projections, then $\sum_n q_n$ is properly infinite.
\end{enumerate}
\end{lemma}

\begin{proof}
$(a)$ This follows for example from  \cite[Proposition 6.3.2]{KR2}.

$(b)$ This is proved in \cite[Proposition 6.3.7]{KR2}.

$(c)$ This follows from \cite[Proposition 6.3.3]{KR2} or from \cite[Proposition V.1.36]{Tak}.

$(d)$ This is proved in \cite[Proposition 2.2.4]{Sak71}.

$(e)$ This fact is part of the folklore of the theory of finite projections. Let us indicate the proof for the sake of completeness.

Set $p=\sum_n q_n$. Assume that $z\in M$ is a central projection with $zp\ne0$. Then $(zq_n)$ is an orthogonal sequence of mutually equivalent projections with sum $zp$.  By \cite[Proposition 4.2]{stra-zsi} we deduce that $zp-zq_1\sim zp$, hence $zp$ is infinite.  Thus $p$ is properly infinite.
\end{proof}

The notion of equivalence of projections cannot be generalized to JBW$^*$-triples as there is no notion of projection. It cannot be generalized to JBW$^*$-algebras either as there is no easy relationship between projections and tripotents. Note that the notions of final and initial projections strongly use the associative structure of the von Neumann algebra. However, the notion of finiteness can be naturally carried over. As we reminded in the introduction, a von Neumann algebra $W$ is finite if and only if all the extreme points of its closed unit ball are unitary (see \cite[Theorem 2]{ChodaKijimaNak69}). We can therefore conclude that a projection $p$ in $W$ is finite if and only if each extreme point of the closed unit ball of $pWp$ is a unitary. This is the starting point for our definition of finite tripotents. 

Let $M$ be a JBW$^*$-triple and $e\in M$ be a tripotent. We say that $e$ is 
\begin{itemize}
    \item {\em finite} if any tripotent $u\in M_2(e)$ which is complete in $M_2(e)$ is already unitary in $M_2(e)$;
    \item {\em infinite} if it is not finite;
    \item {\em properly infinite} if $e\ne 0$ and for each direct summand $I$ of $M$ the tripotent $P_I e$ is infinite whenever it is nonzero.
\end{itemize}
If any tripotent in $M$ is finite, we say that $M$ itself is {\em finite}.
Further, if $M$ is not finite, it is called {\em infinite}. Finally, $M$ is {\em properly infinite} if each nonzero direct summand is infinite.

The following lemma contains some basic facts on finite tripotents.

\begin{lemma}\label{L:finite tripotents} Let $M$ be a JBW$^*$ triple.
\begin{enumerate}[$(a)$]
    \item If $M$ is a von Neumann algebra, the two notion of finiteness coincide for projections;
    \item Let $e\in M$ be a finite tripotent and $u\in M_2(e)$ any tripotent. Then there is a tripotent $v\in M$ such that $u\le v$ and $v\sim_2 e$;
    \item If $e\in M$ is a finite tripotent, then any tripotent $u\in M_2(e)$ is also finite in M;
    \item Let $M$ be a finite JBW$^*$-algebra and $u\in M$ any tripotent. Then there is a unitary element $v\in M$ such that $u\le v$.
    \item Any abelian tripotent is finite.
   \end{enumerate}
\end{lemma}

\begin{proof}
$(a)$ As we commented in the introduction (see page \pageref{eq geometric characterization of finite projections}) this statement is a consequence of \cite[Theorem 2]{ChodaKijimaNak69} and the fact that the extreme points of the closed unit ball of a C$^*$-algebra are precisely the complete tripotents \cite{KadisonAnn51}.

$(b)$ If $u\sim_2 e$, we can take $v=u$. So, assume $u<_2 e$. Since $u$ is not unitary in $M_2(e)$, it is not complete, so $M_0(u)\cap M_2(e)$ is a nontrivial JBW$^*$-subtriple of $M$. Let $w$ be any complete tripotent in this subtriple. Then $v=u+w$ is a tripotent complete in $M_2(e)$, thus unitary in $M_2(e)$ (because $e$ is finite), i.e., $v\sim_2 e$.

$(c)$ Assume that $e$ is a finite tripotent and $u\in M_2(e)$ is a tripotent. Let $v$ be the tripotent provided by $(b)$. Since $M_2(v)=M_2(e)$, $v$ is finite. Hence, up to replacing $e$ by $v$, we may assume that $u\le e$.

Take a complete tripotent $r\in M_2(u)$. Since $u\perp e-u$ and $r\le_2 u$, by Proposition~\ref{P:M2 inclusion} we see that
$$e-u\in M_0(u)\subset M_0(r),$$
hence $e-u\perp r$, thus $e-u+r$ is a tripotent.
Moreover,
$$M_2(e)\cap M_0(e-u+r)\subset M_2(e)\cap M_0(e-u)\cap M_0(r)=M_2(u)\cap M_0(r)=\{0\}$$
as $r$ is complete in $M_2(u)$. It follows that $e-u+r$ is complete in $M_2(e)$.
Since $e$ is finite, $e-u+r$ is even unitary in $M_2(e)$. Hence for any $x\in M_2(u)$
we have
$$x=\J{e-u+r}{e-u+r}x=\J{e-u}{e-u}x+\J rrx=\J rrx.$$
Therefore $r$ is unitary in $M_2(u)$, which completes the proof.

$(d)$ This follows from $(b)$ applied to $e=1$.

$(e)$ If $e$ is an abelian tripotent in a JBW$^*$-triple $M$, then $M_2(e)$ is triple-isomorphic to an abelian von Neumann algebra,
hence $e$ is finite by $(a)$.
\end{proof}

We recall that JB$^*$-triples admitting a unitary tripotent are in one-to-one correspondence with unital JB$^*$-algebras (cf. \cite[Theorem 4.1.55]{Cabrera-Rodriguez-vol1}). Our next result shows how the preorders $\leq_2$ and $\leq_0$ can be also employed to characterize finite JBW$^*$-algebras among JBW$^*$-triples.

\begin{prop}\label{P:le0=le2}
Let $M$ be a JBW$^*$-triple. The following assertions are equivalent.
\begin{enumerate}[$(i)$]
    \item The preorders $\le_2$ and $\le_0$ coincide for tripotents in $M$;
    \item Any complete tripotent in $M$ is unitary;
    \item $M$ is (triple-isomorphic to) a finite JBW$^*$-algebra.
\end{enumerate}
\end{prop}
\begin{proof}
$(i)\Rightarrow(ii)$ If $(ii)$ is not satisfied, then $M$ admits a complete non-unitary tripotent $e$. Fix a tripotent $v\in M_1(e)$ and let $u$ be a complete tripotent in $M$ with $u\ge v$ (cf. \cite[Lemma 3.12$(1)$]{horn1987ideal}). Then $u\sim_0 e$ (as both tripotents are complete) but $u\not\le_2 e$ (as $v\in M_2(u)\setminus M_2(e)$).

$(ii)\Rightarrow(iii)$ It follows from $(ii)$ that there is a unitary tripotent $e\in M$. Hence $M$ is triple-isomorphic to the JBW$^*$-algebra $M_2(e)$. Since $e$ is finite by definition, Lemma~\ref{L:finite tripotents}$(c)$ shows that any tripotent in $M$ is finite. This completes the argument.

$(iii)\Rightarrow (i)$ Assume that $u,e\in M$ are tripotents. Fix unitary tripotents $\tilde u\ge u$ and $\tilde e\ge e$. They exist due to Lemma~\ref{L:finite tripotents}$(d)$.

Observe that $\tilde e-e$ is a tripotent with 
$$M_2(\tilde e-e)=M_0(e)\mbox{ and }M_0(\tilde e-e)=M_2(e).$$
Indeed, since $\tilde e$ is unitary, we can assume without loss of generality that $\tilde e=1$ and $e$ is a projection (up to replacing operations $\circ$ and $*$ by $\circ_e$ and $*_e$).
In this case $\J eex=e\circ x$ for $x\in M$. Thus
$$ x\in M_2(1-e)\Leftrightarrow x=(1-e)\circ x \Leftrightarrow e\circ x=0 \Leftrightarrow x\in M_0(e)$$
and similarly for the second equality.

Analogous equalities hold for $u$ and $\tilde u-u$. Thus
\begin{multline*}
u\le_0 e\Leftrightarrow M_0(e)\subset M_0(u) \Leftrightarrow 
M_2(\tilde e-e)\subset M_2(\tilde u-u) \Leftrightarrow \tilde e-e\le_2\tilde u-u\\ \Rightarrow M_0(\tilde u-u)\subset M_0(\tilde e-e)
\Leftrightarrow M_2(u)\subset M_2(e)\Leftrightarrow u\le_2 e,
\end{multline*}
where we used definitions and Proposition~\ref{P:M2 inclusion}. It follows that $u\le_0 e\Rightarrow u\le_2 e$.
\end{proof}

\begin{prop}\label{P:finitedim is finite}
Any finite-dimensional JBW$^*$-triple is finite.
\end{prop}

\begin{proof}
Let $M$ be a finite-dimensional JBW$^*$-triple. If $M$ is not finite, then there is a tripotent $e\in M$ which is not finite.
Then $M_2(e)$ is not finite either. Hence, without loss of generality assume that $M$ is a finite-dimensional JBW$^*$-algebra.
If it is not finite, there is a complete non-unitary tripotent $u\in M$. Let $N$ denote the closed unital Jordan-$*$ subalgebra of $M$ generated by $u$ and the unit. By \cite[Lemma 6.3]{hamhalter2019mwnc} there is a complex Hilbert space $H$ and a unital Jordan-$*$ monomorphism $\psi:N\to B(H)$ with $\psi(u)^*\psi(u)=1$. 
Let $A$ denote the $C^*$-algebra generated by $\psi(N)$ (equivalently, generated by $\psi(u)$). Let us show that $A$ is finite-dimensional. Since $\psi(u)^*\psi(u)=1$, 
$$A=\overline{\span}\{\psi(u)^k(\psi(u)^*)^l\setsep k,l\in\en\cup\{0\}\}=\overline{\span}\{\psi(x)\psi(y)\setsep x,y\in N\}.$$
The second equality follows from the fact that $\psi(u)^k=\psi(u^k)\in\psi(N)$ and also $(\psi(u)^*)^l=\psi((u^*)^l)\in \psi(N)$. Since $N$ has finite dimension, it easily follows that $A$
has finite dimension as well. We deduce that $\psi(u)\psi(u)^*=1$ as well, hence $u\circ u^*=1$, i.e., $\J uu1=1$, hence $1\le_2 u$. It follows that $1\sim_2 u$, so $u$ is unitary.
This contradiction completes the proof.
\end{proof}

\begin{prop}\label{P:Linftyfindim is finite}
Let $E$ be a finite-dimensional JBW$^*$ triple and let $\mu$ be a probability measure. Then $L^\infty(\mu,E)$ is a finite JBW$^*$-triple. 
\end{prop}

\begin{proof}
$M=L^\infty(\mu,E)$ is a JBW$^*$-triple by Proposition~\ref{P:Linfty refelexive}.
Let us prove that $L^\infty(\mu,E)$ is finite. We will use the characterization of the preorders from Proposition~\ref{P:Linfty refelexive}.

Assume that $\uu,\vv\in L^\infty(\mu,E)$ are two tripotents such that $\uu$ is a complete tripotent in $M_2(\vv)$. In particular, $\uu\in M_2(\vv)$, hence $\uu(\omega)\in E_2(\vv(\omega))$ $\mu$-a.e. The completeness means that $M_2(\vv)\cap M_0(\uu)=0$, hence  $P_0(\uu)P_2(\vv)=0$. It follows that
$$P_0(\uu(\omega))P_2(\vv(\omega))=0 \qquad\mu\mbox{-a.e.},$$
hence $\uu(\omega)$ is a complete tripotent in $E_2(\vv(\omega))$ $\mu$-a.e., hence by Proposition~\ref{P:finitedim is finite}  $\uu(\omega)$ is unitary in $E_2(\vv(\omega))$ $\mu$-a.e. Thus $\uu$ is unitary in $M_2(\vv)$, which completes the proof.
\end{proof}

\begin{prop}\label{P:finiteness in direct sum}
Let $(M_j)_{j\in J}$ be a family of JBW$^*$-triples. Set
$\displaystyle M=\bigoplus_{j\in J}^{\ell_\infty} M_j.$ 
Then a tripotent $(e_j)_{j\in J}$ in $M$ is finite if and only if $e_j$ is a finite tripotent in $M_j$ for each $j\in J$.
\end{prop}

\begin{proof}
The `only if part' follows from Lemma~\ref{L:finite tripotents}$(c)$.

To show the `if part' fix a tripotent $(u_j)_{j\in J}$ complete in $M_2((e_j)_{j\in J})$. Then $u_j$ is a complete tripotent in $M_2(e_j)$ for each $j$.
It follows that $u_j\sim_2 e_j$ for each $j$, this $(u_j)_{j\in J}\sim_2 (e_j)_{j\in J}$ (cf. Proposition~\ref{P:preorders in direct sum}).
\end{proof}

\section{Von Neumann algebras and their right ideals}\label{sec:vN and pV}

In this section we will address triples of the form
\begin{equation*}
 \left(\bigoplus_{j\in J}^{\ell_\infty}A_j\overline{\otimes}C_j\right)\oplus^{\ell_\infty}pV,    
\end{equation*}
where the $A_j$'s are abelian von Neumann algebras, the $C_j$'s are Cartan factors of type $1$, $V$ is a continuous von Neumann algebra and $p\in V$ is a projection.

Recall that a {\em Cartan factor of type 1} is the space $B(H,K)$ where $H$ and $K$ are two complex Hilbert spaces (possibly with different dimensions) equipped with the triple product  given by \eqref{eq JC triple product}.

Using this definition we can easily come to the following observation explained for example after Example 9.1 in \cite{hamhalter2019mwnc}.

\begin{obs}
Let $M$ be a JBW$^*$-triple of the above form
 Then $M$ is triple-isomorphic to a JBW$^*$-triple of the form $qW$, where $W$ is a von Neumann algebra and $q\in W$ is a projection.
\end{obs}

Let $V$ be a von Neumann algebra, $p\in V$ a projection and $M=pV$. We recall that tripotents in $V$ are exactly partial isometries. If $u\in V$ is a partial isometry, then $p_i(u)=u^*u$ denotes the initial projection and $p_f(u)=uu^*$ is the final projection. In this case the Peirce projections are given by the formulas
$$\begin{gathered}P_2(u)x=p_f(u)xp_i(u),\quad P_1(u)x=(1-p_f(u))xp_i(u)+p_f(u)x(1-p_i(u)),\\ P_0(u)x=(1-p_f(u))x(1-p_i(u))\end{gathered}$$
for $x\in V$. Tripotents in $M$ are partial isometries in $V$ which belong to $M$, i.e., those partial isometries $u\in V$ such that $p_f(u)\le p$. In this case the Peirce projections (on $M$) are given 
by the formulas
$$\begin{gathered}P_2(u)x=p_f(u)xp_i(u),\quad P_1(u)x=(p-p_f(u))xp_i(u)+p_f(u)x(1-p_i(u)),\\ P_0(u)x=(p-p_f(u))x(1-p_i(u))\end{gathered}$$
for $x\in M$.

The next proposition collects few basic facts on tripotents in triples of the form $pV$. We shall also need the following standard notation. 

If $A$ is an algebra (Banach algebra, C$^*$-algebra), we shall write $A^{op}$ for the {\em opposite algebra} obtained from the same elements and linear structure (norm and involution), but with the multiplication performed in the reverse order.

\begin{prop}\label{P:Peirce2 in pV}
Let $M=pV$ where $V$ is a von Neumann algebra and $p\in V$ is a projection. Then the following assertions are true.
\begin{enumerate}[$(a)$]
    \item $M_2(e)=V_2(e)$ for any tripotent $e\in M$.
    \item $M_2(e)$ is triple-isomorphic to a von Neumann algebra for any tripotent $e\in M$.
    \item Any direct summand of $M$ which is triple-isomorphic to a JBW$^*$-algebra is triple-isomorphic to a von Neumann algebra.
    \item $M$ is triple-isomorphic to a JBW$^*$-triple of the form $W_1\oplus^{\ell_\infty} qW_2$, where $W_1$ and $W_2$ are von Neumann algebras, $q\in W_2$ is a projection and $qW_2$ has no nonzero direct summand triple-isomorphic to a JBW$^*$-algebra.
\end{enumerate}
\end{prop}

\begin{proof}
$(a)$ It is clear that $M_2(e)=p_f(e)Mp_i(e)=p_f(e)Vp_i(e)=V_2(e)$.

$(b)$ The mapping $x\mapsto xe$ is a linear isometry of the von Neumann algebra $p_f(e)Vp_f(e)$ onto $p_f(e)Vp_i(e)=M_2(e)$. Moreover this mapping is also a triple isomorphism (this can be checked easily and it also follows from Kaup's Banach-Stone theorem \cite[Proposition 5.5]{kaup1983riemann}).

$(c)$ This follows from $(b)$ and Proposition~\ref{P:direct summand}$(b)$.

$(d)$ If there is no tripotent $e\in M$ with $M_1(e)=0$, then we may take  $W_1=0$, $W_2=V$ and $q=p$. The relevant properties then follow from Proposition~\ref{P:direct summand}$(b)$.
 
Otherwise, let $(e_\gamma)_{\gamma\in \Gamma}$ be a maximal family of nonzero mutually orthogonal tripotents such that $M_1(e_\gamma)=0$ for each $\gamma\in\Gamma$. This family exists by Zorn's lemma.

Set $e=\sum_{\gamma\in\Gamma} e_\gamma$. Then $e\in M$ is a tripotent and, moreover, $M_1(e)=0$. (Note that $M_1(u)=0$ if and only if $L(u,u)^2=L(u,u)$ and this property easily passes to sums of orthogonal families of tripotents.) Thus $W_1=M_2(e)$ is a direct summand isomorphic to a JBW$^*$-algebra by Proposition~\ref{P:direct summand}$(b)$, hence to a von Neumann algebra by $(c)$.

The orthogonal direct summand is then 
$$M_0(e)=(p-p_f(e))V(1-p_i(e)).$$
By the construction and Proposition~\ref{P:direct summand}$(b)$ applied to $M$ it admits no nonzero direct summand isomorphic to a JBW$^*$-algebra. It remains to show that it can be represented as $qW_2$. 
By the comparability theorem \cite[Theorem V.1.8]{Tak} there is a central projection $z\in V$ such that
$$z(p-p_f(e))\precsim z(1-p_i(e))\mbox{ and }(1-z)(1-p_i(e))\precsim (1-z)(p-p_f(e)).$$
It means that there are two partial isometries $u,v\in V$ such that
$$p_i(u)=z(p-p_f(e)),\ p_f(u)\le z(1-p_i(e)),$$ $$p_i(v)=(1-z)(1-p_i(e)),\ p_f(v)\le (1-z)(p-p_f(e)).$$
Then 
$zM_0(e)=(p-p_f(e))zV(1-p_i(e))$ is isometric (hence triple-isomorphic) to
$$p_f(u) (1-p_i(e))zV(1-p_i(e))$$
via the mapping
$x\mapsto ux$.
Similarly, $(1-z)M_0(e)=(p-p_f(e))(1-z)V(1-p_i(e))$ is isometric (hence triple-isomorphic) to
$$(p-p_f(e))(1-z)V(p-p_f(e))p_f(v)$$
via the mapping
$x\mapsto xv^*$.

So, we can take $$W_2=(1-p_i(e))zV(1-p_i(e))\oplus^{\ell_\infty} ((p-p_f(e))(1-z)V(p-p_f(e)))^{op}$$ and $q=(p_f(u),p_f(v))$.
\end{proof}

Assertion $(d)$ of the previous proposition provides a canonical decomposition of the triples of the form $pV$ into two summands. We will analyze these two summands separately in the following two subsections.

\subsection{Von Neumann algebras}\label{subsec:vonNeumann algebras}

In this subsection we will focus on the preorders and finiteness in von Neumann algebras. Throughout this subsection let $M$ denote a fixed von Neumann algebra.

We start by a description of unitary and complete tripotents.

\begin{lemma}\label{L:vN complete}
Let $u\in M$ be a tripotent. 
\begin{enumerate}[$(a)$]
    \item $u$ is unitary if and only if $p_i(u)=p_f(u)=1$.
    \item $u$ is complete if and only if there is a central projection $z\in M$ such that $p_i(zu)=z$ and $p_f((1-z)u)=1-z$.
\end{enumerate}
\end{lemma}

\begin{proof}
Assertion $(a)$ is just a reformulation of the definition of a unitary element. Let us prove $(b)$. Note that $u$ is complete if and only if $M_0(u)=0$, i.e., $(1-p_f(u))M(1-p_i(u))=0$. By \cite[Lemma V.1.7]{Tak} this takes place if and only if $1-p_f(u)$ and $1-p_i(u)$ are centrally orthogonal. This is equivalent to the existence of a central projection $z\in M$ such that
$$1-p_f(u)\le z\mbox{ and }1-p_i(u)\le 1-z.$$
This is in turn equivalent to
$$z\le p_i(u)\mbox{ and }1-z\le p_f(u),$$
which is equivalent to the condition given in the statement.
\end{proof}

We will need several times the following easy lemma.

\begin{lemma}\label{L:apq}
Let $A$ be a C$^*$-algebra, $a\in A$ an arbitary element and $p,q\in A$ two projections. If $a=apq$, then $a=ap$.
\end{lemma}

\begin{proof} There is no loss of generality in assuming that $A$ is unital. In this case We have
$$aa^* = a p q p a^*\leq a p a^* \leq a a^*.$$
It follows that $aa^* = a p a^*$, hence  $a(1-p) a^*=0$ and so $a(1-p)=0$ by the Gelfand-Naimark axiom.
\end{proof}

In the next lemma we provide a characterization of the preorders $\le_2$ and $\le_0$.

\begin{lemma}\label{l charcteriz le2 and le0 in vN}
Let $u,v\in M$ be two tripotents (i.e., partial isometries). 
\begin{enumerate}[$(a)$]
    \item $u\le_2 v \Longleftrightarrow p_i(u)\le p_i(v)\ \&\ p_f(u)\le p_f(v)$.
    \item $u\le_0 v$ if and only if there is a central projection $z\in M$ such that $zu\le_2 zv$ and $(1-z)v$ is a complete tripotent in $(1-z)M$. 
    \item Assume that $M$ is a factor. Then $u\le_0 v$ if and only if $v$ is a complete tripotent or $u\le_2 v$.
\end{enumerate}
\end{lemma}   

\begin{proof}
$(a)$ The implication `$\Leftarrow$' is obvious. To prove the implication `$\Rightarrow$' assume $u\le_2 v$. Then $u=p_f(v)up_i(v)$. So,
$$p_i(u)=u^*u=u^*up_i(v)=p_i(u)p_i(v),$$
thus $p_i(u)\le p_i(v)$. Similarly we get $p_f(u)\le p_f(v)$.

$(b)$ The `if part' is obvious (cf. Proposition~\ref{P:impl le20}).

Let us prove the `only if part'. Assume that $u\le_0 v$. We imitate the proof of \cite[Lemma V.1.7]{Tak}. Let
$$J_1=\{x\in M\setsep (1-p_f(v))Mx=\{0\}\}.$$
Then $J_1$ is a weak$^*$-closed ideal in $M$, so there is a central projection $z_1\in M$ such that $J_1=M(1-z_1)$ (see \cite[Proposition II.3.12]{Tak}). Then 
$(1-p_f(v))M(1-z_1)(1-p_i(v))=\{0\}$, so $(1-z_1)v$ is a complete tripotent in $M(1-z_1)$.

Clearly $z_1u\le_0 z_1v$ in $z_1M$. We continue working in $z_1M$, so to simplify notation we assume that $z_1=1$. Set
$$J_2=\{x\in M\setsep xM(1-p_i(v))=\{0\}\}.$$
Then $J_2$ is a weak$^*$-closed ideal in $M$, so there is a central projection $z_2\in M$ such that $J_2=M(1-z_2)$ (see \cite[Proposition II.3.12]{Tak}). Then
$(1-p_f(v))M(1-z_2)(1-p_i(v))=\{0\}$, so $(1-z_2)v$ is a complete tripotent in $M(1-z_2)$.

Clearly $z_2u\le_0 z_2v$ in $z_2M$. We continue working in $z_2M$, so to simplify notation we assume that $z_2=1$. 

Hence, the two reductions yield the following property: For any $x\in M\setminus\{0\}$ we have
$$(1-p_f(v))Mx\ne \{0\} \mbox{ and }xM(1-p_i(v))\ne \{0\}.$$
We claim that $u\le_2 v$. Assume not. Then there is some $x\in M_2(u)\setminus M_2(v)$. The assumption $x\notin M_2(v)$ means that $x(1-p_i(v))\ne0$ or $(1-p_f(v))x\ne0$.

Assume the first possibility takes place. Then, by the reductions above, there is $y\in M$ such that
$(1-p_f(v))yx(1-p_i(v))\ne0$. This element belongs to $M_0(v)\subset M_0(u)$, 
so
$$\begin{aligned}0&=(1-p_f(v))yx(1-p_i(v))p_i(u)=(1-p_f(v))yx p_i(u)-(1-p_f(v))yx p_i(v)p_i(u)\\&=(1-p_f(v))yx -(1-p_f(v))yx p_i(v)p_i(u),\end{aligned}$$
where we used the assumption $x\in M_2(u)$. Hence
$$(1-p_f(v))yx =(1-p_f(v))yx p_i(v)p_i(u),$$
so 
$$(1-p_f(v))yx =(1-p_f(v))yx p_i(v)$$
by Lemma~\ref{L:apq}.
But this equality means that
$(1-p_f(v))yx(1-p_i(v))=0$, a contradiction.

The second case is symmetric.

$(c)$ This follows immediately from $(b)$.
\end{proof}

The next result characterizing the partial order $\le$ is part of the folklore in JB$^*$-triples, it is included here for the lacking of an explicit reference.

\begin{prop}\label{P:vN order charact}
Let $u,e\in M$ be two tripotents. The following assertions are equivalent.
\begin{enumerate}[$(i)$]
    \item $u\le e$;
    \item There is a projection $p\le p_f(e)$ such that $u=pe$;
    \item There is a projection $q\le p_i(e)$ such that $u=eq$;
    \item There are projections $p\le p_f(e)$ and $q\le p_i(e)$ such that $u=pe=eq$.
\end{enumerate}
\end{prop}

\begin{proof}
The implication $(iv)\Rightarrow (ii)\&(iii)$ is obvious.

$(ii)\Rightarrow (i)$ Assume $u=pe$ where $p\le p_f(e)$. Then $u$ is a partial isometry with final projection $p$. Moreover,
$$\J ueu= ue^*u=pee^*pe=p p_f(e) pe=pe=u,$$
hence $u\le e$ by Proposition~\ref{P:order char}.

$(iii)\Rightarrow(i)$ This is completely analogous to $(ii)\Rightarrow(i)$.

$(i)\Rightarrow(iv)$ Assume $u\le e$. Let $p=uu^*$ be the final projection and $q=u^*u$ the initial projection of $u$. Since $u\le_2 e$, we have $p\le p_f(e)$ and $q\le p_i(e)$. Moreover, by Proposition~\ref{P:order char} we have
$$u=\J uue =\frac12(uu^*e+eu^*u)=\frac12(pe+eq).$$
It follows that
$$ u=pu=\frac 12 (pe+peq)\mbox{ and }
u=uq=\frac 12(peq+eq),$$
and hence $pe=eq$. We conclude that $u=pe=eq$.
\end{proof}

\begin{prop}\label{P:vN finite trip}
Let $u\in M$ be a tripotent. The following assertions are equivalent:
\begin{enumerate}[$(i)$]
    \item $u$ is finite.
    \item $p_f(u)$ is finite.
    \item $p_i(u)$ is finite.
\end{enumerate}
\end{prop}

\begin{proof}
The equivalences follow from the easy fact that the triples $M_2(u)$, $M_2(p_f(u))$ and $M_2(p_i(u))$ are mutually triple-isomorphic (cf. the proof of Proposition~\ref{P:Peirce2 in pV}$(b)$). 

Let us observe that the term ``finite'' in statements $(ii)$ and $(iii)$ is not ambiguous (cf. Lemma \ref{L:finite tripotents}).
\end{proof}

We continue by the following lemma saying that direct summands of von Neumann algebras are only the obvious ones.

\begin{lemma}\label{L:vN direct summand}
Any direct summand of $M$ is of the form $zM$ for a central projection $z\in M$.
\end{lemma}

\begin{proof}
Let $I$ be a direct summand of $M$. Then it is a weak$^*$-closed (triple-)ideal in $M$. Since $M$ is a C$^*$-algebra, by \cite[Proposition 5.8(c)]{harris-general} $I$ is a two-sided ideal. Thus, by \cite[Proposition II.3.12]{Tak} $I$ has the required form.
\end{proof}

The following statement is an immediate consequence of the preceding lemma.

\begin{cor}\label{cor:properly infinite projection}
\begin{enumerate}[$(i)$]
    \item Let $u\in M$ be a tripotent. Then $u$ is properly infinite if and only if $zu$ is infinite whenever $z$ is a central projection such that $zu\ne0$.
    \item Let $p\in M$ be a projection. Then $p$ is properly infinite as a projection if and only if it is properly infinite as a tripotent.
\end{enumerate}
\end{cor}

The following proposition provides a characterization of properly infinite tripotents in $M$.

\begin{prop}\label{P:vN properly infinite trip}
Let $u\in M$ be a tripotent. The following assertions are equivalent.
\begin{enumerate}[$(i)$]
    \item $u$ is properly infinite;
    \item $p_i(u)$ is properly infinite;
    \item $p_f(u)$ is properly infinite;
    \item Every complete tripotent in $M_2(u)$ is properly infinite;
    \item Some complete tripotent in $M_2(u)$ is properly infinite.
\end{enumerate}
\end{prop}

\begin{proof}
$(i)\Leftrightarrow(ii)\Leftrightarrow(iii)$
If $z\in M$ is a central projection, then $zu$ is a tripotent, $p_i(zu)=zp_i(u)$ and $p_f(zu)=zp_f(u)$.
Hence the equivalences follow from 
Corollary~\ref{cor:properly infinite projection} and Proposition~\ref{P:vN finite trip}.

$(i)\Rightarrow(iv)$ Assume $u$ is properly infinite and $v\in M_2(u)$ is a complete tripotent. Fix a central projection $z\in M$ such that $zv\ne 0$. Then clearly $zv\in M_2(zu)$, hence $zu\ne0$. Moreover, $zv$ is a complete tripotent in $M_2(zu)$. 

Assume that $zv$ is finite. By Proposition~\ref{P:vN finite trip} we see that the projections $p_i(zv)=zp_i(v)$ and $p_f(zv)=zp_f(v)$ are finite. Further, $z(p_f(u)-p_f(v))Mz(p_i(u)-p_i(v))=0$, hence the projections $z(p_f(u)-p_f(v))$ and  $z(p_i(u)-p_i(v))$ are centrally orthogonal by \cite[Lemma V.1.7]{Tak}. It follows that there is a central projection $y\le z$ such that
$$z(p_f(u)-p_f(v))\le y\mbox{ and  }z(p_i(u)-p_i(v))\le z-y.$$
The second inequality implies $p_i(u)y=p_i(v)y$, hence $p_i(u)y$ is finite. Similarly, the first inequality yields that $(z-y)p_f(u)$ is finite. Since $(z-y)p_f(u)\sim (z-y)p_i(u)$, we infer that $(z-y)p_i(u)$
is finite as well. Thus $zp_i(u)$ is finite, hence $zu$ is finite (by Proposition~\ref{P:vN finite trip}). This contradicts the assumption that $u$ is properly infinite.

$(iv)\Rightarrow(v)$ This is obvious.

$(v)\Rightarrow(i)$ Assume that $v\in M_2(u)$ is a complete properly infinite tripotent. Let $z\in M$ be a central projection such that $zu\ne0$. Then $zv$ is a complete tripotent in $M_2(zu)$, thus $zv\ne0$. By the assumption we know that $zv$ is infinite. If $zv\sim_2 zu$, then $zu$ is infinite as well. If $zv<_2 zu$, then $zu$ is infinite as $zv$ is a complete non-unitary tripotent in $M_2(zu)$. 

In any case $zu$ is infinite, which completes the proof.
\end{proof}

\begin{prop}\label{P:vN halving}
Let $u\in M$ be a properly infinite tripotent. Then there is a sequence $(v_n)$ of complete tripotents in $M_2(u)$ such that $v_n\in M_1(v_m)$ whenever $m\ne n$.
\end{prop}

\begin{proof}
By Proposition~\ref{P:vN properly infinite trip} we know that the projection $p_f(u)$ is properly infinite. By 
Lemma~\ref{L:finite projections basic facts}$(d)$ there is an orthogonal sequence of projections $(q_n)$ such that $p_f(u)=\sum_n q_n$ 
 and $q_m\sim p_f(u)$ for each $m$. Since $p_f(u)\sim p_i(u)$, we deduce that $q_n\sim p_i(u)$ for each $n\in\en$. It is enough to take partial isometries $v_n\in M$ such that $p_i(v_n)=p_i(u)$ and $p_f(v_n)=q_n$.
\end{proof}

\begin{remark}\label{rem:vN subtriples}
Some results of this section remain valid for JBW$^*$-subtriples of von Neumann algebras. This is the case for Lemma~\ref{l charcteriz le2 and le0 in vN}$(a)$ and Proposition~\ref{P:vN order charact}. These characterizations are not completely inner as they use elements and operations from the  surrounding von Neumann algebra. 

On the other hand, for example Lemma~\ref{l charcteriz le2 and le0 in vN}$(b)$ need not be valid in this case.
\end{remark}

\subsection{Triples of the form $pV$}

Throughout this subsection $V$ will be a fixed von Neumann algebra, $p\in V$ a fixed projection and $M=pV$. We will also assume that $M$ has no nonzero direct summand triple-isomorphic to a JBW$^*$-algebra, i.e., that $M_1(u)\ne0$ for any nonzero tripotent $u\in M$ (cf. Proposition~\ref{P:direct summand}$(b)$).

\begin{remark}\label{rem:pV reduced to V}
The characterizations of $\le_2$ from Lemma~\ref{l charcteriz le2 and le0 in vN}$(a)$ and that of $\le$ from Proposition~\ref{P:vN order charact} remain valid for $M$.
This follows from Remark~\ref{rem:vN subtriples}.

Further, also the characterization of finiteness in Proposition~\ref{P:vN finite trip} remains valid for $M$.
This follows from Proposition~\ref{P:Peirce2 in pV}$(a)$.
\end{remark}

Lemma~\ref{L:vN complete} is in this case replaced by the following one.

\begin{lemma}\label{L:pV complete}
\begin{enumerate}[$(a)$]
    \item $M$ admits no unitary tripotent.
    \item A tripotent $e\in M$ is complete if and only if $p_f(e)=p$.
\end{enumerate}
\end{lemma}

\begin{proof}
Assertion $(a)$ follows from the assumptions.

$(b)$ The `if part' is obvious. Let us prove the `only if part'. 

Assume that $e$ is complete, i.e., $M_0(e)=\{0\}$. Note that $p_f(e)\le p$ and $P_0(e)x=(p-p_f(e))x(1-p_i(e))$ for $x\in M$. It follows that
$$\{0\}=(p-p_f(e))M(1-p_i(e))=(p-p_f(e))V(1-p_i(e)),$$
i.e., $p-p_f(e)$ and $1-p_i(e)$ are centrally orthogonal (cf. \cite[Lemma V.1.7]{Tak}). Therefore there is a central projection $z\in V$ with $p-p_f(e)\le z$ and $1-p_i(e)\le 1-z$. The latter inequality implies $z\le p_i(e)$, i.e., $zp_i(e)=z$. We deduce that
$$z=zp_i(e)\sim zp_f(e)\le zp\le z,$$
hence $z\sim zp$. Fix a partial isometry $v\in V$ with $p_i(v)=z$ and $p_f(v)=zp$. Then $v\in M$. Moreover,
$$zpM(1-z)=\{0\}\mbox{ and }(p-zp)Mz=\{0\},$$
hence $M_1(v)=\{0\}$. It follows that $M_2(v)$ is a direct summand of $M$, hence $v=0$ by the assumption. We deduce that $z=0$, hence $p_f(e)=p$.
\end{proof}

In the following proposition we characterize the preorder $\le_0$ in our setting (substituting thus Lemma~\ref{l charcteriz le2 and le0 in vN}$(b)$).

\begin{prop}\label{P:le0 in pV}
Let $u,v\in M$ be two tripotents. Then $u\le_0v$ if and only if there is a central projection $z\in V$ with
$zp_f(v)=zp$ and $(1-z)u\le_2(1-z)v$.
\end{prop}

\begin{proof}
Assume that $u\le_0 v$. We modify the proof of Lemma~\ref{l charcteriz le2 and le0 in vN}$(b)$. Let
$$J_1=\{x\in V\setsep (p-p_f(v))Vx=\{0\}\}.$$
Then $J_1$ is a weak$^*$-closed ideal in $V$, so there is a central projection $z_1\in V$ such that $J_1=V(1-z_1)$ (see \cite[Proposition II.3.12 ]{Tak}). Then   
$(p-p_f(v))(1-z_1)=0$, so $(p-z_1)v$ is a complete tripotent in $pV(1-z_1)=M(1-z_1)$.

Clearly $z_1 u\le_0 z_1v$ in $z_1M=pz_1V$. We continue working in $z_1V$, so to simplify notation we assume that $z_1=1$.
This reduction yields the following property: For any $x\in V\setminus\{0\}$ we have
$(p-p_f(v))Vx\ne\{0\}$.

Set
$$J_2=\{x\in V\setsep xV(1-p_i(v))=\{0\}\}.$$
Then $J_2$ is a weak$^*$-closed ideal in $V$, so there is a central projection $z_2\in V$ such that $J_2=V(1-z_2)$ (see \cite[Proposition  II.3.12]{Tak}). Then
$(1-z_2)(1-p_i(v))=0$, so $(1-z_2)v=(1-z_2)pv$ is a complete tripotent in $(1-z_2)V$, hence in $(1-z_2)pV=(1-z_2)M$.

Clearly $z_2u\le_0 z_2v$ in $z_2M=pz_2V$. We continue working in $z_2M$, so to simplify notation we assume that $z_2=1$.    

Hence, the two reductions yield the following property: For any $x\in V\setminus\{0\}$ we have
$$(p-p_f(v))Vx\ne\{0\} \mbox{ and }xV(1-p_i(v))\ne \{0\}.$$

We claim that $u\le_2 v$. Assume not. Then there is some $x\in M_2(u)\setminus M_2(v)$. The assumption $x\notin M_2(v)$ means that $x(1-p_i(v))\ne0$ or $(p-p_f(v))x\ne0$.

Assume the first possibility takes place. Then there is $y\in V$ such that
$(p-p_f(v))yx(1-p_i(v))\ne0$. This element belongs to $M_0(v)\subset M_0(u)$, 
so
$$\begin{aligned}0&=(p-p_f(v))yx(1-p_i(v))p_i(u)=(p-p_f(v))yx p_i(u)-(p-p_f(v))yx p_i(v)p_i(u)\\&=(p-p_f(v))yx -(p-p_f(v))yx p_i(v)p_i(u),\end{aligned}$$
where we used the assumption $x\in M_2(u)$. Hence
$$(p-p_f(v))yx =(p-p_f(v))yx p_i(v)p_i(u),$$
so 
$$(p-p_f(v))yx =(p-p_f(v))yx p_i(v)$$
by Lemma~\ref{L:apq}.
But this equality means that
$(p-p_f(v))yx(1-p_i(v))=0$, a contradiction.

Assume that the second possibility takes place, i.e., 
$(p-p_f(v))x\ne0$. Then there is $y\in V$ with 
$(p-p_f(v))xy(1-p_i(v))\ne0$. This element belongs to $M_0(v)\subset M_0(u)$, hence

$$\begin{aligned}0&=p_f(u)(p-p_f(v))xy(1-p_i(v))=p_f(u)xy(1- p_i(v))-p_f(u)p_f(v)xy (1-p_i(v))\\&=xy(1-p_i(v)) -p_f(u)p_f(v)xy (1-p_i(v)).\end{aligned}$$ 
It follows that
$$xy(1-p_i(v)) =p_f(u)p_f(v)xy (1-p_i(v)),$$
and so
$$xy(1-p_i(v)) =p_f(v)xy (1-p_i(v))$$
by Lemma~\ref{L:apq} (applied in $V^{op}$).
Since $x\in M$, hence $x=px$, we deduce $$(p-p_f(v))xy(1-p_i(v))=0,$$ which leads to a contradiction.

So, to finalize the proof, observe that we have shown that there is a central projection $z\in V$ such that $zv$ is complete in $pzV$ and $(1-z)u\le_2(1-z)v$. An application of Lemma~\ref{L:pV complete}$(b)$ yields $zp_f(v)=p_f(zv)=zp$.
\end{proof}

\begin{prop}\label{P:pV direct summand}
Any direct summand of $M$ is of the form $pzV$ where $z\in V$ is a central projection.
\end{prop}

\begin{proof}
Let $I$ be a direct summand of $M$. Fix a tripotent $u\in I$ complete in $I$. Then $I=M_2(u)+M_1(u)$ and $M_0(u)$ is the orthogonal direct summand. Moreover, the properties from Proposition~\ref{P:direct summand}$(a)$ are satisfied. 

Note that $u$ is a partial isometry in $V$, set $r=p_i(u)$ and $q=p_f(u)$. Clearly $q\le p$. Since
$$q=uu^*=uru^*\mbox{ and }r=u^*u=u^*qu,$$
we deduce that
$$VrV=VqV=VrVqV=VqVrV.$$
Moreover, 
$J=\wscl{\span} VrV$ is a weak$^*$-closed two-sided ideal in $V$, hence $J=zV$ for some central projection $z\in V$ (see \cite[Proposition II.3.12]{Tak}). Since $r,q\in J$, we deduce that
$r\vee q\le z$.

Note that   
$$M_2(u)=qVr\subset VrV\subset J$$
and
$$M_1(u)\subset qV(1-r)+(p-q)Vr\subset VqV+VrV\subset J,$$
hence
$I\subset J=zV$, so
$$I\subset zV\cap M=zV\cap pV=zpV.$$

Next we are going to show that $M_0(u)\cap J=\{0\}$. To this end fix $x\in M_0(u)\cap J$. Then $x=(p-q)x(1-r)$ and $x=zx$. Thus
$x=(p-q)zx(1-r)$. On the other hand,
by Proposition~\ref{P:direct summand}$(a)$ we know that $\J{M_1(u)}{M_2(u)}{M_1(u)}=\{0\}$, in particular
$(p-q)VrVqV(1-r)=\{0\},$
so by the above we get $(p-q)J(1-r)=\{0\}$, hence $(p-q)zV(1-r)=\{0\}$.
It follows that $x=0$.

Now it is clear that $I=zpV$.
\end{proof}

As an easy consequence of the previous results we obtain the following corollary.

\begin{cor}\label{cor:pV properly infinite trip}
Let $u\in M$ be a tripotent. Then $u$ is properly infinite in $M$ if and only if it is properly infinite in $V$.
\end{cor}

Hence, we get also the following generalization of Proposition~\ref{P:vN halving}.

\begin{prop}\label{P:pV halving}
Let $u\in M$ be a properly infinite tripotent. Then there is a sequence $(v_n)$ of complete tripotents in $M_2(u)$ such that $v_n\in M_1(v_m)$ whenever $m\ne n$.
\end{prop}

We finish this section by the following proposition on decomposition of the triples of the form $pV$ to the finite and properly infinite summands.

\begin{prop}\label{P:pV decomp}
There is a decomposition $M=p_1V_1\oplus^{\ell_\infty}p_2V_2$ such that $p_1$ is finite and either $V_2=\{0\}$ or $p_2$ is properly infinite. Moreover, the following holds.
\begin{enumerate}[$(i)$]
    \item $p_1V_1$ is a finite JBW$^*$-triple.
    \item If $p_2\ne 0$, then any complete tripotent in $p_2V_2$ is properly infinite.
\end{enumerate}
\end{prop}

\begin{proof} By assumptions $M= pV,$ where $p$ is a projection in the von Neumann algebra $V$. If $p$ is finite, set $p_1=p$, $V_1=V$, $V_2=\{0\}$ and $p_2=0$. If $p$ is infinite, the existence of the decomposition follows from \cite[Proposition 6.3.7]{KR2}.

Assertion $(i)$ is a consequence of Lemma~\ref{L:finite projections basic facts}$(a)$ and Proposition~\ref{P:vN finite trip}.

Assertion $(ii)$ then follows by combining Corollary~\ref{cor:pV properly infinite trip}, Proposition~\ref{P:vN properly infinite trip} and  Lemma~\ref{L:pV complete}$(b)$.
\end{proof}

\section{Symmetric and antisymmetric parts of von Neumann algebras}\label{sec:symmetric and antisymmetric}

In this section we study the three preorders and the notion of finiteness in three types of summands from the representation of JBW$^*$-triples. We will focus on the summands $A\overline{\otimes}C$ where $A$ is an abelian von Neumann algebra and $C$ is a Cartan factor of type $2$ or $3$ and on the summand $H(W,\alpha)$. Although the first two types of summands are of type I and the third one is continuous, it turns out that their structure is in a sense similar.

This section is divided into several subsections. In Subsection~\ref{sec:involutions general} we develop an abstract approach which can be used in all three cases. Subsection~\ref{sec:involutions cts} settles the continuous summand $H(W,\alpha)$ -- this turns out to be quite easy. In Subsection~\ref{subsec:involutions from conjugations} we show how the abstract setting can be applied to the type I summands and, moreover, we provide some results on the structure of the tensor product $A\overline{\otimes}B(H)$. The next two subsections are devoted to the analysis of the two type I summands. It turns out that the properties of $A\overline{\otimes}C$ where $C$ is a Cartan factor of type $3$ are very similar to the properties of the continuous summands. The remaining case, i.e., the summand  $A\overline{\otimes}C$ where $C$ is a Cartan factor of type $2$, turns out to be more complicated.

\subsection{Some general facts on linear involutions}\label{sec:involutions general}

Assume $W$ is a fixed von Neumann algebra and $\alpha$ is a linear involution on $W$ commuting with $^*$, i.e.,
\begin{enumerate}[$(i)$]
    \item $\alpha:W\to W$ is a linear mapping;
    \item $\alpha(xy)=\alpha(y)\alpha(x)$ for $x,y\in W$;
    \item $\alpha(x^*)=\alpha(x)^*$ for $x\in W$;
    \item $\alpha(\alpha(x))=x$ for $x\in W$.
\end{enumerate}
Since $\alpha$ can be viewed as a $*$-isomorphism of $W$ onto the opposite algebra $W^{op}$, $\alpha$ is necessarily an isometry and a weak$^*$-to-weak$^*$ isomorphism.

Set
$$H(W,\alpha)=\{x\in W\setsep \alpha(x)=x\}\mbox{ and }H^-(W,\alpha)=\{x\in W\setsep \alpha(x)=-x\}.$$
Then $H(W,\alpha)$ is a JBW$^*$-subalgebra of $W$ and $H^-(W,\alpha)$ is a JBW$^*$-subtriple of $W$.
We start by the following abstract decomposition result. 

\begin{lemma}\label{L:H(W,alpha) decomposition}
There are central projections $z_1,z_2\in W$ with the following properties.
\begin{enumerate}[$(a)$]
    \item $z_1z_2=0$;
    \item $\alpha(x)=x$ for each $x$ in the center of $z_1W$;
    \item The von Neumann algebra $z_1W$ is invariant under $\alpha$;
    \item $H(W,\alpha)$ is triple-isomorphic to $H(z_1W,\alpha)\oplus^{\ell_\infty} z_2W$;
    \item $H^-(W,\alpha)$ is triple-isomorphic to $H^-(z_1W,\alpha)\oplus^{\ell_\infty} z_2W$.
    \end{enumerate}
\end{lemma}

\begin{proof} Consider the family 
$$\{z\in W\setsep z\mbox{ is a central projection such that }\alpha(z)z=0\}$$
ordered in the standard way. This family is nonempty (it contains $0$) and it is clear that for each totally ordered subfamily its supremum belongs to the family (as $\alpha$ is weak$^*$-to-weak$^*$ continuous). Hence by Zorn's lemma there is a maximal element of this family. Fix one and denote it by $z_2$.  

Then $z_2$ is a central projection, $\alpha(z_2)$ is also a central projection and $z_2\alpha(z_2)=0$. Set $z_1=1-(z_2+\alpha(z_2))$.
Then $z_1$ is again a central projection. Let us check  properties $(a)$--$(d)$.

Property $(a)$ is obvious. 
Let us continue by proving $(b)$. Let $z\in W$ be a central projection with $z\le z_1$. If $z\ne \alpha(z)$, then one of the projections $z-z\alpha(z)$ and $\alpha(z)-z\alpha(z)$ is not zero. Moreover,
since $\alpha(z-z\alpha(z))=\alpha(z)-z\alpha(z)$, both are nonzero. Set $w=z-z\alpha(z)$. Then $w\le z\le z_1$, $w\ne 0$ and $\alpha(w)w=0$. Thus
$$(z_2+w)\alpha(z_2+w)=z_2\alpha(z_2)+w\alpha(z_2)+z_2\alpha(w)+w\alpha(w)=0,$$
a contradiction with maximality of $z_2$. Since the center is the closed linear span of central projections,  assertion $(b)$ follows.

$(c)$ Fix $x\in z_1W$. Then
$$\alpha(x)=\alpha(xz_1)=\alpha(z_1)\alpha(x)=z_1\alpha(x)\in z_1W,$$
where we used property $(b)$.

$(d)$ and $(e)$: Any element $x\in W$ can be expressed 
as $x=xz_1+xz_2+x\alpha(z_2)$. Define an operator
$T:W\mapsto z_1W\oplus z_2W$ by $Tx=z_1x+z_2x$. It is clearly a weak$^*$-to-weak$^*$ continuous $*$-homomorphism. Moreover,
$T|_{H(W,\alpha)}: H(W,\alpha) \to H(z_1W,\alpha)\oplus z_2W$ is a bijection. 
(Indeed,  $x\in H(W,\alpha)$ if and only if $\alpha(x)\alpha(z_2)=xz_2$ and $\alpha(x)z_1=xz_1$.)
Similarly, $T|_{H^-(W,\alpha)}: H^-(W,\alpha)\to H^-(z_1W,\alpha)\oplus z_2W$ is a bijection. 
This completes the proof.
\end{proof} 

\begin{remark}
The previous lemma shows that we can assume that $\alpha$ satisfies moreover the condition
\begin{enumerate}[$(v)$]
    \item $\alpha(x)=x$ for each $x$ in the center of $W$.
\end{enumerate}
Indeed, we can decompose $H(W,\alpha)$ and $H^{-}(W,\alpha)$ into two direct summands one of them satisfies the additional condition and the second one is a von Neumann algebra (and von Neumann algebras are addressed in Subsection~\ref{subsec:vonNeumann algebras}).

Thus, in the sequel we will assume that $\alpha$ satisfies $(v)$ as well. Linear involutions satisfying $(v)$ are called {\em central} (cf. \cite{stormer}).
\end{remark}

\begin{obs}\label{obs:pf=alpha(pi)}
If $u\in H(W,\alpha)$ or $u\in H^-(W,\alpha)$ is a tripotent, then $u$ is a partial isometry in $W$. Moreover, $p_i(u)=\alpha(p_f(u))$. In particular, the initial projection is uniquely determined by the final one (and vice versa). 
\end{obs}

\begin{proof}
Assume $u\in H(W,\alpha)$. Then
$$\alpha(p_f(u))=\alpha(uu^*)=\alpha(u)^*\alpha(u)=u^*u=p_i(u).$$
If $u\in H^-(W,\alpha)$, the same equalities are valid.
\end{proof}

An immediate consequence is the following characterization (cf. Remark~\ref{rem:vN subtriples}).

\begin{prop}\label{P:H(W,alpha) abstract le2} Let $M=H(W,\alpha)$ or $M=H^-(W,\alpha)$. Let $u,e\in M$ be two tripotents. Then
$$u\le_2 e\Leftrightarrow p_i(u)\le p_i(e) \Leftrightarrow p_f(u)\le p_f(e).$$
\end{prop}

\begin{lemma}\label{L:p sim alpha(p)}
Let $p\in W$ be a projection. Then $p\sim \alpha(p)$.
\end{lemma}

\begin{proof}
By the comparability theorem \cite[Theorem 2.1.3]{Sak71} there is a central projection $z$ such that $zp\precsim z\alpha(p)$ and $(1-z)\alpha(p)\precsim (1-z)p$. Since $\alpha(zp)=z\alpha(p)$ (recall that $\alpha(z)=z$ by condition $(v)$) and $\alpha(\alpha(p))=p$, without  loss of generality assume $\alpha(p)\precsim p$. It follows that
$$p=\alpha(\alpha(p))\precsim \alpha(p),$$
hence $p\sim \alpha(p)$ (by \cite[Proposition V.1.3]{Tak}).
\end{proof}

The following proposition says that direct summands of $H(W,\alpha)$ and those of $H^{-}(W,\alpha)$ which are isomorphic to a JBW$^*$-algebra are only the trivial ones. Assertion $(a)$ may be deduced also from \cite[Proposition 4.3.6]{hanche1984jordan} or \cite[Theorem 2.3]{Edwards77}.

\begin{prop}\label{P:H(W,alpha) direct summand}
\begin{enumerate}[$(a)$]
    \item Any direct summand of $H(W,\alpha)$ has form $H(zW,\alpha|_{zW})$ where $z\in W$ is a central projection.
    \item Any direct summand of $H^{-}(W,\alpha)$ which is isomorphic to a JBW$^*$-algebra is of the form $H^{-}(zW,\alpha|_{zW})$ where $z\in W$ is a central projection.
\end{enumerate}
\end{prop}

\begin{proof}
The two cases may be proved simultaneously.  Since $H(W,\alpha)$ is a JBW$^*$-algebra, each of its direct summands is also a JBW$^*$-algebra. Hence assume  $M=H(W,\alpha)$ or $M=H^{-}(W,\alpha)$ and let $I$ be a direct summand of $M$ isomorphic to a JBW$^*$-algebra. By Proposition~\ref{P:direct summand}$(b)$ there is a tripotent $u\in M$ with $I=M_2(u)$ and $M_1(u)=0$.

Let $p=p_i(u)$. Then $p_f(u)=\alpha(p)$ by Observation~\ref{obs:pf=alpha(pi)}.
We claim that the projections $p$ and $1-\alpha(p)$ are centrally orthogonal.
Assume not. By \cite[Lemma V.1.7]{Tak} there is a nonzero $x\in W$ with 
$x=(1-\alpha(p))xp$. Then $\alpha(x)=\alpha(p)\alpha(x)(1-p)$, hence clearly $x$ and $\alpha(x)$ are linearly independent. Note that
$y=x+\alpha(x)\in H(W,\alpha)$, $w=x-\alpha(x)\in H^{-}(W,\alpha)$, both are nonzero and belong to $W_1(u)$.
This is a contradiction with the assumption that $M_1(u)=\{0\}$. 

Similarly we prove that $\alpha(p)$ and $1-p$ are centrally orthogonal.
Next, in the same way as in the proof of Lemma~\ref{L:vN direct summand} we 
show that $p=\alpha(p)$ and it is a central projection. This completes the proof.
\end{proof}

\subsection{The case $H(W,\alpha)$ for $W$ continuous}\label{sec:involutions cts}

Assume that $W$ is a continuous von Neumann algebra and $\alpha$ is a central involution, i.e.,  an involution on $W$ with the properties $(i)$--$(v)$. Then the situation is quite easy as witnessed by the following results.

\begin{lemma}\label{L:H(W,alpha) construction of tripotent}
Let $p\in W$ be a projection. Then there is a tripotent $u\in H(W,\alpha)$ with $p_i(u)=p$ (and $p_f(u)=\alpha(p)$).
\end{lemma}

\begin{proof}
Since $W$ is continuous, there are two projections $q_1\perp q_2$, $q_1\sim q_2$ such that $p=q_1+q_2$ (see \cite[Proposition 2.2.13]{Sak71}). By the Lemma~\ref{L:p sim alpha(p)} we have $q_2\sim\alpha(q_2)$, thus $q_1\sim\alpha(q_2)$. Let $v\in W$ be a partial isometry with $p_i(v)=q_1$ and $p_f(v)=\alpha(q_2)$. Then $\alpha(v)$ is also a partial isometry and $p_i(\alpha(v))=q_2$ and $p_f(\alpha(v))=\alpha(q_1)$. Since $\alpha(q_1)\perp\alpha(q_2)$, $u=v+\alpha(v)$ is a tripotent in $H(W,\alpha)$ with $p_i(u)=p$. 
\end{proof}

The following proposition follows from  \cite[Lemma 2.6]{BuPe02} and Proposition~\ref{P:le0=le2}. We present here an alternative proof because it is essentially the same as the proof for the summand $A\overline{\otimes}C$ with $C$ being a type 3 Cartan factor given in Subsection~\ref{subsec:C3}, demonstrating so similarity of the structure of the two summands.

\begin{prop}\label{P:H(W,alpha) is finite} $H(W,\alpha)$ is a finite JBW$^*$-algebra. In particular, the relations $\le_0$ and $\le_2$ coincide in $H(W,\alpha)$.
\end{prop}

\begin{proof} 
It is clear that $M=H(W,\alpha)$ is a JBW$^*$-subalgebra of $W$. Let us prove that $\le_0$ and $\le_2$ coincide. So, assume that $u,e\in H(W,\alpha)$ are two tripotents with $u\not\le_2 e$.
Then $p_i(u)\not\le p_i(e)$. Then $1-p_i(e)\not\le 1-p_i(u)$, thus $q=(1-p_i(e))-(1-p_i(e))\wedge(1-p_i(u))$ is a nonzero projection. By Lemma~\ref{L:H(W,alpha) construction of tripotent} there is a tripotent $v\in H(W,\alpha)$ with $p_i(v)=q$. Then $v\in M_0(e)\setminus M_0(u)$, thus $u\not\le_0 e$.

So, $M$ is finite by Proposition~\ref{P:le0=le2}.
\end{proof}

\subsection{Linear involutions induced by conjugations and structure of $A\overline{\otimes}B(H)$}\label{subsec:involutions from conjugations}

An important example of an involution, used among others to define Cartan factors of types $2$ and $3$ is the transpose. It is defined using a conjugation on a Hilbert space. We recall that a {\em conjugation} is a conjugate-linear isometry of period two. Any such mapping can be expressed as the canonical coordinate-wise conjugation with respect to an orthonormal basis. We will use this concrete representation. Let us fix the basic setting.

Let $H=\ell^2(\Gamma)$ for a set $\Gamma$ (i.e., we have a Hilbert space with a fixed orthonormal basis). We will assume that $\dim H\ge 2$, i.e., $\Gamma$ contains at least two distinct points. Let $(e_\gamma)_{\gamma\in\Gamma}$ be the canonical orthonormal basis. For $\xi\in H$ let $\overline{\xi}$ denote the canonical conjugate of $\xi$, i.e., $$\overline{\xi}(\gamma)=\overline{\xi(\gamma)},\qquad\gamma\in\Gamma.$$
For $x\in B(H)$ we denote by $x^t$ the operator defined by
$$x^t(\xi)=\overline{x^*(\overline{\xi})},\qquad\xi\in H.$$
Then $x^t$ is the transpose of $x$ with respect to the canonical basis, as
$$\ip{x^t (e_\gamma)}{e_\delta}=
\ip{\overline{x^*(e_\gamma)}}{e_\delta}=\overline{\ip{x^*(e_\gamma)}{e_\delta}}=\ip{x(e_\delta)}{e_\gamma},
\qquad \gamma,\delta\in\Gamma.
$$
Moreover, the mapping $\alpha:x\mapsto x^t$ is a central linear involution
on  $B(H)$ (i.e., it satisfies conditions $(i)$--$(v)$ from Section~\ref{sec:involutions general}). Note that $B(H)$ is a factor, so its center is trivial.
Hence, if $M\subset B(H)$ is a von Neumann algebra such that $x^t\in M$ for each $x\in M$ we may define the following two subtriples of $M$:
$$M_s=\{x\in M\setsep x^t=x\} \mbox{ and }M_a=\{x\in M\setsep x^t=-x\}.$$
They are of the form $H(M,\alpha)$ and $H^-(M,\alpha)$, hence the results of Section~\ref{sec:involutions general} apply. However, the restriction of the involution to $M$ need not be central, hence Lemma~\ref{L:H(W,alpha) decomposition} is important when treating abstract involutions. This is witnessed by the following example which may be proved by a direct calculation.

\begin{example} Let $H=\ce^2$ and $M\subset B(H)$ be formed by diagonal operators. Then $M$ is an abelian von Neumann algebra, isomorphic to $\ce^2$ (with the maximum norm). Moreover, the following holds.
\begin{enumerate}[$(1)$]
    \item If $H$ is equipped with the conjugation generated by the canonical orthonormal basis (i.e., formed by the vectors $(1,0)$, $(0,1)$), then $M$ is invariant under taking the transpose and, moreover, $M_s=M$ (and $M_a=\{0\}$).
    In particular, taking the transpose is a central involution on $M$.
    \item Equip $H$ with the conjugation defined by the ortnonormal basis formed by the vectors $\xi_1=\frac1{\sqrt2}(1,i)$, $\xi_2=\frac1{\sqrt2}(i,1)$.
    Then $M$ is invariant under taking the transpose  and, moreover, $$M_s=\{(x_1,x_2)\setsep x_1=x_2\} \mbox{ and }
M_a=\{(x_1,x_2)\setsep x_1=-x_2\}.$$
In particular, the involution on $M$ defined by taking the transpose is not central.
    \item Equip $H$ with the conjugation defined by the ortnonormal basis formed by the vectors $\eta_1=(\frac{\sqrt3}2,\frac i2)$, $\eta_2=(\frac i2,\frac{\sqrt3}2)$. Then $M$ is not invariant under taking the transpose.
\end{enumerate}
\end{example} 

We will not consider the abstract setting, but only the setting corresponding to Cartan factors. (This is enough due to the structure theory of JBW$^*$-triples and in this case we avoid the pathologies pointed out in the previous example.)

Recall that {\em Cartan factors of type 2} are exactly triples of the form $B(H)_a$ and {\em Cartan factors of type 3} are triples of the form $B(H)_s$.

Moreover, fix a $\sigma$-finite abelian von Neumann algebra $A$. We will address the triples $A\overline{\otimes}B(H)_a$ and $A\overline{\otimes}B(H)_s$ (which are the summands from \eqref{eq:representation of JBW* triples} corresponding to Cartan factors of types 2 and 3).

Fix a representation $A=C(L)=L^\infty(\mu)$ where $L$ is a compact space and $\mu$ is a probability measure on $L$ (see Subsection \ref{subsec:representation}). Then $A$ is canonically represented as the maximal abelian von Neumann subalgebra of  $B(L^2(\mu))$ consisting of multiplication operators (cf. \cite[comments before Lemma IV.7.5]{Tak} or \cite[Corollary 2.9.3]{Sak71}). If we equip $L^2(\mu)$ by the canonical conjugation (i.e., pointwise complex conjugation), then $A$ is clearly invariant under taking the transpose. Moreover, the involution on $A$ defined by taking the transpose is central,  i.e., $A=A_s$.

Let us look at the tensor products. Recall that $A\overline{\otimes}B(H)$ is canonically embedded into $B(L^2(\mu)\otimes H)$. Moreover, the tensor product Hilbert space  
$L^2(\mu)\otimes H$ can be identified with $\ell^2(\Gamma,L^2(\mu))$ (recall that $H=\ell^2(\Gamma)$) or with the Lebesgue-Bochner space $L^2(\mu,H)$ (cf. \cite[p. 257]{Tak}). These identifications are done by the equalities
$$ \sum_{\gamma\in \Gamma} f_\gamma\otimes e_\gamma = (f_\gamma)_{\gamma\in\Gamma}=\sum_{\gamma\in \Gamma} f_\gamma e_\gamma
$$
where 
$f_\gamma\in L^2(\mu)$ for $\gamma\in\Gamma$.
Note that the middle expression is a general element of $\ell^2(\Gamma,L^2(\mu))$ and any $\f\in L^2(\mu,\ell^2(\Gamma))$ can be expressed as the right-hand side.
Indeed, let us define a sesquilinear operator $[\cdot,\cdot]:L^2(\mu,H)^2\to L^1(\mu)$ by
$$[\f_1,\f_2](\omega)=\ip{\f_1(\omega)}{\f_2(\omega)}, \quad\omega\in L,\ \f_1,\f_2\in   L^2(\mu,\ell^2(\Gamma)).$$
It is clear that $[\f_1,\f_2]$ is a measurable function and, moreover, the Cauchy-Schwarz inequality implies that 
$$ \norm{[\f_1,\f_2]}_1\le \norm{\f_1}_2\norm{\f_2}_2.$$
Then any $\f\in L^2(\mu,\ell^2(\Gamma))$ can be expressed
as
$$\f=\sum_{\gamma\in\Gamma} [\f,e_\gamma]e_\gamma,$$
where the first occurence of $e_\gamma$ denotes the respective constant function.

Moreover,  equip $L^2(\mu,H)$ with the canonical pointwise conjugation defined by $\overline{\f}(\omega)=\overline{\f(\omega)}$ for $\omega\in L$.
Given $f\in L^\infty(\mu)$ and $x\in B(H)$ the operator $f\otimes x$ is defined by
$$f\otimes x ((g_\gamma)_{\gamma\in\Gamma}) =\sum_{\gamma\in\Gamma} (fg_\gamma)\otimes x(e_\gamma)=\sum_{\gamma\in\Gamma} fg_\gamma x(e_\gamma),$$
hence
$$\begin{aligned}(f\otimes x)^t((g_\gamma)_{\gamma\in\Gamma})&=\overline{(f\otimes x)^*((\overline{g_\gamma})_{\gamma\in\Gamma})}
=\overline{({\overline f}\otimes x^*)((\overline{g_\gamma})_{\gamma\in\Gamma})}
=\overline{\sum_{\gamma\in\Gamma} \overline{f}\overline{g_\gamma} x^*(e_\gamma)}
\\&=\sum_{\gamma\in\Gamma} fg_\gamma\overline{x^*(e_\gamma)}=
\sum_{\gamma\in\Gamma} fg_\gamma{x^t(e_\gamma)}
=f\otimes x^t ((g_\gamma)_{\gamma\in\Gamma}),\end{aligned}$$
i.e.,
$$(f\otimes x)^t=f\otimes x^t.$$
It follows that 
$A\overline{\otimes}B(H)_a=(A\overline{\otimes} B(H))_a$ and $A\overline{\otimes}B(H)_s=(A\overline{\otimes} B(H))_s$, hence the results of Subsection~\ref{sec:involutions general} apply.

In order to investigate deeper properties of these two triples we need a description of the von Neumann algebra $A\overline{\otimes}B(H)$.

Employing the terminology coming from the celebrated von Neumann bicommutant theorem, for each subset $\mathcal{S}$ of $B(H)$, we denote by $\mathcal{S}^{\prime}$ the set of elements of $B(H)$ commuting with all elements of $\mathcal{S}$, and we call $\mathcal{S}^{\prime}$  the commutant of $\mathcal{S}$. If $\mathcal{S}$ is self-adjoint, $\mathcal{S}^{\prime}$ is a self-adjoint subalgebra. 

\begin{lemma}\label{L:commutant of AotimesB(H)}
Both the commutant of $A\overline{\otimes}B(H)$ in $B(L^2(\mu,H))$ and its center equal to
$$\{f\otimes I\setsep f\in L^\infty(\mu)\}.$$
In particular, central projections in  $A\overline{\otimes}B(H)$
are of the form $\chi_E\otimes I$, where $E\subset L$ is a measurable set.
Moreover, the operator $f\otimes I$ acts on $L^2(\mu,H)$ as the multiplication by the scalar-valued function $f$.
\end{lemma}

\begin{proof} By \cite[Proposition IV.1.6(ii)]{Tak} we have
$$(A\overline{\otimes}B(H))'=A'\overline{\otimes}\ce I.$$
Moreover, $A$ is a maximal abelian C$^*$-subalgebra of $B(H)$ (cf. \cite[comments before Lemma IV.7.5]{Tak} or \cite[Corollary 2.9.3]{Sak71}), hence it is easy to see that 
$A'=A$. It follows that
$$(A\overline{\otimes}B(H))'=A\overline{\otimes}\ce I.$$
Moreover, it follows from 
\cite[Corollary IV.1.5]{Tak} that 
$$A\overline{\otimes}\ce I=\{f\otimes I\setsep f\in A\}$$
which completes the proof of the representation of the commutant and the center. The description of central projections then follows from the description of the projections in $L^\infty(\mu)$. The fact that $f\otimes I$ acts as described follows easily from the very definition of $f\otimes I$ (see, e.g., \cite[p. 183]{Tak}).
 \end{proof}

Now we will consider basic building blocks of the Hilbert space $L^2(\mu,H)$. For $\f\in L^2(\mu,H)$ set
$$X_{\f}=\{g\f\setsep g\in L^\infty(\mu)\}.$$ 
It is clearly a linear subspace of $L^2(\mu,H)$, but not necessarily closed. Denote by $\widehat{X}_{\f}$ its closure.

\begin{lemma}\label{L:X_f} Let $\f\in L^2(\mu,H)$. Then
\begin{enumerate}[$(a)$]
    \item $\widehat{X}_{\f}=\Big\{g\f\setsep g:L\to\ce\mbox{ measurable }, g\f\in L^2(\mu,H)\Big\}$;
    \item $(\widehat{X}_{\f})^\perp=\{\g\in L^2(\mu,H)\setsep [\f,\g]=0 \  \mu\mbox{-a.e.}\}$.
\end{enumerate}
\end{lemma}

\begin{proof} $(a)$ Let 
$$Y=\{g:L\to \ce\mbox{ measurable}\setsep \int \abs{g(\omega)}^2\norm{\f(\omega)}^2\di\mu(\omega)<\infty\}.$$
Then $Y$ is a Hilbert space when equipped with the norm
$$\norm{g}_Y^2=\int \abs{g(\omega)}^2\norm{\f(\omega)}^2\di\mu(\omega)$$
and $g\mapsto g\f$ is an isometric embedding of $Y$ into $L^2(\mu,H)$. Since bounded functions are dense in $Y$, it follows that $X_{\f}$ is dense in the image of $Y$. Therefore the image of $Y$ equals to $\widehat{X}_{\f}$.

$(b)$ The inclusion `$\supset$' is obvious. To prove the converse assume that 
$$E=\{\omega\in L\setsep[\f,\g](\omega)\ne0\}$$ has a nonzero measure. Then there is $K>0$ such that 
$$E'=\{\omega\in E\setsep \abs{[\f,\g](\omega)}\le K\}$$ has a nonzero measure. Then $h=\chi_{E'} \overline{[\f,\g]}\in L^\infty(\mu)$, hence $h\f\in X_{\f}$. Thus
$$\ip{h\f}{\g}=\int_{E'}\abs{[f,g]}^2\di\mu>0,$$
hence $\g\notin(\widehat{X}_{\f})^\perp$.
\end{proof}

The following lemma provides an effective characterization of the operators on $L^2(\mu,H)$ which belong to $A\overline{\otimes}B(H)$.

 \begin{lemma}\label{L:AotimesB(H) charact} 
Let $T\in B(L^2(\mu,H))$. The following assertions are equivalent
\begin{enumerate}[$(1)$]
    \item $T\in A\overline{\otimes}B(H)$;
    \item $T(h\f)=hT(\f)$ for any $\f\in L^2(\mu,H)$ and $h\in L^\infty(\mu)$;
    \item Whenever $\f\in L^2(\mu,H)$ and $B\subset L$ is a measurable set such that $\f|_B=0$ $\mu$-a.e., then $T(\f)|_B=0$ $\mu$-a.e.;
    \item $\ip{T(\f_1)}{\f_2}=0$ whenever $\f_1,\f_2\in L^2(\mu,H)$ and there is a measurable set $B\subset L$ such that $f_1|_B=0$ and $f_2|_{L\setminus B}=0$.
  \end{enumerate}
\end{lemma}

\begin{proof}
$(1)\Leftrightarrow(2)$ Observe that condition $(2)$ means that $T$ commutes with $h\otimes I$ for each $h\in L^\infty(\mu)$. Hence the equivalence follows from Lemma~\ref{L:commutant of AotimesB(H)} taking into account that $A\overline{\otimes}B(H)$ equals to its bicommutant (being a von Neumann algebra). 

$(2)\Rightarrow(3)$ Assuming $(2)$ we have $\chi_B T(\f)=T(\chi_B \f)=T(0)=0$.

$(3)\Rightarrow(4)$ This is obvious.

$(4)\Rightarrow(3)$ Assume $(3)$ does not hold and $B$ and $\f$ witness it. Then the pair $\f_1=\f$ and $\f_2=\chi_{B}T\f$ witnesses the failure of $(4)$.

$(3)\Rightarrow(2)$ Fix $\f\in L^2(\mu,H)$ and a measurable set $B\subset L$. Then 
$$T(\f)=T(\chi_B\f+\chi_{L\setminus B}\f)=T(\chi_B\f)+T(\chi_{L\setminus B}\f).$$
Moreover, by $(3)$ we deduce that $T(\chi_B\f)=0$ a.e. on $L\setminus B$ and $T(\chi_{L\setminus B}\f)=0$ a.e. on $B$.
Necessarily $T(\chi_B\f)=\chi_B T(\f)$ and $T(\chi_{L\setminus B}\f)=\chi_{L\setminus B} T(\f)$.

In other words, the formula from $(2)$ hold if $h$ is a characteristic function. By linearity of $T$ this formula can be extended to simple functions, by density of simple functions it holds for any $h\in L^\infty(\mu)$.
\end{proof}

The previous lemma enables us to construct some special operators in $A\overline{\otimes}B(H)$. The construction is described in the following lemma.

\begin{lemma}\label{L:operators on Xf}
Let $\f,\g\in L^2(\mu,H)$ be such that $\norm{\g(\omega)}\le C\norm{\f(\omega)}$ $\mu$-a.e. (for some $C>0$). Define an operator $T\in B(L^2(\mu,H))$ by setting
$$T(h\f)=h\g,\qquad h\f\in \widehat{X}_{\f}$$
and $T(\h)=0$ for $\h\in(\widehat{X}_{\f})^\perp$.
Then $T\in A\overline{\otimes}B(H)$. 

In particular, the orthogonal projection onto $\widehat{X}_{\f}$ belongs to $A\overline{\otimes}B(H)$. Further, if $\norm{g(\omega)}=\norm{\f(\omega)}$ $\mu$-a.e., then $T$ is a partial isometry.
\end{lemma}

\begin{proof}
It is clear that  $T$ is indeed a bounded linear operator. Moreover, it satisfies condition $(2)$ from Lemma~\ref{L:AotimesB(H) charact}, which completes the proof. The additional statements are obvious.
\end{proof}

Next we are going to present three lemmata characterizing abelian, finite and properly infinite projections in $A\overline{\otimes}B(H)$. Before formulating them we give an easy observation which we will often use.

\begin{obs}\label{obs:invariance Xf}
Let $P\in A\overline{\otimes}B(H)$ be a projection and let $\f\in\ran(P)$ be arbitrary. Then $\widehat{X}_{\f}\subset\ran(P)$.
\end{obs}

\begin{proof}
It follows immeadiately from Lemma~\ref{L:AotimesB(H) charact} that $X_{\f}\subset \ran(P)$ (using condition $(2)$). Since $\ran(P)$ is closed, necessarily $\widehat{X}_{\f}=\overline{X_{\f}}\subset\ran(P)$.
\end{proof}

\begin{lemma}\label{L:abelian pr in AotimesB(H)}
Let $P\in A\overline{\otimes}B(H)$ be a projection.
The following assertions are equivalent
    \begin{enumerate}[$(a)$]
        \item $P$ is abelian;
        \item $\ran(P)=\widehat{X}_{\f}$ for some $\f\in L^2(\mu,H)$;
        \item For any $\f,\g\in\ran(P)$ 
$$\f(\omega),\g(\omega)\mbox{ are linearly dependent for $\mu$-a.a. }\omega.$$
    \end{enumerate}
\end{lemma}

\begin{proof}
$(b)\Rightarrow (a)$
Let $T\in A\overline{\otimes} B(H)$ be such that $T=PTP$. It follows that $\widehat{X}_{\f}=\ran(P)$ is invariant for $T$, hence $T$ generates a bounded linear operator on the Hilbert space $Y$ defined in the proof of Lemma~\ref{L:X_f}, denoted again by $T$. 

Set $h=T(1)$. Condition $(2)$ of Lemma~\ref{L:AotimesB(H) charact} assures that $T(\chi_E)=\chi_E h$ for $E\subset L$ measurable.
Hence
$$\begin{aligned}
\int_E \abs{h(\omega)}^2\norm{\f(\omega)}^2\di\mu(\omega)&=\norm{\chi_E h}_Y^2=\norm{T(\chi_E)}_Y^2\le \norm{T}^2\norm{\chi_E}^2\\&
=\norm{T}^2\int_E \norm{\f(\omega)}^2\di\mu(\omega)\end{aligned}$$
for $E\subset K$ measurable. It follows that $\abs{h(\omega)}\le \norm{T}$ $\mu$-a.e. Since $T(g)=gh$ for any simple function and simple functions are dense, we deduce that $T$ is a multiplication operator.

Finally, any two multiplication operators commute, which completes the argument.

$(a)\Rightarrow(c)$: Assume there are $\f,\g\in\ran(P)$ and a set $E$ of positive measure such that $\f(\omega)$, $\g(\omega)$ are linearly independent for $\omega\in E$.

Set 
$$\f_1(\omega)=\begin{cases}\frac{\f(\omega)}{\norm{\f(\omega)}},& \omega\in E,\\ 0, &\omega\in L\setminus E,\end{cases}\quad\f_2(\omega)=\begin{cases}
\frac{\g(\omega)-\ip{\g(\omega)}{\f_1(\omega)}\f_1(\omega)}{\norm{\g(\omega)-\ip{\g(\omega)}{\f_1(\omega)}\f_1(\omega)}},& \omega\in E,\\ 0, &\omega\in L\setminus E.\end{cases}$$
Then $\f_1,\f_2\in\ran(P)$ (by Observation~\ref{obs:invariance Xf} and Lemma \ref{L:X_f}), $[\f_1,\f_2]=0$ and $\norm{\f_j(\omega)}=1$ for $\omega\in E$ and $j=1,2$.

Consider the operators defined by $T_1(\h)=[\h,\f_1]\f_2$ and $T_2(\h)=[\h,\f_2]\f_1$ for $\h\in L^2(\mu,H)$. Then $T_1$ and $T_2$ are contractive operators on $L^2(\mu,H)$. Moreover, they satisfy condition $(2)$ from Lemma~\ref{L:AotimesB(H) charact}, hence they belong to $A\overline{\otimes}B(H)$. Moreover,  $T_j=PT_jP$ for $j=1,2$ and $T_1T_2\ne T_2T_1$. Hence $P$ is not abelian.

$(c)\Rightarrow(b)$ Assume $(c)$ holds. Let $(\f_\lambda)_{\lambda\in\Lambda}$  be a maximal family of elements in $\ran(P)$ such that $\{\omega\in L \setsep  \f_{\lambda_1}(\omega)\ne0\ne \f_{\lambda_2}(\omega)\}$ has measure zero whenever $\lambda_1,\lambda_2$ are two distinct elements of $\Lambda$. Since $\mu$ is a probability measure,  $\Lambda$ is countable, thus we may assume $\Lambda\subset\en$. Set 
$$\f=\sum_{n\in\Lambda}\frac{1}{2^n}\frac{\f_n}{\norm{\f_n}}.$$
Then clearly $\f\in \ran(P)$, thus $\widehat{X}_{\f}\subset\ran(P)$. We claim that in fact $\widehat{X}_{\f}=\ran(P)$. To prove this take any $\g\in\ran(P)$. Then $E=\{\omega\in L\setsep \g(\omega)\ne0=\f(\omega)\}$ has measure zero. Indeed, otherwise the function $\chi_E\g$ would contradict maximality of the family  $(\f_\lambda)$. Hence, by the assumption we get that
$\g(\omega)$ is a scalar multiple of $\f(\omega)$ $\mu$-a.e.
Hence clearly $\g\in\widehat{X}_{\f}$.
\end{proof}

\begin{lemma}\label{L:finite in AotimesB(H)}
 Let $P\in A\overline{\otimes}B(H)$ be a projection.
 \begin{enumerate}[$(i)$]
     \item Assume that $P$ is finite. Then $P$ is $\sigma$-finite. Moreover, there is a sequence $(\f_n)$ in $\ran(P)$ such that
     \begin{enumerate}[$(a)$]
         \item $[\f_n,\f_m]=0$ for $m\ne n$;
         \item $\span\bigcup_n \widehat{X}_{\f_n}$ is dense in $\ran(P)$;
         \item $\dim\span\{\f_n(\omega)\setsep n\in\en\}<\infty$ for $\mu$-almost all $\omega$.
     \end{enumerate}
    Moreover, the dimension function from assertion $(c)$ does not depend on the choice of the sequence $(\f_n)$.
    \item Assume that $P$ is infinite. Then there is a sequence $(\f_n)$ in $\ran(P)$ and a measurable set $E\subset L$ with $\mu(E)>0$ such that $(\f_n(\omega))_{n\in\en}$ is an orthonormal sequence for each $\omega\in E$.
 \end{enumerate}
\end{lemma}

\begin{proof}
$(i)$ Assume that $P$ is finite.
As the center of $A\overline{\otimes}B(H)$ is *-isomorphic to $L^\infty(\mu)$, where $\mu$ is a probability measure, it is a $\sigma$-finite algebra. By Proposition~6.3.10 in \cite{KR2} a finite projection whose central cover is $\sigma$-finite in the center has to be $\sigma$-finite. Therefore $P$ is $\sigma$-finite and the first assertion is proved.  

Further, let $(\f_j)_{j\in J}$ be a maximal family of nonzero elements of $\ran(P)$ such that $[\f_j,\f_k]=0$ whenever $j\ne k$. Let $Q_j$ denote the projection onto $\widehat{X}_{\f_j}$. Then $Q_j\in A\overline{\otimes} B(H)$ by Lemma~\ref{L:operators on Xf}, $Q_j\le P$ and the family $(Q_j)$ is orthogonal. So, the family is countable as $P$ is $\sigma$-finite by the previous paragraph. Moreover, by maximality we easily get that $P=\sum_j Q_j$. 

Hence, properties $(a)$ and $(b)$ are satisfied. Let us continue by proving $(c)$. We proceed by contradiction.
Assume that there is a measurable set $E\subset L$ with $\mu(E)>0$ such that
$$\dim\span \{\f_j(\omega)\setsep j\in J\}=\infty\mbox{ for }\omega\in E.$$
We will  show that there is a sequence $(\g_n)$ in $\ran(P)$ such that the sequence $(\g_n(\omega))$ is  orthonormal for almost all $\omega\in E$. This will be done by induction. 

Assume that $n\in\en$ and we have already constructed $\g_k$ for $k<n$ (the case $n=1$ covers the first induction step).

Let $(n_j,B_j)_{j\in C}$ be a maximal family such that for each $j\in C$ we have
\begin{itemize}
    \item $n_j\in J$ and $B_j\subset L$ is a measurable set of positive measure;
    \item  $\mu(B_j\cap B_k)=0$ for $k\ne j$;
    \item $\f_{n_j}(\omega)\ne 0$ and $\f_{n_j}(\omega) \perp\span\{\g_k(\omega)\setsep k<n\}$ for $\omega\in B_j$.
\end{itemize}
The existence of such a family follows from Zorn's lemma. Moreover, the family is countable (as $\mu$ is a probability measure) and $E\setminus\bigcup_{j\in C} B_j$ has measure zero (by the assumption and maximality). So, assuming $C\subset\en$ it is enough to set $\displaystyle \g'_n=\sum_{j\in C}\frac{1}{2^j}\frac{\chi_{B_j}\f_{n_j}}{\norm{\chi_{B_j}\f_{n_j}}}$ and $\displaystyle\g_n(\omega)=\frac{\g'_n(\omega)}{\norm{\g'_n(\omega)}}$ for $\omega\in E$. 

Further, let $R_n$ be the orthogonal projection onto $\widehat{X}_{\chi_E\g_n}$. Then the projections $(R_n)$ are mutually orthogonal and mutually equivalent (by Lemma~\ref{L:operators on Xf}). Thus $\sum_n R_n$ is properly infinite by Lemma~\ref{L:finite projections basic facts}$(e)$. Since this projection is majorized by $P$, necessarily 
$P$ is infinite, a contradiction.

This completes the proof of properties $(a)-(c)$. Let us point out that we have in fact proved that $(c)$ is a consequence of $(a)$ and $(b)$. It remains to prove the independence of the dimension function. This will be done by contradiction. Assume $(\f_n)$ and $(\g_n)$ are two sequences satisfying $(a)-(c)$ with different dimension functions. It means that there is a measurable set $E\subset L$ with $\mu(E)>0$ and natural numbers $k\ne m$ such that for each $\omega\in E$ we have
$$\dim\span\{\f_n(\omega)\setsep n\in\en\}=k \ \&\ 
\dim\span\{\g_n(\omega)\setsep n\in\en\}=m.$$
Similarly as above we may prove that $(\chi_E\otimes I)P$ can be expressed in two ways -- as a sum of $k$ mutually equivalent orthogonal abelian projections and also as a sum of $m$ mutually equivalent orthogonal abelian projections. It follows from \cite[Lemma V.1.26]{Tak} applied to the von Neumann algebra $(\chi_E\otimes I)PMP$ that $k=m$, a contradiction.

$(ii)$ Assume that $P$ is infinite, i.e., there is a projection $Q<P$ with $P\sim Q$. Let $U$ be a partial isometry with
$p_i(U)=P$ and $p_f(U)=Q$. Fix $\f\in\ran(P-Q)$. Without loss of generality we may assume that $\norm{\f(\omega)}=1$ whenever it is nonzero. Consider the sequence $\f_n$ defined by $\f_1=\f$ and $\f_{n+1}=U(\f_n)$ for $n\in\en$. Since $\widehat{X}_{\f_2}\perp \widehat{X}_{\f_1}$, it follows by an easy induction that the spaces $\widehat{X}_{\f_n}$, $n\in\en$, are mutually orthogonal. It follows that the sequence $(\f_n)$ has the required properties.
\end{proof}

\begin{lemma}\label{L:properly infinite in AotimesB(H)}
Let $P\in A\overline{\otimes}B(H)$ be a projection.
The following assertions are equivalent.
\begin{enumerate}[$(a)$]
     \item $P$ is properly infinite;
     \item There is a measurable set $E\subset L$ with $\mu(E)>0$ and a sequence $(\f_n)$ in $\ran(P)$ such that
     \begin{itemize}
         \item $\chi_{L\setminus E}\f=0$ for $\f\in \ran(P)$, and
         \item $(f_n(\omega))$ is a linearly indepenent sequence for each $\omega\in E$.
     \end{itemize}
     \item There is a measurable set $E\subset L$ with $\mu(E)>0$ and a sequence $(\f_n)$ in $\ran(P)$ such that
     \begin{itemize}
         \item $\chi_{L\setminus E}\f=0$ for $\f\in \ran(P)$, and
         \item $(f_n(\omega))$ is an orthonormal sequence for each $\omega\in E$.
     \end{itemize}
 \end{enumerate} 
\end{lemma}

\begin{proof}
$(c)\Rightarrow(b)$ This is trivial.

$(b)\Rightarrow(a)$ This follows easily from Lemma~\ref{L:finite in AotimesB(H)}$(i)$ (using the description of central projections given in Lemma \ref{L:commutant of AotimesB(H)}).

$(a)\Rightarrow(c)$ Let $(E_j)_{j\in J}$ be a maximal disjoint family of measurable sets of positive measure in $L$ such that for each $j\in J$ there is a family $(\f^j_n)$ in $\ran(P)$ such that $(\f^j_n(\omega))_n$ is an orthonormal sequence for each $\omega\in E_j$. Then $J$ is countable. Set $E=\bigcup_j E_j$ and 
$\f_n=\sum_j \chi_{E_j}\f^j_n$. Then $(\f_n)$ satisfies the second property. The first property is also satisfied, as otherwise $(\chi_{L\setminus E}\otimes I)P\ne0$. This projection is infinite, hence application of Lemma~\ref{L:finite in AotimesB(H)}$(ii)$ yields a contradiction with the maximality of the family $(E_j)$.
\end{proof}

We finish this subsection by the following easy observation describing the range of the transpose of a projection.

\begin{obs}\label{obs:range pt}
Let $P\in A\overline{\otimes}B(H)$ be a projection. Then the range of $P^t$ is $\{\overline{\f}\setsep \f\in\ran(P)\}$.
\end{obs}

\begin{proof}
$$\f\in \ran(P^t)\Longleftrightarrow P^t (\f)=\f\Longleftrightarrow \overline{P^*(\overline{\f})}=\f
\Longleftrightarrow P(\overline{\f})=\overline{\f}\Longleftrightarrow \overline{\f}\in\ran(P).$$
\end{proof}

\subsection{The symmetric case}\label{subsec:C3}

In this subsection we clarify the situation in the triples of the form $A\overline{\otimes}C$ where $C$ is a Cartan factor of type 3, i.e., of $A\overline{\otimes}B(H)_s$. (We keep the notation from the previous subsection.)
It turns out that the results are completely analogous to that from Subsection~\ref{sec:involutions cts} but the methods of proofs need to be different.

The following lemma is an analogue of Lemma~\ref{L:H(W,alpha) construction of tripotent}.

\begin{lemma}\label{L:C3 construction of tripotent}
Let $P\in A\overline{\otimes}B(H)$ be a projection. Then there is a partial isometry $U\in A\overline{\otimes}B(H)_s$ such that
$p_i(U)=P$ (and $p_f(U)=P^t$).
\end{lemma}

\begin{proof} The proof will be done in two steps.

{\sc Step 1:} Assume $P\ne 0$. Then there is a nonzero partial isometry $U\in A\overline{\otimes}B(H)_s$ such that
$p_i(U)\le P$ (and $p_f(U)\le P^t$).

\smallskip

Fix any $\f\ne 0$ in the range of $P$ (considered as a projection on $L^2(\mu,H)$). Then $\widehat{X}_{\f}\subset\ran(P)$ (by Lemma~\ref{L:AotimesB(H) charact}$(2)$).
Let us define a partial isometry by setting
$$U(g\f)=g\overline{\f}, \quad g\f\in \widehat{X}_{\f}$$
and $U=0$ on  $\widehat{X}_{\f}^\perp$ (see Lemma~\ref{L:operators on Xf}). Moreover, $U^*$ is the `inverse partial isometry', i.e., $U^*(g\overline{\f})=g\f$ for $g\overline{\f}\in \widehat{X}_{\overline{\f}}$ and $U^*=0$ on $\widehat{X}_{\overline{\f}}^\perp$. It follows that
$$U^t(g\f)=\overline{U^*(\overline{g\f})}=\overline{U^*(\overline{g}\overline{\f})}=\overline{\overline{g}\f}=g\overline{\f}=U(g
\f)$$
for $g\f\in \widehat{X}_{\f}$ and, for any $\h\in X_{\f}^\perp$ we have 
$$U^t(\h)=\overline{U^*(\overline{\h })}=0=U(\h)$$
as $\overline{\h}\in X_{\overline{\f}}^\perp$.
It follows that $U^t=U$.

\medskip

{\sc Step 2:} Let $(U_\theta)_{\alpha\in\Lambda}$ be a maximal family of nonzero partial isometries in $A\overline{\otimes}B(H)_s$ such that $p_i(U_\theta)\le P$ for each $\alpha$ and $p_i(U_\theta)$ are pairwise orthogonal. Then $U=\sum_\theta U_\theta$ is a partial isometry in $A\overline{\otimes}B(H)_s$
such that $p_i(U)=\bigvee_\theta p_i(U_\theta)\le P$. Moreover, $p_i(U)=P$, otherwise we could apply Step 1 to $P-p_i(U)$ and come to a contradiction with maximality.
\end{proof}

Further we get the following proposition, which can be proved exactly as Proposition~\ref{P:H(W,alpha) is finite} by just replacing Lemma~\ref{L:H(W,alpha) construction of tripotent} with Lemma~\ref{L:C3 construction of tripotent}.

\begin{prop}\label{P:AotimesB(H)s is finite} $A\overline{\otimes} B(H)_s$ is a finite JBW$^*$-algebra. 
In particular, the relations $\le_0$ and $\le_2$ 
coincide.
\end{prop}

\subsection{The antisymmetric case} \label{subsec:antisymmetric}       

In this subsection we analyze the behaviour of tripotents in $A\overline{\otimes}C$ where $C$ is a Cartan factor of type 2, i.e., in $A\overline{\otimes}B(H)_a$.
It turns out that the structure of $A\overline{\otimes}B(H)_a$ is more complicated than that of its symmetric counterpart. In particular, $A\overline{\otimes}B(H)_a$ is not a JBW$^*$-subalgebra of $A\overline{\otimes}B(H)$ as it is not stable under the Jordan product. (In fact, $x\circ y\in A\overline{\otimes}B(H)_s$ whenever $x,y\in A\overline{\otimes}B(H)_a$.) However, if $\dim H$ is either even or infinite, then $A\overline{\otimes}B(H)_a$ admits a unitary element (see \cite[Proposition 2]{HoMarPeRu2002} for an explicit formula), so it is a JBW$^*$-algebra. 
If $\dim H$ is odd, it is not hard to show that there is no unitary element.

We are going to analyze the behaviour of tripotents in detail and compare it with the results on the symmetric case. It turns out that the analogue of Lemma~\ref{L:C3 construction of tripotent} (and of Lemma~\ref{L:H(W,alpha) construction of tripotent}) fails in this case. Instead, we have the characterizations below. As a byproduct of our approach we reprove the above-mentioned results of existence and non-existence of unitary elements, providing so a deeper understanding of these phenomenon.

 To simplify the notation we will denote $M=A\overline{\otimes}B(H)$ (then $M_a=A\overline{\otimes}B(H)_a$).
We start by the following observation on the property of antisymmetric operators.

\begin{lemma}\label{L:type2 perp} Let $u\in M_a$  and $\f\in L^2(\mu,H)$. Then we have the following.
\begin{enumerate}[$(i)$]
    \item  $\ip {u(\f)}{\overline\f}=0$;
    \item   If $u$ is a tripotent and $\f\in\ran p_i(u)$, then $\g=\overline{u(\f)}\in\ran(p_i(u))$, $\g\perp\f$ and $u(\g)=-\overline{\f}$.
\end{enumerate}
\end{lemma}

\begin{proof}
$(i)$ we have
$$\ip{u^t(\f)}{\overline{\f}}
=\ip{\overline{u^*(\overline\f)}}{\overline{\f}} 
=\overline{\ip{u^*(\overline{\f})}\f} 
=\overline{\ip{\overline{\f}}{u(\f)}}
=\ip{u(\f)}{\overline{\f}},$$
thus, assuming $u^t=-u$, we deduce  $\ip {u(\f)}{\overline\f}=0$.

$(ii)$ By $(i)$ we have $u(\f)\perp\overline{\f}$, thus $\overline{u(\f)}\perp\f$, i.e., $\g\perp\f$. Since $u(\f)\in\ran p_f(u)$, $\g=\overline{u(\f)}\in \ran p_i(u)$ (cf. Observation~\ref{obs:range pt}). Moreover,
$$u(\g)=-u^t(\overline{u(\f)})=-\overline{u^*u(\f)}=-\overline{\f}.$$
\end{proof}

The next lemma characterizes antisymmetric tripotents.

\begin{lemma}\label{L:C2 construction of tripotent}
Let $u\in M$ and let $p\in M$ be a projection. Then $u$ is a partial isometry in $M_a$ with $p=p_i(u)$ if and only if there are a projection $q\le p$ and a partial isometry $v\in M$ such that:
\begin{enumerate}[$(i)$]
    \item $p_i(v)=q$ and $p_f(v)=p^t-q^t$;
    \item $u=v-v^t$.
\end{enumerate}
\end{lemma}

\begin{proof} Let us start by the `if part'. Assume that $q$ and $v$ satisfying the required properties exist. It follows easily from $(i)$ that $v^t$ is a partial isometry with $p_i(v^t)=p-q$ and $p_f(v^t)=q^t$. Hence $v\perp v^t$, so by $(ii)$ we deduce that $u$ is a partial isometry with $p_i(u)=q+p-q=p$.

\smallskip

Let us continue with the `only if part'. It will be done in two steps.

{\sc Step 1:} Assume $u\ne0$. Then there are a nonzero projection $p'\le p$, a projection $q'\le p'$ and a partial isometry $v'\in M$ such that
\begin{enumerate}[$(i')$]
    \item $p_i(v')=q'$ and $p_f(v)=(p')^t-(q')^t$;
    \item $up'=v'-(v')^t$.
\end{enumerate}

Since $u\ne0$, we have $p\ne 0$.
Fix any $\f\ne 0$ in the range of $p$ (considered as a projection on $L^2(\mu,H)$). Then $\widehat{X}_{\f}\subset\ran(p)$ (by Observation~\ref{obs:invariance Xf}).
Moreover,  Lemma~\ref{L:AotimesB(H) charact}$(2)$ yields that $u(g\f)=g u(\f)$ whenever $g\f\in\widehat{X}_{\f}$.
By Lemma~\ref{L:type2 perp} we get 
$$0=\ip{u(g\f)}{\overline{g\f}}=\int |g|^2 [u(\f),\overline{\f}]\di\mu$$
for any $g\f\in \widehat{X}_{\f}$, thus $[u(\f),\overline{\f}]=0$ a.e.
It follows that $\widehat{X}_{\overline{\f}}\perp \widehat{X}_{u(\f)}$.

Let $q'$ be the projection onto $\widehat{X}_{\f}$. Then $q'\in M$ (by Lemma~\ref{L:operators on Xf}) and $(q')^t$ is the projection onto   $\widehat{X}_{\overline{f}}$ (by Observation~\ref{obs:range pt}). Let $v'=uq'$. Since $q'\le p$, $v'$ is a partial isometry with $p_i(v')=q'$. Moreover,  $\ran(p_f(v'))=\widehat{X}_{u(\f)}$ (cf. Lemmata \ref{L:X_f}$(a)$ and \ref{L:AotimesB(H) charact}$(2)$). Then $\ran(p_f(v')^t)=\widehat{X}_{\overline{u(\f)}}$ (again by Observation~\ref{obs:range pt}), thus $p_f(v')^t\perp q'$. Moreover,
$p_f(v')\le p_f(u)=p^t$, thus $p_f(v')^t\le p$. Hence, $p'=q'+p_f(v')^t\le p$ (and it is a projection).

Hence, $q'\le p'\le p$ and $(i')$ is valid. Moreover, $(ii')$ is satisfied by the definition of $p'$ as
$p_i((v')^t)=p_f(v')^t$.

{\sc Step 2:} Let $(p_\theta)_{\alpha\in\Lambda}$ be a maximal orthogonal family of projections below $p$ such that for each $\alpha\in \Lambda$ there are a projection $q_\theta\le p$ and a partial isometry $v_\theta\in M$ satisfying conditions $(i')$ and $(ii')$ (with $p_\theta,q_\theta,v_\theta$ in place of $p',q',v'$).  Set $p'=\sum_\theta p_\theta$, $q'=\sum_\theta q_\theta$ and $v'=\sum_\theta v_\theta$. Then $p',q',v'$ satisfy conditions $(i')$ and $(ii')$.

If $p'\ne p$, then $w=u(p-p')$ is a nozero partial isometry. Moreover, $w=u-up'$ and $up'$ is antisymmetric by the already proved `if part'. Thus $w$ is antisymmetric. If we apply Step 1 to $w$ we get a contradiction with the maximality of the family $(p_\theta)_{\alpha\in\Lambda}$.
\end{proof}

The next proposition is a variant of Lemmata~\ref{L:C3 construction of tripotent} and~\ref{L:H(W,alpha) construction of tripotent} in the antisymmetric case. Assertion $(i)$ shows that antisymmetric tripotents have an `even nature'.

\begin{prop}\label{P:C2 projections}
Let  $p\in M$ be a projection. Then the following holds.
\begin{enumerate}[$(i)$]
    \item $p=p_i(u)$ for some $u\in M_a$ if and only if there are projections $q_1,q_2\in M$ with $q_1\perp q_2$, $q_1\sim q_2$ and $p=q_1+q_2$;
    \item There is a nonzero tripotent $u\in A \overline{\otimes}B(H)_a$ with $p_i(u)\le p$ if and only if $p$ is not abelian.
\end{enumerate}
\end{prop} 

\begin{proof} $(i)$
The `only if part' follows from Lemma~\ref{L:C2 construction of tripotent}. Indeed, let $q$ be the projection provided by this lemma. Then $q\sim p^t-q^t$. Further, $p^t-q^t\sim p-q$ by Lemma~\ref{L:p sim alpha(p)} (recall that taking the transpose is a central involution), hence the assertion follows.

Let us continue by the `if part'. So, assume we have $q_1$ and $q_2$
as above. By Lemma~\ref{L:p sim alpha(p)} we know that $q_2\sim q_2^t$, hence $q_1\sim q_2^t$. Fix a partial isometry $v\in M$ with $p_i(v)=q_1$ and $p_f(v)=q_2^t$. Then clearly $v^t$ is also a partial isometry in $M$ and $p_i(v^t)=q_2$ and $p_f(v_t)=q_1^t$. Therefore $v\perp v^t.$
If we define $u=v-v^t$, then $u\in M_a$, it is a partial isometry and $p_i(u)=p$ (by Lemma~\ref{L:C2 construction of tripotent}).

$(ii)$ It follows from $(i)$ that such $u$ exists if and only if there are two nonzero orthogonal projections $q_1,q_2\le p$ with $q_1\sim q_2$. If $p$ is abelian, such projections cannot exist. Conversely, assume $p$ is not abelian, i.e., the von Neumann algebra $pMp$ is not abelian. It follows that there is a projection $q\in pMp$ which is not central. Denote by $z$ its central carrier (in $pMp$). Then $z-q\ne 0$, hence $z-q$ and $q$ are not centrally orthogonal and  \cite[Lemma V.1.7]{Tak} shows that there are nonzero projections $q_1\le q$ and $q_2\le z-q$ with $q_1\sim q_2$.
\end{proof}

The next proposition characterizes unitary and complete tripotents in $M_a$.

\begin{prop}\label{P:complete = abelian}
 Let $u\in M_a$ be a tripotent.
 \begin{enumerate}[$(a)$]
     \item $u$ is unitary $\Longleftrightarrow$ $p_i(u)=1$ $\Longleftrightarrow$ $p_f(u)=1$.
     \item $u$ is complete $\Longleftrightarrow$ $1-p_i(u)$ is abelian$\Longleftrightarrow$ $1-p_f(u)$ is abelian.
     \item Let $w\in M_a$ be a tripotent such that $u\le_2 w$. Then $u$ is complete in $(M_a)_2(w)$ if and only if $p_i(w)-p_i(u)$ is abelian.
 \end{enumerate}
\end{prop}

\begin{proof}
$(a)$ Assume $p_i(u)=1$. By
Observation~\ref{obs:pf=alpha(pi)} we get $p_f(u)=1$, hence $u$ is unitary even in $M$.

Conversely, assume that $u$ is unitary. Then, in particular, $(M_a)_1(u)=\{0\}$. In the same way as in the proof of Proposition~\ref{P:H(W,alpha) direct summand} we deduce that  $p_i(u)=p_f(u)$ and it is a central projection. Hence $p_i(u)=\chi_E\otimes I$ for a measurable set $E\subset L$ (see Lemma~\ref{L:commutant of AotimesB(H)}).
Finally, we have also $(M_a)_0(u)=0$, hence 
$$\{0\}=(\chi_{L\setminus E)}\otimes I)(M_a)=L^\infty(\mu|_{L\setminus E})\overline{\otimes}B(H)_a,$$
hence $\mu(L\setminus E)$ has zero measure (as $B(H)_a\ne\{0\}$ whenever $\dim H\ge 2$), i.e., $p_i(u)=1$.

 $(b)$ The second equivalence follows from Observation~\ref{obs:pf=alpha(pi)}. Let us prove the first one.
 
 If $1-p_i(u)$ is not abelian, by Proposition~\ref{P:C2 projections}$(ii)$ there is a nonzero tripotent $v\in M_a$ with $p_i(v)\le 1-p_i(u)$. Having in mind that $u^t= - u$, $v^t=- v$, we deduce that $p_f(v) = p_i(v)^t$, $p_f(u) = p_i(u)^t$, and thus $v\perp u$, hence $u$ is not complete. Conversely, assume $u$ is not complete. It means there is a nonzero tripotent $v$ with $v
\perp u$. It follows that $(1-p_f(u))v(1-p_i(u))=v$, thus $p_i(v)\le 1-p_i(u)$. By Proposition~\ref{P:C2 projections}$(ii)$  $1-p_i(u)$ is not abelian.

$(c)$ The proof is completely analogous to that of the first equivalence in $(b)$.\end{proof}

The next lemma provides a decomposition of an antisymmetric tripotent into basic building blocks.

\begin{lemma}\label{L:C2 decomposition of tripoten}
Let $u\in A\overline{\otimes}B(H)_a$ be a tripotent. Then
$u=\sum_\theta u_\theta$, where $(u_\theta)$ is an orthogonal family of tripotents in $A\overline{\otimes} B(H)_a$ such that each $p_i(u_\theta)$ is the sum of a pair of mutually orthogonal and equivalent abelian projections.
\end{lemma}

\begin{proof}
 Let $p=p_i(u)$. Let $q$ and $v$ be the projection and the partial isometry provided by Lemma~\ref{L:C2 construction of tripotent}. Since $A\overline{\otimes}B(H)$ is of type I, $q$ can be expressed as the sum of an orthogonal family of abelian projections $(q_\theta)$. Let $v_\theta=vq_\theta$. Then $v_\theta$ is a partial isometry with $p_i(v_\theta)=q_\theta$, hence
$r_\theta=p_f(v_\theta)^t$ is a projection equivalent to $q_\theta$ (cf. Lemma~\ref{L:p sim alpha(p)}).
Moreover, $v=\sum_{\alpha} v_{\alpha}$, the family $(v_{\alpha})_\theta$ is  orthogonal and $r_\theta\le p_f(v)^t=p-q$. In particular, $r_\theta\perp q$ and, moreover, the family $(r_\theta)$ is orthogonal. 
Set $p_\theta=q_\theta+r_\theta$. Then $p=\sum_\theta p_\theta$ and the family $(p_\theta)$ is orthogonal. Set $u_\theta=up_\theta$.
Then $u_\theta$ is a partial isometry with  $p_i(u_\theta)=p_\theta$, $u_{\alpha} = v_{\alpha}-v_{\alpha}^t$, $u_{\alpha}\perp u_{\beta}$ for $\alpha\neq \beta$.
It remains to observe that $u_\theta\in  A\overline{\otimes}B(H)_a$, but this is clear from the identity $u_\theta = v_\theta- v_\theta^t$.
\end{proof}

Now we are going to distinguish several cases depending on the dimension of $H$. The first proposition settles the finite-dimensional case. Note that the first statement in assertion $(b)$ follows also from \cite[Proposition 2]{HoMarPeRu2002} where an explicit formula for a unitary element is given.
Assertion $(c)$ is also known -- in this case it is easy to find a complete non-unitary tripotent in $B(H)_a$ and then use Proposition~\ref{P:Linftyfindim is finite}.

\begin{prop}\label{P:C2 finite dim}
Assume that $\dim H$ is finite. Then the following holds:
\begin{enumerate}[$(a)$]
    \item $A\overline{\otimes}B(H)_a$ is triple-isomorphic to $L^\infty(\mu, B(H)_a)$. In particular, it is a finite JBW$^*$-triple; 
    \item If $\dim H$ is even, then $A\overline{\otimes}B(H)_a$ has a unitary element. It is therefore  triple-isomorphic to a (finite) JBW$^*$-algebra. Hence, the relations $\le_0$ and $\le_2$  coincide;
    \item If $\dim H$ is odd, then $A\overline{\otimes}B(H)_a$ has no unitary element. 
\end{enumerate}
\end{prop}

\begin{proof}
$(a)$ The first statement follows from
Lemma~\ref{L:tensor=Bochner}.
The second one then follows from Proposition~\ref{P:Linftyfindim is finite}.

$(b)$ If $\dim H$ is even, then there is a projection $p\in B(H)$ with $p\sim 1-p$. Thus $1\otimes p$ is a projection in $A\overline{\otimes}B(H)$ equivalent to $1-(1\otimes p)=1\otimes(1-p)$. Hence, by Proposition~\ref{P:C2 projections}$(i)$ there is a unitary element in $A\overline{\otimes}B(H)_a$.
The rest follows from $(a)$ and Proposition~\ref{P:le0=le2}.

$(c)$ Assume $\dim H$ is odd and $\uu\in L^\infty(\mu,B(H)_a)$ is a tripotent. Then $p_i(\uu)(\omega)=p_i(\uu(\omega))$ $\mu$-a.e. It follows from Proposition~\ref{P:C2 projections}$(i)$ that $\dim\ran(p_i(\uu(\omega)))$ is even for each $\omega$, thus $p_i(\uu)\ne 1$. It follows that $\uu$ is not unitary (cf. Proposition~\ref{P:complete = abelian}$(a)$).
\end{proof}

The next proposition deals with the infinite-dimensional case. Note that assertion $(a)$ follows also from \cite[Proposition 2]{HoMarPeRu2002}.

\begin{prop}\label{P:C2 infinite dim}
Assume that $\dim H$ is infinite. Then the following holds.
\begin{enumerate}[$(a)$]
    \item $A\overline{\otimes}B(H)_a$ admits a unitary element, i.e., it is triple-isomorphic to a JBW$^*$-algebra.
    \item There are complete non-unitary elements in  $A\overline{\otimes}B(H)_a$. In particular, the triple $A\overline{\otimes}B(H)_a$ is not finite. 
  
\end{enumerate}
\end{prop}

\begin{proof}
$(a)$ If $\dim H=\infty$, then there is a projection $p\in B(H)$ with $p\sim 1-p$. The rest of the argument is the same as in Proposition~\ref{P:C2 finite dim}$(b)$. 

$(b)$  Since $H$ is infinite-dimensional, there are two mutually orthogonal projections $q_1,q_2\in B(H)$ such that $q_1\sim q_2$  and $1-q_1-q_2$ is a one-dimensional projection. 
Then $1\otimes q_1$ and $1\otimes q_2$ are two equivalent mutually orthogonal projections in $A\overline{\otimes}B(H)$. By Proposition~\ref{P:C2 projections}$(i)$ there is a tripotent $U\in A\overline{\otimes}B(H)_a$ with $p_i(U)=1\otimes q_1+1\otimes q_2=1\otimes(q_1+q_2)$. Then $p_i(U)\ne 1$ and $1-p_i(U)=1\otimes(1-q_1-q_2)$ is abelian by Lemma~\ref{L:abelian pr in AotimesB(H)}. Hence $U$ is complete but not unitary by Proposition~\ref{P:complete = abelian}. \end{proof}

Hence, summarizing, $A\overline{\otimes}B(H)_a$ is a finite JBW$^*$-algebra if $H$ has  finite even dimension; it is a finite JBW$^*$-triple with no unitary element if $H$ has  finite odd dimension and it is an infinite JBW$^*$-algebra if $H$ has infinite dimension. The case of finite even dimension is in this way completely settled. In the remaining two cases we are going further to characterize the relation $\le_0$ and in the infinite-dimensional case moreover analyze finite and properly infinite tripotents.

We start by characterizing $\le_0$ in $B(H)_a$.

\begin{lemma}
Let $u,v\in B(H)_a$ be two tripotents. 
Then the following assertions are equivalent.
\begin{enumerate}[$(i)$]
    \item $u\le_0 v$;
    \item Either $v$ is complete or $u\le_2 v$;
    \item Either $1-p_i(v)$ is a rank-one operator, or $p_i(u)\le p_i(v)$.
\end{enumerate}
\end{lemma}

\begin{proof} The implication $(iii)\Rightarrow(ii)$ follows from Proposition~\ref{P:complete = abelian} and Proposition~\ref{P:H(W,alpha) abstract le2}.

The implication $(ii)\Rightarrow(i)$ is obvious.

$(i)\Rightarrow(iii)$ Assume $(iii)$ does not hold, i.e., the projection $1-p_i(v)$ has rank at least two and $p_i(u)\not\le p_i(v)$. The latter inequality implies $1-p_i(v)\not\le 1-p_i(u)$. Hence, there is a rank-one projection $r_1$ with $r_1\le 1-p_i(v)$ but $r_1\not\le1-p_i(u)$. Since $1-p_i(v)$ has rank at least two, there is a rank-one projection $r_2\le1-p_i(v)$ orthogonal to $r_1$. Since $r_1\sim r_2$, by Proposition~\ref{P:C2 projections}$(i)$ there is a tripotent $w\in B(H)_a$ with $p_i(w)=r_1+r_2$. Then clearly $w\in B(H)_0(v)\setminus B(H)_0(u)$.
\end{proof}

Using Proposition~\ref{P:Linfty refelexive} the following proposition easily follows.

\begin{prop}\label{P:le0 in AotimesB(H)aodd}
Assume that the Hilbert space $H$ has finite odd dimension. Let $\uu,\vv\in L^\infty(\mu,B(H)_a)=A\overline{\otimes}B(H)_a$ be two tripotents.
Then the following assertions are equivalent.
\begin{enumerate}[$(i)$]
    \item $\uu\le_0 \vv$;
    \item There is a measurable set $E\subset L$ such that $\vv(\omega)$ is complete for $\omega\in E$ and  $\uu(\omega)\le_2 \vv(\omega)$ for $\omega\in L\setminus E$;
    \item There is a measurable set $E\subset L$ such that $1-p_i(\vv)(\omega)$ is a rank-one operator for $\omega\in L$ 
    and $p_i(\uu)(\omega)\le p_i(\vv)(\omega)$ for $\omega\in L\setminus E$.
\end{enumerate}
\end{prop}

Note that the proposition holds also if $H$ has a finite even dimension, but in this case we get $E=\emptyset$ in $(ii)$ and $(iii)$ (cf. Proposition \ref{P:C2 finite dim}$(b)$). For infinite-dimensional $H$ an analogue holds as well, but it requires a different proof. It is contained in the following proposition.

\begin{prop}\label{P:le0 in AotimesB(H)ainfty}
Assume $\dim H=\infty$. 
Let $u,v\in A\overline{\otimes}B(H)_a=L^\infty(\mu)\overline{\otimes}B(H)_a$ be two tripotents. Then the following assertions are equivalent.
\begin{enumerate}[$(i)$]
    \item $u\le_0 v$;
    \item There is a measurable set $E\subset L$  such that $(\chi_E\otimes I)u\le_2 (\chi_E\otimes I)v$ and $(\chi_{L\setminus E}\otimes I)v$ is complete in $L^\infty(\mu|_{L\setminus E})\overline{\otimes}B(H)_a$;
    \item There is a measurable set $E\subset L$  such that  $(\chi_E\otimes I)p_i(u)\le (\chi_E\otimes I)p_i(v)$ and $(\chi_{L\setminus E}\otimes I)(1-p_i(v))$ is abelian.
\end{enumerate}
\end{prop}

\begin{proof}
The implication $(iii)\Rightarrow(ii)$ follows from Proposition~\ref{P:complete = abelian} and Proposition~\ref{P:H(W,alpha) abstract le2}.

The implication $(ii)\Rightarrow(i)$ is obvious (using Proposition~\ref{P:preorders in direct sum}).

$(i)\Rightarrow(iii)$ Let $M=A\overline{\otimes} B(H)$. Consider the family
$$\{z\in M\setsep z\mbox{ is a central projection and }zp_i(u)\le zp_i(v)\}.$$
This family is nonempty (it contains the zero projection) and is closed under taking suprema of nonempty subsets. So, it has a maximum, call it $z$. Then $zp_i(u)\le zp_i(v)$. By Lemma~\ref{L:commutant of AotimesB(H)} we know that $z=\chi_E\otimes I$ for a measurable set $E\subset L$.

If $E=L$ (i.e., $z=1$), the proof is completed. So, assume that $E\subsetneqq L$ ($z<1$) and that $(\chi_{L\setminus E}\otimes I)(1-p_i(v))=(1-z)(1-p_i(v))$ is not abelian. We may work in $(1-z)M=L^\infty(\mu|_{L\setminus E})\overline{\otimes}B(H)$, so without loss of generality assume $z=0$.

Consider the family
$$\{w\in M\setsep w\mbox{ is a central projection and }w(1-p_i(v))\mbox{ is abelian}\}.$$
This family is nonempty (it contains the zero projection) and is closed under taking suprema of nonempty subsets. So, it has a maximum, call it $w$. If $w=1$, the proof is complete. If $w<1$, we may work in $(1-w)M$, so, similarly as above we may assume without loss of generality that $w=0$, and thus, for every central projection $z$ in $M$, $z (1-p_i(v))$ is either zero or not abelian.

Since $p_i(u)\not\le p_i(v)$, we get $1-p_i(v)\not\le 1-p_i(u)$, so the projection
$$r_0=(1-p_i(v))-(1-p_i(v))\wedge(1-p_i(u))$$ 
is nonzero.
Fix a nonzero abelian projection $r\le r_0$. Since $1-p_i(v)$ is not abelian, the projection
$$s=1-p_i(v)-r$$
is nonzero.

 We claim that $s$ and $r$ are not centrally orthogonal. Let us prove it by contradiction. Assume that there is a central projection $z\in M$ with $r\le z$ and $s\le 1-z$. Since $r+s=1-p_i(v)$, we get $r=z(1-p_i(v))$. Since $r$ is non-zero and abelian, necessarily $z=0$, a contradiction. 

Now we can use \cite[Lemma V.1.7]{Tak} to get projections $r_1\le r$ and $s_1\le s$ with $r_1\sim s_1$. Then there is a tripotent $w\in M_a$ with $p_i(w)=r_1+s_1$. It follows that $w\in M_0(v)\setminus M_0(u)$.
\end{proof}

We continue by a characterization of finite, infinite and properly infinite tripotents in $A\overline{\otimes}B(H)_a$.

\begin{prop}\label{finite-finite}
Assume that $\dim H=\infty$. Let $u\in A\overline{\otimes}B(H)_a$ be a tripotent.
\begin{enumerate}[$(i)$]
    \item $u$ is finite if and only if $p_i(u)$ is a finite projection in $A\overline{\otimes}B(H)$;
     \item $u$ is properly infinite if and only if $p_i(u)$ is a properly infinite projection in $A\overline{\otimes}B(H)$.
\end{enumerate}
\end{prop}

\begin{proof}
Assume first that $p_i(u)$ is a properly infinite projection. Fix $\f\in\ran p_i(u)$. Let $q$ denote the projection onto $\widehat{X}_{\f}$. Then $q\le p_i(u)$ (by Observation~\ref{obs:invariance Xf}) and $q$ is abelian by Lemma~\ref{L:abelian pr in AotimesB(H)}. Hence $p_i(u)-q$ is also properly infinite. By \cite[Proposition V.1.36]{Tak} or \cite[Lemma 6.3.3]{KR2} there are two orthogonal equivalent projections $r_1,r_2$ with $p_i(u)-q=r_1+r_2$. Hence, by Proposition~\ref{P:C2 projections} there is a tripotent $v\in A\overline{\otimes}B(H)_a$ with $p_i(v)=p_i(u)-q$. Then $v\le_2 u$, $v$ is complete and not unitary in  $\left( A\overline{\otimes}B(H)_a\right)_2(u)$ (by Proposition~\ref{P:complete = abelian}). Thus $u$ is infinite.

Further, let $E\subset L$ be measurable such that $(\chi_E\otimes I)u\ne 0$. Then $p_i((\chi_E\otimes I)u)=(\chi_E\otimes I)p_i(u)$ which is a properly infinite projection. The previous paragraph shows that in this case $(\chi_E\otimes I)u$ is infinite.

Thus, $u$ is properly infinite and the proof of the `if part' of $(ii)$ is completed.

\smallskip

Next assume that $p_i(u)$ is infinite. By \cite[Proposition 6.3.7]{KR2} there is a properly infinite projection $q\le p_i(u)$.  By \cite[Lemma 6.3.3]{KR2} there are two orthogonal equivalent projections $q_1,q_2$ with $q=q_1+q_2$. Hence, by Proposition~\ref{P:C2 projections} there is a tripotent $v\in A\overline{\otimes}B(H)_a$ with $p_i(v)=q$. By the already proved part we know that $v$ is properly infinite. Since $v\le_2 u$, $u$ is infinite by Lemma~\ref{L:finite tripotents}. This completes the proof of the `only if part' of assertion $(i)$.

\smallskip

Assume now that $p_i(u)$ is finite. Let $v\in M_2(u)$ be complete. Then $p_i(v)\le p_i(u)$, hence $p_i(v)$ is also finite and, moreover, $p_i(u)-p_i(v)$ is abelian (cf. Proposition~\ref{P:complete = abelian}).

By Proposition~\ref{P:C2 projections} we have $p_i(u)=q_1+q_2$ for a pair of orthogonal equivalent projections $q_1$ and $q_2$. Apply Lemma~\ref{L:finite in AotimesB(H)}$(i)$ to $q_1$ and $q_2$.  Since $q_1\sim q_2$, their dimension functions from condition $(c)$ are equal. Since the dimension function of $p_i(u)$ is the sum of the dimension functions of $q_1$ and $q_2$, it attains even values. By the same argument the dimension function of $p_i(v)$ attains even values.
Since $p_i(u)-p_i(v)$ is abelian, it follows from Lemma~\ref{L:abelian pr in AotimesB(H)} that its dimension function (which equals the difference of dimension functions of $p_i(u)$ and $p_i(v)$) is bounded by $1$. Since it must attain even values, 
it is necessarily zero. But this means that $p_i(v)=p_i(u)$, i.e., $v$ is unitary in $M_2(u)$.
So, $u$ is finite. This completes the proof of the `if part' of $(i)$.

\smallskip

Finally, assume that $u$ is properly infinite. Let $z$ be a central projection such that $zp_i(u)\ne 0$. Since $zp_i(u)=p_i(zu)$, we have $zu\ne0$. Thus $zu$ is infinite, so by the already proved assertion $(i)$ we deduce that $zp_i(u)$ is infinite. Hence $p_i(u)$ is properly infinite and the proof is complete. 
\end{proof}

The next corollary follows from Proposition~\ref{finite-finite} using \cite[Proposition 6.3.7]{KR2} and Lemma  \ref{L:commutant of AotimesB(H)}.

\begin{cor}
Let $u\in A\overline{\otimes} B(H)_a$ be a tripotent. Then there is a measurable set $E\subset L$ such that $(\chi_E\otimes I)u$ is finite and $(\chi_{L\setminus E}\otimes I)u$ is either zero or properly infinite. 
\end{cor}

In the following proposition we establish some further properties of properly infinite tripotents. Assertion $(a)$ is an analogue of Proposition~\ref{P:vN properly infinite trip}$(iv)$ and Proposition~\ref{P:pV decomp}$(ii)$. Assertion $(b)$ says that the analogue of Proposition~\ref{P:vN halving} and Proposition~\ref{P:pV halving} fails in this case. Recall that $M$ denotes the von Neumann algebra $A\overline{\otimes}B(H)$.

\begin{prop}\label{P:C2 properly infinite}
Let  $u\in A\overline{\otimes} B(H)_a$ be a properly infinite tripotent. 
\begin{enumerate}[$(a)$]
    \item Any complete tripotent in $(M_a)_2(u)$ is again properly infinite.
    \item Assume that $v\in (M_a)_2(u)$ is complete and $w\in (M_a)_2(u)\cap (M_a)_1(v)$. Then $w$ is abelian (and hence finite).
\end{enumerate}
\end{prop}

\begin{proof}
$(a)$ Since $u$ is properly infinite, $p_i(u)$ is also properly infinite by Proposition~\ref{finite-finite}.
Assume that $v\in (M_a)_2(u)$ is complete. 
Then $p_i(v)\le p_i(u)$ and $p_i(u)-p_i(v)$ is abelian (by Proposition~\ref{P:complete = abelian}$(c)$). Thus $p_i(v)$ is properly infinite and so is $v$ (by Proposition~\ref{finite-finite}).

$(b)$ We have $w=w_1+w_2$ where
$$w_1=p_f(v)w(p_i(u)-p_i(v))\mbox{ and }w_2=(p_f(u)-p_f(v))wp_i(v).$$
Note that $w_1^t=-w_2$ and $w_1,w_2$ are orthogonal tripotents in $A\overline{\otimes}B(H)$. Moreover, $p_i(w_1)$ is abelian since $p_i(w_1)\le p_i(u)-p_i(v)$ and $p_i(u)-p_i(v)$ is abelian. Further,
$$p_i(w_2)=p_i(-w_1^t)=p_f(w_1)^t\sim p_f(w_1)\sim p_i(w_1),$$  hence
$p_i(w_2)$ is also abelian. So there are $\f_1,\f_2\in L^2(\mu,H)$ such that  $\ran(p_i(w_j))=\widehat{X}_{\f_j}$ for $j=1,2$ (cf. Lemma \ref{L:abelian pr in AotimesB(H)}). Since $w_1\perp w_2$, we have $[\f_1,\f_2]=0$. The equality $w_1^t=-w_2$ shows that $\ran(p_f(w_1))=\widehat{X}_{\overline{\f_2}}$ 
and $\ran(p_f(w_2))=\widehat{X}_{\overline{\f_1}}$ (see Observation \ref{obs:range pt}). Hence we may and shall assume that $\overline{\f_2}=w_1\f_1$, in particular
$\norm{\f_1(\omega)}=\norm{\f_2(\omega)}$ $\mu$-a.e.

It follows that $(M_a)_2(w)$ consists of those operators from $A\overline{\otimes}B(H)_a$ which map $\widehat{X}_{\f_1}+\widehat{X}_{\f_2}$ into 
$\widehat{X}_{\overline{\f_1}}+\widehat{X}_{\overline{\f_2}}$ and are zero on the orthogonal complement.
 Let $T$ be any such operator. By Lemmata \ref{L:AotimesB(H) charact} and \ref{L:type2 perp} we get that for any measurable set $E\subset L$ we have 
$$\ip{\chi_E T(\f_1)}{\chi_E\overline{\f_1}}=
\ip{T(\chi_E \f_1)}{\overline{\chi_E\f_1}}=0,$$
thus $[T(\f_1),\overline{\f_1}]=0$. It follows that there is a measurable function $h_1:L\to\ce$ such that $T(\f_1)=h_1\overline{\f_2}$. Hence, again by Lemma \ref{L:AotimesB(H) charact}, we get  
$$T(g\f_1)=h_1g\overline{\f_2}, \qquad g\f_1\in \widehat{X}_{\f_1}.$$
Since $T$ is a bounded operator, necessarily $h_1\in L^\infty(\mu)$.
Similarly we deduce that there is a  function $h_2\in L^\infty(\mu)$ such that
$$T(g\f_2)=h_2g\overline{\f_1}, \qquad g\f_2\in \widehat{X}_{\f_2}.$$
Finally, given any measurable set $E\subset L$ we get
$$\begin{aligned}
\int_E \norm{\f_2(\omega)}^2h_1\di\mu&=\ip{h_1\chi_E\overline{\f_2}}{\chi_E\overline{\f_2}}
=\ip{T(\chi_E\f_1)}{\chi_E\overline{\f_2}}
\\&=-\ip{T^t(\chi_E\f_1)}{\chi_E\overline{\f_2}}
=-\ip{\overline{T^*(\chi_E\overline{\f_1})}}{\chi_E\overline{\f_2}}
\\&=-\ip{\chi_E\f_2}{T^*(\chi_E\overline{\f_1})}
=-\ip{T(\chi_E\f_2)}{\chi_E\overline{\f_1}}
\\&=-\ip{h_2\chi_E\overline{\f_1}}{\chi_E\overline{\f_1}}=-\int_E \norm{\f_1(\omega)}^2h_2\di\mu
\end{aligned}
$$
Since  $\norm{\f_1(\omega)}=\norm{\f_2(\omega)}$ $\mu$-a.e., we deduce that $h_2=-h_1$ $\mu$-a.e. on the set $\{\omega\in L\setsep \f_1(\omega)\ne0\}$.

It now easily follows that $(M_a)_2(w)$ is canonically isomorphic to the abelian von Neumann algebra $L^\infty(\mu|_{\{\omega\in L\setsep \f_1(\omega)\ne0\}})$, hence $w$ is abelian.
\end{proof}

\section{Spin factors and exceptional JBW$^*$-triples}\label{sec:6 - spin and exceptional}

In this section we will deal with the summands of the form $A\overline{\otimes}C$, where $A$ is an abelian von Neumann algebra and $C$ is a Cartan factor of type $4$, $5$ or $6$. This section is divided into four subsections. Subsection~\ref{subsec:spin} is devoted to Cartan factors of type $4$, usually called spin factors. In Subsection~\ref{sec:CD} we explain the relationship of certain finite-dimensional spin factors to the Cayley-Dickson doubling process, in particular to the complex octonion algebra. Subsections~\ref{subsec:C6} and~\ref{subsec:C5} are then devoted to Cartan factors of type $6$ and $5$ which are defined using the octonion algebra.

\subsection{Spin factors}\label{subsec:spin}

This subsection is devoted to the summands of the form $A\overline{\otimes}C$ where $C$ is a type $4$ Cartan factor, i.e. a spin factor. Let us start by focusing on the structure of spin factors itself.

Throughout this subsection $C$ will denote the Hilbert space
 $\ell^2(\Gamma)$ for a set $\Gamma$, equipped with the canonical (coordinatewise) conjugation, and with the triple product and the norm  defined by
$$\begin{aligned}
\J xyz&=\ip xy z+ \ip zy x - \ip{x}{\overline{z}}\overline{y}\\
\norm{x}^2&=\ip xx +\sqrt{\ip xx^2-\abs{\ip{x}{\overline{x}}}^2}.
\end{aligned}
$$
Then $C$ is indeed a JB$^*$-triple.
(This follows, for example, from the explanation given in \cite[pp. 16-17 and pp. 19-20]{harris1974bounded}.)

There are two canonical norms on $C$ -- the hilbertian norm $\norm{x}_2=\sqrt{\ip xx}$ and the above-defined norm $\norm{\cdot}$.
These two norms are equivalent as clearly
$$\norm{x}_2\le\norm{x}\le\sqrt2\norm{x}_2\mbox{ for }x\in C.$$
It follows, in particular, that $C$ is a JBW$^*$-triple (being isomorphic to a Hilbert space and hence reflexive).

Moreover, we have on $C$ two notions of orthogonality -- one defined by the inner product and the second one coming from the structure of a JB$^*$-triple. They do not coincide, so to distinguish them we will use $\perp_2$ to denote the (Euclidean) orthogonality with respect to the inner product (i.e., $x\perp_2 y$ if and only if $\ip xy=0$) and $\perp$ to denote the triple orthogonality (i.e., $x\perp y$ if and only if $L(x,y)=0$).

 For a non-zero $x\in C$ the equality $\norm{x}=\norm{x}_2$ holds if and only if $\ip xx=\abs{\ip{x}{\overline{x}}}$, which in a Hilbert space is equivalent to say that $\overline{x}$ is a scalar mutliple of $x$. Moreover, since we are working with the canonical coordinatewise conjugation, this takes place if and only if $x$ is a scalar multiple of a vector with real-valued coordinates.  Indeed the `if part' is obvious, to show the `only if' part assume that $\overline{x}=cx$ and $x=(\alpha_\gamma r_\gamma)_{\gamma\in\Gamma}$, where $r_\gamma\ge0$ and $\alpha_\gamma$ is a complex unit for each $\gamma\in\Gamma$. Then $\overline{x}=(\overline{\alpha_\gamma}r_\gamma)_{\gamma\in\Gamma}$.
If $\overline{x}=cx$ for some $c\in\mathbb{T}$, then $\overline{\alpha_\gamma}r_\gamma=c\alpha_\gamma r_\gamma$ for $\gamma\in\Gamma$. Hence $\alpha_\gamma^2=\frac1c$ whenever $r_\gamma>0$. It follows that $\alpha_\gamma=\pm\frac1{\sqrt{c}}$ (whenever $r_\gamma>0$).  Hence, $x = \frac1{\sqrt{c}} ( \sqrt{c} \ \alpha_{\gamma} r_{\gamma})$.

In particular, if $\dim C=1$, then $\norm{\cdot}=\norm{\cdot}_2$.

On the other hand, $\norm{x}=\sqrt2\norm{x}_2$ if and only if $x\perp_2 \overline{x}$. This may happen whenever $\dim C\ge 2$ (an example of such a vector in the two-dimensional space is $(1,i)$).

 It is usually assumed that $\dim C\ge 3$. The reason is that if $\dim C=1$, then $C$ is just the complex field, hence a Hilbert space; and if $\dim C=2$, then $C$ is not a factor, it is actually triple-isomorphic to the abelian C$^*$-algebra $\ce\oplus^{\infty} \ce$ (see, for example \cite{loos1977bounded,kaup1997real} or the recent reference \cite[\S 3]{KP2019}, see also Lemma \ref{l 6.7 two three and four dimensional spins}$(iv)$ below).
 The next results are part of the folklore in JB$^*$-triple theory.

\begin{lemma}\label{L:tripotents in spin factor}
Nonzero tripotents in $C$ are exactly the elements of one of the following forms.
\begin{enumerate}[$(a)$]
    \item $u=\alpha z$, where $z\in C$ is an element with real-valued coordinates such that $\norm{z}_2=1$ and $\alpha$ is a complex unit. In this case $u$ is unitary.
    \item $u\perp_2 \overline{u}$ and $\norm{u}_2=\frac1{\sqrt{2}}$. In this case $u$ is minimal, 
    $C_2(u)=\span \{u\}$, $C_0(u)=\span\{\overline{u}\},$  $C_1(u)=\{u,\overline u\}^{\perp_2}$ and the Peirce projections are the respective orthogonal projections.
\end{enumerate}
\end{lemma}

\begin{proof} The statement of the lemma is essentially well-known and easy to see (see, for example,  \cite[\S 3]{KP2019}).
Let us just briefly indicate the argument.
Assume $u\in C$ is a nonzero tripotent. 
Then
$$u=\J uuu =2\ip uu u - \ip{u}{\overline{u}}\overline{u},$$
hence either $\overline{u}$ is a scalar multiple of $u$ or $\ip{u}{\overline{u}}=0$.
It is straightforward to show that these two cases correspond to the cases $(a)$ and $(b)$.
\end{proof} 
 
 Now we get easily the following result.
 
\begin{prop}\label{P:spin is finite}
Any spin factor is (triple-isomorphic to) a finite JBW$^*$-algebra. 
In particular, the relations $\le_0$ and $\le_2$ coincide.
\end{prop}

\begin{proof}
By Lemma~\ref{L:tripotents in spin factor} there are unitary tripotents and, moreover, tripotents which are not unitary are zero or minimal and hence not complete either. This shows that every spin factor is finite. The coincidence of the two relations follows from Proposition~\ref{P:le0=le2}.
\end{proof} 
 
In the following proposition we characterize the relations $\le$, $\le_2$, $\le_0$ and $\perp$ in the spin factor. The proof follows immeadiately from Lemma~\ref{L:tripotents in spin factor} and Proposition~\ref{P:spin is finite}.
 
\begin{prop}\label{P:relations in spin} Let $u,e\in C$ be two nonzero tripotents.
\begin{enumerate}[$(1)$]
    \item $u\le e$ if and only if one of the following assertions holds
    \begin{itemize}
        \item $u=e$;
        \item $e$ is unitary, $u$ is minimal and $e=u+
        \alpha\overline{u}$ for a complex unit $\alpha$.
    \end{itemize}
    \item $u\le_0 e$ $\Leftrightarrow$ $u\le_2 e$ $\Leftrightarrow$ either $e$ is unitary or $u=\alpha e$ for a complex unit $\alpha$.
    \item $u\perp e$ if and only if $u$ and $e$ are minimal and $e=\alpha\overline{u}$ for a complex unit $\alpha$.
\end{enumerate}
\end{prop}

 Since any spin factor is reflexive (being isomorphic to a Hilbert space), Proposition~\ref{P:Linfty refelexive} may be applied to describe the relations $\le$, $\le_2$, $\le_0$ and $\perp$ in the triple $L^\infty(\mu,C)=L^\infty(\mu)\overline{\otimes}C$. As a consequence we get the following result.

\begin{cor}\label{cor:Aotimes spin is finite}
Let $C$ be a spin factor and let $A$ be an abelian von Neumann algebra. Then $A\overline{\otimes}C$ is a finite JBW$^*$-algebra.
In particular, the relations $\le_2$ and $\le_0$ coincide.
\end{cor}

\subsection{Spin factors and Cayley-Dickson doubling process}\label{sec:CD} 

There is a closed relation between finite-dimensional spin factors and the Cayley-Dickson doubling process. This is an abstract construction described for example in \cite[\S6.1.30]{Cabrera-Rodriguez-vol2}.
We are going to describe it in a special case, starting from the complex field. It enables us to equip certain finite-dimensional spin factors with some additional structure which is then used to define exceptional Cartan factors.

Let $\A_0=\ce$ and $\A_n=\A_{n-1}\times \A_{n-1}$ for $n\in\en,$
 i.e., for $n\in\en\cup\{0\}$ we have $\A_n=\ce^{2^n}$. This space is equipped with the standard conjugation, standard inner product and with the triple product making it the spin factor of dimension $2^n$. Except for this structure we equip $\A_n$ by two involutions -- a linear one denoted by $\inv_n$ and a conjugate-linear one denoted by $*_n$ and, finally, by a product denoted by $\boxdot_n$.
 (Later we will omit the subscript $n$.)  These operations are defined by induction.
 
$\A_0$ is the complex field, hence $\boxdot_0$ will be the standard multiplication of complex numbers,  the linear involution $\inv_0$ will be the identity mapping and tho conjugate-linear involution $*_0$ will be the complex conjugation.

Given the structure defined on $\A_n$ we define it on $\A_{n+1}=\A_n\times \A_n$ as follows
$$\begin{aligned}
(x_1,x_2)^{\inv_{n+1}}&=(x_1^{\inv_n},-x_2),\\
(x_1,x_2)^{*_{n+1}}&=(x_1^{*_n},-\overline{x_2},)\\
(x_1,x_2)\boxdot_{n+1}(y_1,y_2)&=(x_1\boxdot_n y_1-y_2\boxdot_n x_2^{\inv_n},x_1^{\inv_n} \boxdot_n y_2+y_1\boxdot_n x_2)
\end{aligned}$$
for $(x_1,x_2), (y_1,y_2)\in \A_n\times \A_n=\A_{n+1}$.

In the sequel we will omit the subscript $n$ at the two involutions.
The following lemma summarizes basic properties of these operations.

\begin{lemma}\label{L:CD doubling}
\begin{enumerate}[$(i)$]
    \item The mapping $x\mapsto x^\inv$ is a linear mapping on $\A_n$ such that $(x^\inv)^\inv=x$ for $x\in \A_n$;
    \item The mapping $x\mapsto x^*$ is a conjugate linear mapping on $\A_n$ such that $(x^*)^*=x$ for $x\in \A_n$;
    \item $x^*=\overline{x^\inv}=\overline{x}^\inv$ and $(x^*)^\inv=(x^\inv)^*=\overline{x}$ for $x\in \A_n$;
    \item $\A_n$ is a (possibly non-associative) algebra, i.e., the mapping $(x,y)\mapsto x\boxdot_n y$ is bilinear;
    \item $\overline{x\boxdot_n y}=\overline{x}\boxdot_n\overline{y}$, $(x\boxdot_n y)^\inv=y^\inv\boxdot_n x^\inv$ and $(x\boxdot_n y)^*=y^*\boxdot_n x^*$ for $x,y\in \A_n$ (hence $\inv$ and $*$ are involutions);
    \item $x\mapsto (x,0)$ is an isomorphic embedding of the algebra $\A_n$ (equipped with the involutions $\inv$ and $*$) into the algebra $\A_{n+1}$ (equipped with the involutions $\inv$ and $*$), so we have canonical inclusions
    $$\ce=\A_0\subset \A_1\subset \A_2\subset \A_3\subset\dots;$$
    \item If $y\in \ce=\A_0$, then $x\boxdot_n y=y\boxdot_n x=yx$ for $x\in \A_n$;
    \item The algebra $\A_n$ is unital, its unit is $1\in\ce=\A_0$;
    \item If $x\in \A_n$, then $x^\inv=x$ if and only if $x\in \ce=\A_0$;
    
    \item $x^\inv\boxdot_n x\in \A_0=\ce$ for $x\in \A_n$;
    \item $\A_0$ is a commutative field, $\A_1$ is a commutative and associative algebra and $\A_2$ is a non-commutative associative algebra;
    \item $\A_3$ is neither commutative nor associative, but it is alternative, i.e., $$x\boxdot_3(x\boxdot_3 y)=(x\boxdot_3 x)\boxdot_3 y \mbox{ and }(y\boxdot_3 x)\boxdot_3 x=y\boxdot_3(x\boxdot_3 x)\mbox{\qquad for }x,y\in \A_3;$$
    \item $x\boxdot_3(y\boxdot_3 z)+z\boxdot_3(y\boxdot_3 x)=(x\boxdot_3 y)\boxdot_3 z+(z\boxdot_3 y)\boxdot_3 x$ for $x,y,z\in \A_3$.
 \end{enumerate}
\end{lemma}

\begin{proof} These properties are well known. Let us briefly comment them for the sake of completeness.

The properties $(i)$--$(ix)$ follow by an easy induction using the definitions and properties of the complex field. 

$(x)$ This follows from $(v)$, $(i)$ and $(ix)$.

$(xi)$ $\A_0$ is a commutative field as it is the complex field. 

The commutativity of $\boxdot_1$ follows from the commutativity of $\boxdot_0$ and the fact that $x^\inv=x$ for $x\in \A_0$. The associativity of $\boxdot_1$ follows from the observation that it can be represented as matrix multiplication if any $x=(x_1,x_2)\in \A_1$ is represented by the matrix $\begin{pmatrix}
x_1 & - x_2\\ x_2 & x_1
\end{pmatrix}.$
Hence, $\A_1$ is a commutative and associative algebra.

Commutativity of $\boxdot_1$ implies the associativity of $\boxdot_2$, as it can be represented by the matrix multiplication
in the algebra of $2\times 2$ matrices with entries in $\A_1$. Indeed, any $x=(x_1,x_2)\in \A_2$ corresponds to the matrix
$\begin{pmatrix}
x_1 & - x_2^\inv\\ x_2 & x_1^\inv
\end{pmatrix}.$
Thus $\A_2$ is an associative algebra. 

$(xii)$ The equality may be proved by a direct computation, using the previous properties. Indeed, assume $x=(x_1,x_2)$ and $y=(y_1,y_2)$. Then
$$\begin{aligned}
x\boxdot_3&(x\boxdot_3 y)=x\boxdot_3 (x_1\boxdot_2 y_1-y_2\boxdot_2 x_2^\inv,x_1^\inv \boxdot_2 y_2+y_1\boxdot_2 x_2)
\\&=(x_1\boxdot_2 x_1\boxdot_2 y_1-x_1\boxdot_2 y_2\boxdot_2 x_2^\inv
-x_1^\inv \boxdot_2 y_2\boxdot_2 x_2^\inv-y_1\boxdot_2 x_2\boxdot x_2^\inv,\\&\qquad x_1^\inv\boxdot_2 x_1^\inv \boxdot_2 y_2+x_1^\inv\boxdot_2 y_1\boxdot_2 x_2 + x_1\boxdot_2 y_1\boxdot_2 x_2-y_2\boxdot_2 x_2^\inv\boxdot_2 x_2)
\end{aligned}$$
and
$$\begin{aligned}
(x\boxdot_3 x)&\boxdot_3 y=(x_1\boxdot_2 x_1-x_2\boxdot_2 x_2^\inv,x_1^\inv\boxdot_2 x_2+x_1\boxdot_2 x_2)\boxdot_3 y
\\&=(x_1\boxdot_2 x_1\boxdot_2 y_1-x_2\boxdot_2 x_2^\inv\boxdot_2 y_1 - y_2\boxdot_2 x_2^\inv\boxdot_2 x_1 - y_2\boxdot_2x_2^\inv\boxdot_2 x_1^\inv,\\
&\qquad x_1^\inv\boxdot_2 x_1^\inv\boxdot_2 y_2-x_2\boxdot_2 x_2^\inv\boxdot_2 y_2+ y_1\boxdot_2 x_1^\inv\boxdot_2 x_2+y_1\boxdot_2 x_1\boxdot_2 x_2),
\end{aligned}$$
where we used the definitions, the associtativity of $\boxdot_2$ and  $(v)$. Let us compare the results. In the first coordinate, the first summands coincide. Moreover, the fourth summand from the first case coincide with the second summand from the second case by $(x)$ and $(vii)$.  Moreover,
$$x_1\boxdot_2 y_2\boxdot_2 x_2^\inv
+x_1^\inv \boxdot_2 y_2\boxdot_2 x_2^\inv 
=(x_1+x_1^\inv)\boxdot_2 y_2\boxdot_2 x_2^\inv\overset{(vi)}{=}
 y_2\boxdot_2 x_2^\inv\boxdot_2(x_1+x_1^\inv),$$
which completes the proof of the equality of the first coordinates.
The equality of the second coordinates is analogous and the second equality is similar.

$(xiii)$ As mentioned in \cite[p. 152]{Cabrera-Rodriguez-vol1}, it is easy to derive from $(xii)$ by linearization that
\begin{gather*}
    x\boxdot_3 (y\boxdot_3 z)+ y\boxdot_3(x\boxdot_3 z)= (x\boxdot_3 y)\boxdot_3 z+(y\boxdot_3 x)\boxdot_3 z,\\
    z\boxdot_3 (y\boxdot_3 x)+ y\boxdot_3(z\boxdot_3 x)= (z\boxdot_3 y)\boxdot_3 x+(y\boxdot_3 z)\boxdot_3 x,\\
    y\boxdot_3(z\boxdot_3 x)+ y\boxdot_3(x\boxdot_3 z)=(y\boxdot_3 z)\boxdot_3 x+(y\boxdot_3 x)\boxdot_3 z.
\end{gather*}
We conclude by adding the first two equalities and subtrackting the third one.
\end{proof}

The next lemma explains the connection of the algebra $\A_n$ to the structure of a spin factor on $\A_n$.

\begin{lemma}\label{L:spin factor as CD}
\begin{enumerate}[$(a)$]
     \item $\ip xy=\ip{x^\inv}{y^\inv}=\frac12 (x\boxdot_n y^*+\overline{y}\boxdot_n x^\inv)$ for $x,y\in \A_n$;
    \item $1$ is a unitary element of $\A_n$. Moreover, the operations in the respective JB$^*$-algebra are given by
    $$x^{*_1}=\J 1x1=x^*\mbox{ and }x\circ_1 y=    \J x1y=\frac12(x\boxdot_n y+y\boxdot_n x)$$
    for $x,y\in\A_n$;
    \item $\ip xy=\frac12(x\circ_1 y^*+x^\inv\circ_1\overline{y})$ for $x,y\in \A_n$; 
    \item $\J xyz= \frac12(x\boxdot_n(y^*\boxdot_n z)+z\boxdot_n(y^*\boxdot_n x))=\frac12((x\boxdot_n y^*)\boxdot_n z+(z\boxdot_n y^*)\boxdot_n x)$ for $x,y,z\in \A_n$ if $n\le 3$;
    \item $\ip{x\boxdot_n z^*}{y}=\ip{x}{y\boxdot_n z}$ and 
    $\ip{z^*\boxdot_n x}{y}=\ip{x}{z\boxdot_n y}$ for $x,y,z\in \A_n$ if $n\le 3$.
\end{enumerate}
\end{lemma}

\begin{proof}
$(a)$ We proceed by induction on $n$. For $n=0$ the equality is obvious. Assume that it holds for $n$ and fix $x=(x_1,x_2)$ and $y=(y_1,y_2)$ in $\A_{n+1}$. Then
$$\begin{aligned}
\frac12(x\boxdot_{n+1}y^*+\overline{y}\boxdot_{n+1}x^\inv)&=
\frac12((x_1,x_2)\boxdot_{n+1}(y_1^*,-\overline{y_2})+(\overline{y_1},\overline{y_2})\boxdot_{n+1}(x_1^\inv,-x_2)) 
\\&=\frac12((x_1\boxdot_n y_1^*+\overline{y_2}\boxdot_n x_2^\inv, -x_1^\inv\boxdot_n \overline{y_2}+y_1^*\boxdot_n x_2)
\\&\qquad+(\overline{y_1}\boxdot_n x_1^\inv + x_2\boxdot_n \overline{y_2}^\inv,-\overline{y_1}^\inv\boxdot_n x_2+x_1^\inv\boxdot_n \overline{y_2})) 
\\&=\frac12(x_1\boxdot_n y_1^*+\overline{y_1}\boxdot_n x_1^\inv
+  x_2\boxdot_n y_2^*+\overline{y_2}\boxdot_n x_2^\inv,0)
\\&=(\ip{x_1}{y_1}+\ip{x_2}{y_2},0)=\ip xy
\end{aligned}$$
and
$$\ip{x^\inv}{y^\inv}=\ip{(x_1^\inv,-x_2)}{(y_1^\inv,-y_2)}=
\ip{x_1^\inv}{y_1^\inv}+\ip{x_2}{y_2}=\ip{x_1}{y_1}+\ip{x_2}{y_2}
=\ip xy.$$

$(b)$  Using the definition and $(a)$ we get 
$$\J 11x=\ip 11 x +\ip x1 1 -\ip 1{\overline{x}}1
=x+\frac12(x+x^\inv)-\frac12(\overline{x}^*+x)=x$$
for any $x\in \A_n$, thus $1$ is unitary. Further, for $x,y\in \A_n$ we have
$$\J 1x1= 2\ip 1x 1-\ip 11 \overline{x}\overset{(a)}{=}x^*+\overline{x}-\overline{x}=x^*$$
and
$$\begin{aligned}
\J x1y&=\ip x1 y+\ip y1 x-\ip{x}{\overline{y}}1
=y\boxdot_n\frac12(x+x^\inv) +x\boxdot_n \frac12(y+y^\inv)\\&\qquad-
\frac12(x\boxdot_n \overline{y}^*+y\boxdot x^\inv)\\&=\frac12(x\boxdot_n y+y\boxdot_n x).\end{aligned}$$

$(c)$ This follows from $(a)$ and $(b)$.

$(d)$ For $n\le2$ the algebra $\A_n$ is associative, hence the triple product must be given by this formula.  For $n=3$ it follows
from the formula
$$\J xyz=(x\circ_1 y^*)\circ_1 z+ x\circ_1(y^*\circ_1 z)-(x\circ_1 z)\circ_1 y^*$$
using $(b)$ and Lemma~\ref{L:CD doubling}$(xiii)$.

$(e)$ If $n\le 2$, then $\boxdot_n$ is associative, so the equalities follow easily from $(a)$ and Lemma~\ref{L:CD doubling}. For $n=3$ we have
$$\begin{aligned}
\ip{x\boxdot_3 z^*}{y}&=\ip{(x_1,x_2)\boxdot_3(z_1^*,-\overline{z_2})}{(y_1,y_2)}\\&=\ip{(x_1\boxdot_2z_1^*+\overline{z_2}\boxdot_2 x_2^\inv,-x_1^\inv\boxdot_2\overline{z_2}+z_1^*\boxdot_2x_2)}{(y_1,y_2)}
\\&=\ip{x_1\boxdot_2z_1^*}{y_1}+\ip{\overline{z_2}\boxdot_2x_2^\inv}{y_1} -\ip{x_1^\inv\boxdot_2\overline{z_2}}{y_2}+\ip{z_1^*\boxdot_2x_2}{y_2}
\\&=\ip{x_1}{y_1\boxdot_2z_1}+\ip{x_2^\inv}{z_2^\inv\boxdot_2y_1}-\ip{x_1^\inv}{y_2\boxdot_2z_2^\inv}+\ip{x_2}{z_1\boxdot_2y_2}
\\&=\ip{x_1}{y_1\boxdot_2z_1}+\ip{x_2}{y_1^\inv\boxdot_2z_2}-\ip{x_1}{z_2\boxdot_2y_2^\inv}+\ip{x_2}{z_1\boxdot_2y_2}
\\&=\ip{(x_1,x_2)}{(y_1\boxdot_2z_1-z_2\boxdot_2y_2^\inv,y_1^\inv\boxdot_2z_2+z_1\boxdot_2y_2)}=\ip{x}{y\boxdot_3 z}
\end{aligned}$$
The second case is similar.
 \end{proof}
 
We finish this section by discussing coincidence of several structures.

\begin{lemma}\label{l 6.7 two three and four dimensional spins}
\begin{enumerate}[$(i)$]
    \item The algebra $\A_2$ is isomorphic to the biquaternion algebra, i.e., to the algebra of quaternions with complex coefficients. The identity mapping is a witnessing isomorphism.
    \item The algebra $\A_2$ is $*$-isomorphic to the matrix algebra $M_2$. A witnessing isomorphism is defined by the following assignment.
    $$e_1\mapsto\begin{pmatrix}
1 & 0 \\ 0 & 1 
\end{pmatrix}, e_2\mapsto
\begin{pmatrix}
i & 0 \\ 0 & -i 
\end{pmatrix},
e_3\mapsto
\begin{pmatrix}
0 & 1 \\ -1 & 0 
\end{pmatrix}, e_4\mapsto
\begin{pmatrix}
0 & i \\ i & 0 
\end{pmatrix}.$$
    \item The isomorphism from $(ii)$ is a triple-isomorphism of the four-dimeansional spin factor $\A_2$  onto the spin factor defined on $M_2$ equipped with the normalized Hilbert-Schmidt inner product and the conjugation with respect to the orthonormal basis from $(ii)$. Moreover, the norm of this spin factor coincides with the operator norm on $M_2$.
    \item The algebra $\A_1$ is $*$-isomorphic to $\ce\oplus^{\infty}\ce$. Hence, the two-dimensional spin factor $\A_1$ is triple-isomorphic to $\ce\oplus^{\infty}\ce$. A witnessing isomorphism is given by $(x_1,x_2)\mapsto(x_1+ix_2,x_1-ix_2)$.
\end{enumerate}
\end{lemma}

\begin{proof}
$(i)$ It is enough to observe that $e_1=1$, $e_2\boxdot_2 e_2=e_3\boxdot_2 e_3=e_4\boxdot_2 e_4=-1$,
$e_2\boxdot_2 e_3=-e_3\boxdot_2 e_2=e_4$, $e_3\boxdot_2 e_4=-e_4\boxdot_2 e_3=e_2$, $e_4\boxdot e_2=-e_2\boxdot e_4=e_3$.

$(ii)$ It is enough to observe that the given matrices form a quaternion basis and $e_1^*=e_1$, $e_j^*=-e_j$ for $j=2,3,4$ and the matrices satisfy the same for the standard matrix involution.

$(iii)$ To prove that the respective mapping is a triple isomorphism, it is enough to observe that it is a unitary operator. To this end it is enough to verify that it maps an orthonormal basis onto an orthonormal basis, which is clear.

The coincidence of norms follows form $(ii)$ and Lemma~\ref{L:spin factor as CD}$(d)$. Indeed, it follows that the triple product coming from the spin factor coincides with the triple product coming from the standard C$^*$-algebra structure on $M_2$.

$(iv)$ It follows from $(ii)$ and $(iii)$ that the mapping
$$(x_1,x_2)\mapsto\begin{pmatrix}
x_1+ix_2 &0\\0& x_1-ix_2
\end{pmatrix}$$ 
is simultaneously a $*$-isomorphism and a triple isomorphism of $\A_1$ into $M_2$. Now the statement easily follows.
\end{proof}

\subsection{Type $6$ Cartan factor}\label{subsec:C6}

Exceptional JBW$^*$-triples, i.e., those triples which cannot be found as subtriples of von Neumann algebras,  are described using the complex Cayley numbers. The algebra of complex Cayley numbers is just the algebra $\A_3$ from Section~\ref{sec:CD}. Let us introduce a simplified notation.

The mentioned algebra $\A_3$ will be denoted by $\O$ (for octonions). Recall that it has several structures -- the structure of a Hilbert space with the canonical conjugation, the structure of eight-dimensional spin factor, two involutions (a linear $\inv$ and a conjugate-linear $*$), a product $\boxdot_3$ which is non-commutative and non-associative but is alternative. These structures are interrelated as explained in Section~\ref{sec:CD}. The product $\boxdot_3$ will be in the sequel denoted just by $\boxdot$; the associative products $\boxdot_n$ for $n\le 2$ will be denoted by $\cdot$ or simply omitted as it is usual.

We will consider matrices with entries in $\O$. 
We define for them two involutions and matrix multiplication (denoted again by $\boxdot$) in the standard way, i.e., if $A=(a_{ij})$ is a matrix of type $m\times n$, then $A^*$ and $A^\inv$ are matrices of type $n\times m$; on the place $ij$ the first one has the element $a_{ji}^*$ and the second one the element $a_{ji}^\inv$. Moreover, if $A=(a_{ij})$ is a matrix of type $m\times n$ and $B=(b_{jk})$ is a matrix of type $n\times k$, then 
$A\boxdot B$ is the matrix of type $m\times k$ which has the element $\sum_{j=1}^n a_{ij}\boxdot b_{jk}$ on place $ik$. Note that multiplication is neither commutative nor associative. For square matrices it is neither alternative (cf. \cite[Lemma 2.4.23 and Theorem 2.3.61]{Cabrera-Rodriguez-vol1}). In fact, it is easy to show that the algebra $M_2(\O)$ of $2\times 2$ matrices is even not power-associative. Anyway, using the definitions and Lemma~\ref{L:CD doubling}$(v)$ we easily get
\begin{equation}
    (A\boxdot B)^*=B^*\boxdot A^*\mbox{ and }(A\boxdot B)^\inv=B^\inv\boxdot A^\inv\end{equation}
for matrices of compatible types.

By $H_3(\O)$ denote the ($\inv$-)hermitian $3\times 3$ matrices with entries in $\O$, i.e., those $3\times 3$ matrices $\x$ which satisfy $\x^\inv=\x$. If we define 
$$\x\circ \y=\frac12(\x\boxdot \y+\y\boxdot \x),$$
we obtain an exceptional JB$^*$-algebra (of dimension $27$). 
It follows from \cite[Theorem 2.5.7]{hanche1984jordan} that $\circ$ is indeed a Jordan product (the quoted result deals with the real version). By \cite[Propositions 2.9.2 and 3.1.6 and Corollary 3.1.7]{hanche1984jordan} and \cite[Theorem 2.8]{Wright1977}  there is a (unique, canonical) norm such that $H_3(\O)$ is indeed a JB$^*$-algebra (cf. \cite[Proposition 5.5]{kaup1983riemann}
 for the uniqueness).

Hence $H_3(\O)$ is a JB$^*$-triple when equipped with the triple product
$$\J{\x}{\y}{\z}=(\x\circ\y^*)\circ\z+\x\circ(\y^*\circ \z)-(\x\circ \z)\circ\y^*.  $$ 
 
The just described triple $H_3(\O)$ is the Cartan factor of type $6$. Note that it is a finite-dimensional JB$^*$-algebra, hence using Proposition~\ref{P:Linftyfindim is finite} we get

\begin{prop}
Let $A$ be any abelian von Neumann algebra. Then $A\overline{\otimes}H_3(\O)$ is a finite JBW$^*$-algebra. In particular, the relations $\le_2$ and $\le_0$ coincide.
\end{prop}
 
\begin{remark} It is possible to say something more on the structure of tripotents in $M=H_3(\O)$. There are three types of them.

The first type consists of unitary tripotents. An example is the unit of $M$, i.e., the unit matrix.

The second type are minimal tripotents, for example 
$$u=\begin{pmatrix}
1 & 0 & 0 \\ 0 & 0 &0 \\ 0 & 0 & 0
\end{pmatrix}.$$
It is a projection in $H_3(\O)$. In this case
$M_2(u)=\span\{u\}$. Moreover,
\begin{equation}\label{eq Peirce1 minimal trip 6.3} M_1(u)=\left\{\begin{pmatrix}
0 & x & y \\ x^\inv & 0 &0 \\ y^\inv & 0 & 0
\end{pmatrix}\setsep x,y\in \O\right\},
\end{equation}
so it is isomorphic to the type 5 Cartan factor (see Subsection~\ref{subsec:C5} below).
Furhter,
$$M_0(u)=\left\{\begin{pmatrix}
0 & 0 & 0 \\ 0 & x & z \\ 0 & z^\inv & y
\end{pmatrix}\setsep x,y\in \ce, z  \in\O\right\},$$
which is triple-isomorphic to the ten-dimensional spin factor by the mapping given by the formula
$$(x_1,\dots,x_{10})\mapsto\begin{pmatrix} 0&0&0 \\ 0&
-x_1+x_2 & (x_2,x_3,\dots,x_{10})\\ 
0&(x_2,-x_3,\dots,-x_{10}) & -x_1-x_2
\end{pmatrix}.$$

The third type is represented by $1-u$, where $u$ is as above. Then $M_2(1-u)=M_0(u)$, $M_1(1-u)=M_1(u)$ and $M_0(1-u)=M_2(u)$.

These three types cover all tripotents, because for any tripotent $e\in M$ there is a triple-automorphism of $M$ sending $e$ to one of the projections $1$, $u$ and $1-u$ (see \cite[Proposition 5.8]{kaup1997real} or \cite[\S 8]{loos1977bounded}).
 \end{remark}

\subsection{Type $5$ Cartan factor}\label{subsec:C5} 

The Cartan factor of type $5$ is the subtriple of $H_3(\O)$ defined by 
$$C_5=\left\{\begin{pmatrix}
0 & x_1 & x_2 \\ x_1^\inv & 0&0 \\ x_2^\inv &0&0
\end{pmatrix}\setsep \qquad x_1,x_2\in\O\right\}\subset H_3(\O).$$ 
It is not hard to observe that it is indeed a subtriple. Let us observe that, under this point of view, $C_5$ coincides with the Peirce-1 subspace of a minimal tripotent in $H_3(\O)$ described in the previous subsection (cf. \eqref{eq Peirce1 minimal trip 6.3}). Another possibility is to represent $C_5$
 as the space $M_{1,2}(\O)$ of $1\times 2$ matrices with entries in $\O$ equipped with the triple product defined by
 $$\J{\x}{\y}{\z}=\frac12 (\x\boxdot(\y^*\boxdot\z)+\z\boxdot(\y^*\boxdot\x)),$$
 where $\boxdot$ denotes the matrix multiplication from the previous section. Note that
 the position of  brackets does matter, as it is easy to check that the analogue of Lemma~\ref{L:CD doubling}$(xiii)$ fails in the matrix case.
 
 By a direct computation using Lemma~\ref{L:CD doubling} and Lemma~\ref{L:spin factor as CD} one can verify that the assignment
 $$(x_1,x_2)\mapsto\begin{pmatrix}
0 & x_1 & x_2 \\ x_1^\inv & 0&0 \\ x_2^\inv &0&0
\end{pmatrix}$$
is a triple isomorphism between the two representations and to derive the following more detailed formula for the triple product.
 \begin{equation}\label{eq:triple product in C5}
\begin{aligned}
\J{\x}{\y}{\z}&=( \J{x_1}{y_1}{z_1}+\frac12(z_2\boxdot(y_2^*\boxdot x_1)+x_2\boxdot(y_2^*\boxdot z_1)),\\&\qquad\J{x_2}{y_2}{z_2}+\frac12(z_1\boxdot(y_1^*\boxdot x_2)+x_1\boxdot(y_1^*\boxdot z_2)))   
\end{aligned}\end{equation}
 
 Since $C_5$ has a finite dimension (in fact, $\dim C_5=16$), Proposition~\ref{P:Linftyfindim is finite} can be applied to yield the following proposition.
 
 \begin{prop}
 Let $A$ be any abelian von Neumann algebra. Then $A\overline{\otimes}C_5$ is a finite JBW$^*$-triple.
 \end{prop}
 
 However, the triple $C_5$ admits no unitary element, so unlike in the case of type $6$ Cartan factor, we do not have a JBW$^*$-algebra. We are going to describe two types of tripotents in $C_5$ using the results on spin factors, Subsection~\ref{sec:CD} and the formula \eqref{eq:triple product in C5}.
 
 Set 
 $$Y_1=\{(x_1,x_2)\in C_5\setsep x_2=0\}\mbox{ and }Y_2=\{(x_1,x_2)\in C_5\setsep x_1=0\}.$$ It follows from \eqref{eq:triple product in C5} that $Y_1$ and $Y_2$ are subtriples of $C_5$ canonically isomorphic to $\O$, i.e., they are canonically isomorphic to the eight-dimensional spin factor. Hence, if $u\in \O$ is a tripotent, then $(u,0)$ is a tripotent in $C_5$. Moreover, as $\O$ is a spin factor, nonzero tripotents are of two types from Lemma~\ref{L:tripotents in spin factor}. Let us describe Peirce decompositions of these tripotents.

\begin{prop} Let $u\in\O$ be a tripotent. Then $(u,0)$ is a tripotent in $C_5$. Moreover, the following assertions hold.
\begin{enumerate}[$(i)$]
    \item For any  $\x=(x_1,x_2)\in C_5$ we have
    $$
\begin{aligned}
L((u,0),(u,0))\x&=(\J uu{x_1},\frac12 u\boxdot(u^*\boxdot x_2));
\\
Q((u,0))\x&=(Q(u)x_1,0),\\
P_2((u,0))\x&=(P_2(u)x_1,0).\end{aligned}$$
In particular, 
$$(C_5)_2(u,0)=\{(x,0)\setsep x\in\O_2(u)\}.$$

\item If $u$ is unitary in $\O$, then $(u,0)$ is complete in $C_5$ and
$$(C_5)_2(u,0)=Y_1\mbox{ and }(C_5)_1(u,0)=Y_2.$$

\item Assume $u$ is minimal in $\O$. Then $(u,0)$ is minimal in $C_5$ and
$$\begin{aligned}
(C_5)_2(u,0)&=\span\{(u,0)\},\\
(C_5)_1(u,0)&= \{(x,0)\setsep x\in\{u,\overline{u}\}^{\perp_2}\}\oplus\{(0,x)\setsep u\boxdot(u^*\boxdot x)=x\}\\&=\{(x,0)\setsep x\in\{u,\overline{u}\}^{\perp_2}\}\oplus\span\{(0,u)\}\\&\qquad\qquad\oplus\{(0,x)\setsep x\in\{u,\overline{u}\}^{\perp_2}\ \&\ x\boxdot(u^*\boxdot u)=0\} \\
(C_5)_0(u,0)&=\span\{(\overline{u},0)\}\oplus\{(0,x)\setsep u\boxdot(u^*\boxdot x)=0\}\\&=
\span\{(\overline{u},0),(0,\overline{u})\}\oplus\{(0,x)\setsep x\in\{u,\overline{u}\}^\perp\ \&\ x\boxdot(u^*\boxdot u)=x\}
\end{aligned}$$
Moreover,
$$(C_5)_0(u,0)\cap Y_2=\{(0,\overline{x})\setsep (0,x)\in (C_5)_1(u,0)\},$$
hence
$$\dim (C_5)_2(u,0)=1,\dim (C_5)_0(u,0)=5,\dim (C_5)_1(u,0)=10.$$
\end{enumerate}
\end{prop}

\begin{proof}
$(i)$ The first two equalities follow immediately from \eqref{eq:triple product in C5}, the third one follows from the second one. The formula for the Peirce-2 subspace follows from the formula for the projection.

$(ii)$ If $u\in\O$ is unitary, then $\O_2(u)=\O$, hence by $(i)$ we get $(C_5)_2(u,0)=Y_1$.

Further, by Lemma~\ref{L:tripotents in spin factor} we know that
$u=\alpha e$ where $e\in\O$ has real coordinates, $\ip ee=1$ and $\alpha$ is a complex unit. Thus, using Lemma~\ref{L:spin factor as CD}$(a)$ we get
$$1=\ip ee=\frac12(e\boxdot e^*+\overline{e}\boxdot e^\inv)=e\boxdot e^\inv=
\alpha e \boxdot \overline{\alpha} e^\inv=u\boxdot u^*,$$
hence for any $x\in\O$ we have
$$\frac12 u\boxdot(u^*\boxdot x)=\J uux - \frac12(u\boxdot u^*)\boxdot x=x-\frac12x=\frac12x,$$
so $Y_2\subset (C_5)_1(u,0)$. Since $C_5$ is (as a linear space) a direct sum of $Y_1$ and $Y_2$, we deduce that $(C_5)_1(u,0)=Y_2$
and $(u,0)$ is complete.

$(iii)$ The formula for $(C_5)_2(u,0)$ follows from $(i)$ and 
Lemma~\ref{L:tripotents in spin factor}.

Further observe that
$$(C_5)_j(u,0)=((C_5)_j(u,0)\cap Y_1)+((C_5)_j(u,0)\cap Y_2)\mbox{ for }j=0,1.$$
Indeed, $e=u+\overline{u}$ is a unitary tripotent in $\O$, thus by $(ii)$ the canonical projections onto $Y_1$ and $Y_2$ coincide with the Peirce-$2$ and Peirce-$1$ projections of $(e,0)$. Since $(u,0)\in Y_1=(C_5)_2(e,0)$, by Proposition~\ref{P:M2 inclusion} we deduce that Peirce projections of $(u,0)$ and $(e,0)$ commute, so the above decomposition follows. Using this observation, $(i)$ and Lemma~\ref{L:tripotents in spin factor} we get
$$(C_5)_0(u,0)\cap Y_1=\span\{(\overline{u},0)\} \mbox{ and }
(C_5)_1(u,0)\cap Y_1=\{(x,0)\setsep x\in\{u,\overline{u}\}^{\perp_2}\}.$$
It remains to determine the intersections with $Y_2$.
It follows from $(i)$ that
$$\begin{aligned}
(C_5)_1(u,0)\cap Y_2&=\{(0,x)\setsep u\boxdot(u^*\boxdot x)=x\},
\mbox{ and }\\ (C_5)_0(u,0)\cap Y_2&=\{(0,x)\setsep u\boxdot(u^*\boxdot x)=0\}.\end{aligned}$$
Further, observe that the tripotents $(u,0)$ and $(\overline{u},0)$ are orthogonal, thus
$$L((e,0),(e,0))=L((u+\overline{u},0),(u+\overline{u},0))=L((u,0),(u,0))+L((\overline{u},0),(\overline{u},0)).$$
Taking into account that $Y_2=(C_5)_1(e,0)$, we deduce that
$$(C_5)_1(u,0)\cap Y_2=(C_5)_0(\overline{u},0)\cap Y_2
\mbox{ and }(C_5)_0(u,0)\cap Y_2=(C_5)_1(\overline{u},0)\cap Y_2.$$
Further, for $x\in \O$ we have
\begin{multline*}
(0,x)\in (C_5)_0(u,0)\Leftrightarrow u\boxdot(u^*\boxdot x)=0
\Leftrightarrow \overline{u}\boxdot(\overline{u}^*\boxdot \overline{x})=0\\ \Leftrightarrow (0,\overline{x})\in (C_5)_0(\overline{u},0)\Leftrightarrow (0,\overline{x})\in (C_5)_1(u,0),\end{multline*}
hence
$$(C_5)_0(u,0)\cap Y_2=\{(0,\overline{x})\setsep (0,x)\in (C_5)_1(u,0)\},$$
in particular
$$\dim (C_5)_0(u,0)\cap Y_2=\dim (C_5)_1(u,0)\cap Y_2=4.$$

Further, since
$$u\boxdot (u^*\boxdot u)=\J uuu=u,$$
we deduce that $(0,u)\in(C_5)_1(u,0)$ and hence $(0,\overline{u})\in (C_5)_0(u,0)$.

It follows from Lemma~\ref{L:spin factor as CD}$(e)$ that the operator
$$x\mapsto u\boxdot(u^*\boxdot x), \qquad x\in\O$$
is self-adjoint (when $\O$ is considered as a Hilbert space).
Hence, it has an orthonormal basis consisting of eigenvectors.
By the above we know that $u$ and $\overline{u}$ are eigenvectors (with respect to the eigenvalues $1$ and $0$), hence it easily follows that $E=\{u,\overline{u}\}^{\perp_2}$ is invariant for this operator. Moreover, for any $x\in E$ we have

$$u\boxdot(u^*\boxdot x)=2\J uux- x\boxdot(u^*\boxdot u)=
x- x\boxdot(u^*\boxdot u).$$
Now the formulas easily follow.
\end{proof}

In $C_5$ there are just two types of nonzero tripotents essentially described above. Let us explain it:

Let $e$ be a complete tripotent in $C_5$. Since $C_5$ is a Cartan factor, there is an automorphism of $C_5$ transforming $e$ to $(u,0)$ where $u$ is a complete (hence unitary) tripotent in $\O$ (cf. \cite[Proposition 5.8]{kaup1997real} and \cite{loos1977bounded}). In particular,
both subspaces $(C_5)_2(e)$ and $(C_5)_1(e)$ are eight-dimensional and isomorphic to $\O$.

 Let $v$ be a non-complete non-zero tripotent in $C_5$. Then there is a complete tripotent $e\ge v$. Then $e$ is, up to an automorphism
of the form $(u,0)$ where $u$ is unitary in $\O$. Thus $v$ must be of the form $(w,0)$ where $w\le u$ is a non-complete non-zero tripotent in $\O$, so it has the above form.

Observe that, given a complete tripotent $e\in C_5$ and a minimal tripotent $v\le e$, there is, up to a scalar multiple, a unique non-zero tripotent $w\le e$ orthogonal to $v$. Indeed, it must be a mutliple of $e-v$.

On the other hand, if $v$ is a minimal tripotent, then the choice of a complete tripotent $e\ge v$ is highly non-unique. Even the Peirce-2 subspace of $e$ is non-unique.
Indeed, $(C_5)_0(v)$ has dimension $5$, so if we choose any nonzero tripotent $w\in (C_5)_0(w)$, then $e=v+w$ is a complete tripotent above $v$. Moreover, 
$$(C_5)_2(e)=\span\{v,w\}+((C_5)_1(v)\cap(C_5)_1(w))$$ 
and $$(C_5)_1(e)=(C_5)_1(v)\cap(C_5)_0(w) + (C_5)_0(v)\cap(C_5)_1(w).$$

 We finish this section by the following characterization of the preorder $\le_0$ in $C_5$.
 
 \begin{prop}\label{P:le0 in C5}
Let $u,v\in C_5$ be two nonzero tripotents. Then $u\le_0 v$ if and only if either $v$ is complete or $u=\alpha v$ for a complex unit $\alpha$.
\end{prop}

\begin{proof} 
 The `if part' is obvious. To prove the `only if part' assume that $v$ is not complete. Since $(C_5)_0(v)\subset (C_5)_0(u)$ and both are subspaces of dimension $5$, we have $(C_5)_0(v)= (C_5)_0(u)$.  
Fix any nonzero tripotent $w\in (C_5)_0(u) (=(C_5)_0(v))$. Then $u+w$ is a complete tripotent, so,   up to an automorphism we may assume that $u+w=(e,0)$, where $e$ is a unitary tripotent in $\O$.
Then $u$ also has the second coordinate zero, thus let us write 
it as $(u,0)$ where $u\in\O$ is a non-unitary tripotent.
Then $w$ is a multiple of $(\overline{u},0)$. Without loss of generality assume $w=(\overline{u},0)$ (this assumption changes $e$ but preserves the Peirce decomposition of $e$).
Since $w\in (C_5)_0(v)$, we have also $v\in (C_5)_0(w)$. The above representation yields that
$v=(\alpha u,\beta u+z)$, where $\alpha,\beta\in\ce$ and $z\in\{u,\overline{u}\}^{\perp_2}$ satisfies $u\boxdot(u^*\boxdot z)=z$ and
$\overline{u}\boxdot(\overline{u}^*\boxdot z)=0$.

Since $(0,\overline{u})\in (C_5)_0(u,0)=(C_5)_0(v)$, we have 
$v\in(C_5)_0{(0,\overline{u})}$. By the above representation (exchange of coordinates is an automorphism) we deduce that 
$v=(\alpha'u+z',\beta'u)$ where  where $\alpha',\beta'\in\ce$ and $z'\in\{u,\overline{u}\}^{\perp_2}$ satisfies $u\boxdot(u^*\boxdot z')=z'$ and
$\overline{u}\boxdot(\overline{u}^*\boxdot z')=0$.

If we combine the two representations, we deduce that $v=(\alpha u,\beta u)$ for some $\alpha,\beta\in\ce$.

Fix any $x\in\{u,\overline{u}\}^{\perp_2}$ with $u\boxdot(u^*\boxdot x)=0$. Then $x\in (C_5)_0(u,0)=(C_5)_0(v)$, thus

$$\begin{aligned}
0&=\J vv{(0,x)}\\&=\frac12\left(
(\alpha u,\beta u)\boxdot\left(
\begin{pmatrix}
\overline{\alpha}u^*\\ \overline{\beta}u^*
\end{pmatrix}\boxdot(0,x)\right)+(0,x)
\boxdot\left(
\begin{pmatrix}
\overline{\alpha}u^*\\ \overline{\beta}u^*
\end{pmatrix}\boxdot(\alpha u,\beta u)\right)
\right)
\\&=\frac12\left((\alpha u,\beta u)\boxdot\begin{pmatrix}
0 & \overline{\alpha} u^*\boxdot x \\ 0& \overline{\beta} u^*\boxdot x
\end{pmatrix}
+(0,x)\boxdot\begin{pmatrix}
\abs{\alpha}^2 u^*\boxdot u & \overline{\alpha}\beta u^*\boxdot u \\
\alpha\overline{\beta} u^*\boxdot u & \abs{\beta}^2 u^*\boxdot u 
\end{pmatrix}
\right)
\\&=\frac12((0,(\abs{\alpha}^2+\abs{\beta}^2) u\boxdot (u^*\boxdot x))
+(\overline{\alpha}\beta x\boxdot(u^*\boxdot u),\abs{\beta}^2 x\boxdot(u^*\boxdot u)))
\\&=\frac12(\overline{\alpha}\beta x\boxdot(u^*\boxdot u),\abs{\beta}^2 x\boxdot(u^*\boxdot u)).
\end{aligned}$$
If $\beta\ne0$ we get 
$$0=x\boxdot(u^*\boxdot u)=2\J uux-u\boxdot(u^*\boxdot x)=2\cdot\frac12 x-0=x.$$
Since $x$ may be nonzero, we deduce that $\beta=0$, hence $v$ is a multiple of $(u,0)$. This completes the proof.
\end{proof}

\section{Synthesis of the results}\label{sec:final}

Using the results from Sections~\ref{sec:vN and pV},~\ref{sec:symmetric and antisymmetric} and~\ref{sec:6 - spin and exceptional} we get the following variant of the representation \eqref{eq:representation of JBW* triples}.

\begin{thm}\label{T:synthesis} Let $M$ be a JBW$^*$-triple. Then $M$ can be represented as
$$M=M_1\oplus^{\ell_\infty}M_2\oplus^{\ell_\infty}M_3\oplus^{\ell_\infty}M_4,$$
where the summands have the following form.
\begin{enumerate}[$(a)$]
    \item $M_1$ is a finite JBW$^*$-algebra. $M_1$ can be further decomposed as
    $$M_1=V_1\oplus^{\ell_\infty} \left(\bigoplus^{\ell_\infty}_{j\in J}L^\infty(\mu_j,C_j)\right)\oplus^{\ell_\infty} \left(\bigoplus^{\ell_\infty}_{\lambda\in \Lambda} A_\lambda\overline{\otimes}B(H_\lambda)_s\right) \oplus^{\ell_\infty}H(W,\alpha),$$
    where
\begin{itemize}
\item $V_1$ is a finite von Neumann algebra;
\item $J$ and $\Lambda$ are (possibly empty) sets;
\item $\mu_j$ is a probability measure for each $j\in J$;
\item For each $j\in J$ the space $C_j$ is either a spin factor, or a finite-dimensional Cartan factor of type $3$ or $6$, or $C_j=B(H)_a$ for a Hilbert space $H$ of finite even dimension;
\item $A_\lambda$ is an abelian von Neumann algebra and $H_\lambda$ is an infinite-dimensional Hilbert space for $\lambda\in\Lambda$;
\item $W$ is a (possibly trivial) continuous von Neumann algebra and $\alpha$ is a central linear involution on $W$ commuting with $*$.
\end{itemize}
\item $M_2$ is either a trivial space or a properly infinite JBW$^*$-algebra. Moreover, $M_2$ can be further decomposed as
  $$M_2=V_2\oplus^{\ell_\infty}\left(\bigoplus^{\ell_\infty}_{\gamma\in \Gamma} A_\gamma\overline{\otimes}B(H_\gamma)_a\right),$$
  where
\begin{itemize}
\item $V_2=\{0\}$ or it is a properly infinite von Neumann algebra;
\item $\Gamma$ is a (possibly empty) set;
\item $A_\gamma$ is an abelian von Neumann algebra and $H_\gamma$ is an infinite-dimensional Hilbert space for $\gamma\in\Gamma$.
\end{itemize}
\item $M_3$ is a finite JBW$^*$-triple with no nonzero direct summand isomorphic to a JBW$^*$-algebra. Moreover, $M_3$ can be further decomposed as
 $$M_3=pV_3\oplus^{\ell_\infty} \left(\bigoplus^{\ell_\infty}_{\delta\in \Delta}L^\infty(\nu_\delta,D_\delta)\right),$$
 where
\begin{itemize}
\item $V_3$ is a von Neumann algebra and $p\in V_3$ is a finite projection;
\item $\Delta$ is a (possibly empty) set;
\item $\nu_\delta$ is a probability measure for $\delta\in\Delta$;
\item For each $\delta\in\Delta$ the space $D_\delta$ is either the Cartan factor of type $5$ or the triple $B(H)_a$ where $H$ is a Hilbert space of finite odd dimension (at least $3$).
\end{itemize}
\item $M_4=\{0\}$ or $M_4=qV_4$, where $V_4$ is a von Neumann algebra, $q\in V_4$ is a properly infinite projection such that $qV_4$ has no direct summand isomorphic to a JBW$^*$-algebra. If $M_4$ is not zero, it is properly infinite.
\end{enumerate}
\end{thm}

\begin{remark}
\begin{enumerate}[$(1)$]
    \item The decomposition given in Theorem~\ref{T:synthesis} focuses on distinguishing   directs summands which are finite or properly infinite
    and further, on putting together direct summand with a unitary element. So, it is not, strictly speaking, a refinement of the decomposition \eqref{eq:representation of JBW* triples}, which distinguishes, in the terminology of \cite{horn1987classification,horn1988classification}, type I summands and continuous ones. However, it is easy to provide a common refinement of these two representations. Let us briefly explain it:
    \begin{itemize}
        \item In the summand $M_1$, the part $H(W,\alpha)$ is continuous.
        \item Each of the von Neumann algebras $V_1-V_4$   can be decomposed into two summands -- one of type I, the second one continuous (see \cite[Theorem V.1.19]{Tak}). This yields also the decomposition of $pV_3$ and $qV_4$ to the type I and continuous parts.
        \item The remaining summands from Theorem~\ref{T:synthesis} are of type I. 
    \end{itemize}
  \item The theory of types of von Neumannn algebras includes also investigation of type III (purely infinite) algebras. By analogy we may say that a tripotent $e$ is a JBW$^*$-triple is \emph{purely infinite} if $M_2(e)$ contains no nonzero finite tripotents. However, this approach gives nothing interesting beyond von Neumann algebras. Indeed, the unique purely infinite JBW$^*$-algebras are von Neumann algebras and the unique further purely infinite JBW$^*$-triples are those of the form $qV$ where $V$ is a von Neumann algebra and $q\in V$ is a purely infinite projection.
   \end{enumerate}
\end{remark}

\begin{prop}\label{P:finite one complete}
Let $M$ be a JBW$^*$-triple. The following assertions are equivalent.
\begin{enumerate}[$(1)$]
    \item $M$ is finite;
    \item Any complete tripotent in $M$ is finite;
     \item There is a complete finite tripotent in $M$;
    \item The Peirce-2 subspace of any complete tripotent is maximal with respect to inclusion;
     \item There is a finite tripotent whose Peirce-2 subspace is maximal with respect to inclusion.
\end{enumerate}
\end{prop}

\begin{proof}
The implications $(1)\Rightarrow(2)\Rightarrow(3)$ are trivial.

$(3)\Rightarrow(1)$ Assume $M$ is infinite. By Theorem~\ref{T:synthesis} it has a properly infinite direct summand $N$. Moreover, by combining Theorem~\ref{T:synthesis} with Proposition~\ref{P:vN properly infinite trip}, Proposition~\ref{P:pV decomp} and Proposition~\ref{P:C2 properly infinite} we see that any complete tripotent in $N$ is properly infinite. 
If $e\in M$ is now a complete tripotent, then $P_Ne$ is a complete tripotent in $N$. Sinc $P_Ne$ is properly infinite, $e$ is infinite.

$(2)\Rightarrow(4)\&(5)$ Assume $(2)$ is valid. Let $u$ be a complete tripotent. Assume that $v$ is a tripotent such that $M_2(u)\subset M_2(v)$. Then $v$ is complete by Proposition~\ref{P:M2 inclusion}. By $(2)$ we deduce that $v$ is finite, thus $u$ is unitary in $M_2(v)$, i.e., $M_2(u)=M_2(v)$. It proves both $(4)$ and $(5)$. 

$(5)\Rightarrow(3)$ Let $u$ be a tripotent provided by $(5)$. Then $u$ is finite and, moreover, it is complete. (Otherwise there is a nonzero tripotent $v$ with $v\perp u$ and the Peirce-$2$ subspace of $u+v$ is strictly larger than that of $u$.)

$(4)\Rightarrow(1)$ This follows using Theorem~\ref{T:synthesis}. It is enough to show that any properly infinite JBW$^*$-triple fails $(4)$. If it is a JBW$^*$-algebra, it follows by the very definition (there is a unitary tripotent and another complete non-unitary tripotent). 

It remains to consider the case $M=pV$,  where $p$ is a properly infinite in a von Neumann algebra $V$. Since $p$ is infinite, there is a projection $q<p$ with $q\sim p$ Let $u\in V$ be a partial isometry with $p_i(u)=q$ and $p_f(u)=p$. Then $u$ and $p$ are complete tripotents in $M$ and $M_2(u)=qMp\subsetneqq pMp =M_2(p)$.
\end{proof}

Let $P(M)$ be the projection lattice of a von Neumann algebra $M$. An important feature of Murray-von Neumann concept of finiteness  is that finite projections form a (modular) sublattice in $P(M)$ (see e.g. \cite[Theorem V.1.37]{Tak}). This motivated the interest of von Neumann in the projective geometry. Especially, the supremum of two finite projections in the projection lattice is again a finite projection. Unlike lattice structure enjoyed by the set of projections in a von Neumann algebra, the poset of tripotents in a JBW$^*$-triple need not be a lattice (the infimum of a non-empty set of tripotents in a JBW$^*$-triple always exists, however a set of tripotents admits supremum if and only if it has an upper bound \cite[Theorem 3.6]{BattagliaOrder1991}). The poset of tripotents in a JBW$^*$-triple $M$ admits a greatest element if
and only if $M=\{0\}$.  
Nevertheless, if the  supremum of two finite elements does exists, then it has to be finite again. This is the content of the following Proposition.       
	
\begin{prop}
 Let $M$ be a JBW$^*$-triple and let $u,v\in M$ be two finite tripotents  having supremum $s$ in the tripotent poset. Then $s$ is finite. In particular, if $u$,  $v$ are orthogonal finite tripotents, then $u+v$ is finite. 
	\end{prop}

\begin{proof}
	We can reduce the problem to the case when $M$ is a JBW$^*$-algebra and $u$ and $v$ are projections in $M$ having supremum 1. Indeed, $M_2(s)$ has the structure of a JBW$^*$-algebra in which $u$ and $v$ are projections that are finite (as tripotents) and their supremum $s$ is the unit. In fact,  we have to show that under these circumstances $M$ is a finite $JBW^*$-algebra. In virtue of Theorem~\ref{T:synthesis} (a) and  (b) and using Proposition~\ref{P:preorders in direct sum} we have to prove  that $M$ cannot be equal to one of the the following JBW$^*$-algebras: (a) infinite von Neumann algebra (b) $A\overline{\otimes} B(H)_a$, where $A$ is an abelian von Neumann algebra and $H$ is an infinite dimensional Hilbert space.\\
	
	 Let us consider  case (a). Suppose that $M$ is a von Neumann algebra. We know that a projection in a von Neumann algebra is finite as a tripotent if and only if it is finite as a projection (see
	 Lemma~\ref{L:finite tripotents}$(a)$).
		 As the supremum of two finite projections in a von Neumann algebra is a finite projection (see  \cite[Theorem V.1.37]{Tak}),  we obtain that the unit of $M$ is finite and so $M$ is finite. So  case (a) cannot happen.\\
	 
	  Let us address case (b). Suppose that $M$ is $A\overline{\otimes} B(H)_a$  with $A$ and $H$ specified above and try to  reach a contradiction.  Proposition~\ref{finite-finite} implies that the initial projections $p_i(u)$ and $p_i(v)$ of $u$ and $v$, respectively, are finite  projections in the von Neumann algebra  $N=A\overline{\otimes} B(H)$. Let  $q$ be the supremum of  projections $p_i(u)$ and $p_i(v)$. As $N$ is infinite  we infer that $1-q$ cannot be abelian, otherwise the unit of $N$ would be the supremum of three finite projections and therefore finite (again by \cite[Theorem V.1.37]{Tak}). According to Proposition~\ref{P:C2 projections} (ii) there is a nonzero tripotent in $M$ whose initial projection is under $1-q$.  This tripotent is orthogonal to both $u$ and $v$ and that contradicts with the assumption that $s$, the  supremum of $u$ and $v$, is unitary (hence complete). 
	\end{proof}

\begin{prop}\label{p finite tripotent bounded by unitary}
Let $M$ be a JBW$^*$-algebra and $u\in M$ a finite tripotent. Then there is a unitary element $e\in M$ with $u\le e$.
\end{prop}

\begin{proof}
It is enough to prove it for the individual summands from Theorem~\ref{T:synthesis}. If $M$ is finite, the statement follows from Lemma~\ref{L:finite tripotents}$(d)$. If $M$ is a von Neumann algebra, the statement is proved in \cite[Proposition V.1.38]{Tak}.

It remains to consider the case $M=L^\infty(\mu)\overline{\otimes}B(H)_a$ where $H$ is an infinite-dimensional Hilbert space. Proposition~\ref{finite-finite} says that $p_i(u)$ is a finite projection in the von Neumann algebra
$L^\infty(\mu)\overline{\otimes}B(H)$. Hence $1-p_i(u)$ is properly infinite. By 
Lemma~\ref{L:finite projections basic facts}$(c)$ 
$1-p_i(u)$ is the sum of a pair of equivalent orthogonal projections, thus Proposition~\ref{P:C2 projections}$(i)$ shows that there is a tripotent $v\in M$ with $p_i(v)=1-p_i(u)$. Clearly $v\perp u$ and $p_i(u+v)=1$, so $u+v$ is unitary.
\end{proof}

The conclusion of the previous proposition should be compared with Lemma \ref{L:finite tripotents}$(d)$.
We continue by the following result on decomposing an infinite tripotent.

\begin{prop}
Let $M$ be a JBW$^*$-triple and let $e\in M$ be an infinite tripotent. Then there is a direct summand $N$ of $M$ such that $P_Ne$ is properly infinite and $(I-P_N)e$ is finite.
\end{prop}

\begin{proof}
It is enough to prove the statement in case $M$ is one of the properly infinite summands from the representation given in Theorem~\ref{T:synthesis}.

The proof is in all three cases essentially the same. Let us explain the setting. We have a von Neumann algebra $W$ and its subtriple $M$ of a special form:
\begin{enumerate}[C{a}se 1:]
    \item $M=W$;
    \item $M=pW$ for a projection $p\in W$;
    \item $W=A\overline{\otimes}B(H)$ where $A$ is an abelian von Neumann algebra, $H$ an infinite-dimensional Hilbert space and $M=W_a$.
\end{enumerate}

So, assume that $W$ and $M$ are of one of these three forms. Let $e\in M$ be an infinite tripotent. Then $p_i(e)$ is an infinite projection in $W$ (by Proposition~\ref{P:vN finite trip}, Remark~\ref{rem:pV reduced to V} or Proposition~\ref{finite-finite}). 
By \cite[Proposition 6.3.7]{KR2} there is a central projection $z\in W$ such that $zp_i(e)$ is properly infinite and $(1-z)p_i(e)$ is finite.
It is enough to set $N=zM$ and use again Proposition~\ref{P:vN finite trip}, Remark~\ref{rem:pV reduced to V} or Proposition~\ref{finite-finite}.
\end{proof}

The following proposition reveals the relationship of the relations $\le_0$ and $\le_2$ in general JBW$^*$-triples.

\begin{prop}
Let $M$ be a JBW$^*$-triple and let $e,u\in M$ be two tripotents. Then $u\le_0 e$ if and only if there is a direct summand $N$ of $M$ such that $P_Nu\le_2 P_Ne$ and $(I-P_N)e$ is complete in $(I-P_N)M$.
\end{prop}

\begin{proof}
It is enough to prove the statement for the individual summands in the representation from Theorem~\ref{T:synthesis}. We use the notation from the quoted theorem.

The summand $M_1$ is a finite JBW$^*$-algebra, hence the statement follows from Proposition~\ref{P:le0=le2}.
For $V_2$ we may use Lemma~\ref{l charcteriz le2 and le0 in vN}$(b)$. The case of $A_\gamma\overline{\otimes}B(H_\gamma)_a$ follows from Proposition~\ref{P:le0 in AotimesB(H)ainfty}. The summands $pV_3$ and $qV_4$ are settled by Proposition~\ref{P:le0 in pV}. Finally, the summands $L^\infty(\nu_\delta,D_\delta)$ are covered by Propositions~\ref{P:le0 in AotimesB(H)aodd}
and~\ref{P:le0 in C5}.
\end{proof}

\begin{remark} 
An important role in study of von Neumann algebras is played by so-called halving lemmata. There are two kinds of these results.
\begin{enumerate}[$(1)$]
    \item In a continuous von Neumann algebra any projection can be decomposed into the sum of two orthogonal equivalent projections.
    This property even characterizes continuous von Neumann algebras by \cite[Proposition V.1.35]{Tak}. There is an analogue of this result for continuous JBW$^*$-triples, see \cite[Appendix]{horn1988classification}.
    \item If $p$ is a properly infinite projection in a von Neumann algebra, it can be decomposed into the sum of two orthogonal projections which are both equivalent to $p$ (see Lemma~\ref{L:finite projections basic facts}$(c)$). 
       By Lemma~\ref{L:finite projections basic facts}$(d)$ $p$ may be even decomposed into the sum of an infinite sequence of mutually orthogonal projections such that all of them are equivalent to $p$.
    
    In this case the situation in JBW$^*$-triples is more complicated. An analogue for properly infinite tripotents holds 
    in von Neumann algebras (by Proposition~\ref{P:vN halving}) and for the triples of the form $pV$ (by Proposition~\ref{P:pV halving}). 
    
    However, for triples of the form $A\overline{\otimes}B(H)_a$ the analogue completely fails by Proposition~\ref{P:C2 properly infinite}$(b)$. In fact, this JBW$^*$-triple is properly infinite whenever $\dim H=\infty$, but in some sense it is `almost finite'. Indeed, a complete tripotent need not be unitary but it is `almost unitary' in the sense of Proposition~\ref{P:complete = abelian}.
\end{enumerate}
\end{remark}

\subsection{Finiteness and modularity in JBW$^*$-algebras}

In von Neumann algebras finiteness of projections is closely related to modularity -- by \cite[Theorem V.1.37]{Tak} the set of finite projections in a von Neumann algebra is a modular lattice. A converse holds as well -- if the projection lattice of a von Neumann algebra is modular, the algebra is finite (see \cite[Theorem on page 5]{KaplanskyOrhtocompl55}). In this subsection we clarify the relationship of finiteness and modularity in JBW$^*$-algebras. Let us start by recalling basic notions from lattice theory.

A lattice $L$ (i.e. a partially ordered set in which any two elements $a$ and $b$ have a least upper bound $a \vee b$ and a greatest lower bound
$a \wedge b$) is called \emph{modular} if $e\leq g$ implies $$ (e\vee f)\wedge g= e\vee (f\wedge g),\ \hbox{ for all } f\in L.$$ 
Observe that the inequality $e\vee (f\wedge g)\le (e\vee f)\wedge g$ holds automatically, so the important one is the converse.

If $L$ admits a least and a greatest element they will be denoted by $0$ and $1$, respectively. A lattice is \emph{complete} if for every subset $S\subseteq L$, the supremum $\vee \{a:a\in S\}$ and the infimum $\wedge\{a:a\in S\}$ exist (see \cite[Definition 2.12]{Maeda21970}). Let $a$ and $b$ be elements in $L$ with $a \leq b$. Along this subsection, the set of all elements $c$ in $L$ with $a\leq c \leq b$ will be denoted by $[a, b]$.

The set $\mathcal{P}(M)$ of all projections in a JBW$^*$-algebra $M$ is a complete lattice with respect to the partial order defined by $p\leq q$ if $p\circ  q = p$. (Observe that $p\le q$ in this sense if and only if $p\le q$ as tripotents in the JBW$^*$-triple $M$, cf. Proposition~\ref{P:order char}.) Furthermore, the mapping $p\mapsto p^{\perp}:=1-p$ is  order reversing and satisfies $p^{\perp\perp} =p$, and $p^{\perp}$ is a complement for $p$ (i.e. $p\wedge p^{\perp}=0$ and $p\vee p^{\perp}=1$). Therefore $\mathcal{P}(M)$ is \emph{orthocomplemented} in the sense employed in \cite[5.1.2]{hanche1984jordan} and \cite[Definition 29.1]{Maeda21970}. Considering the usual orthogonality in $\mathcal{P}(M)$ (i.e., $p\perp q$ if $p\circ q =0$; this takes place if and only if $p$ and $q$ are orthogonal as tripotents), it follows from \cite[Theorem 2.9, Corollary 2.10 and Theorem 29.13]{Maeda21970} that the lattice $\mathcal{P}(M)$ is an \emph{orthomodular} lattice, that is, $p\perp q$ implies that $ r = p \vee (r \wedge q)$ for every $r\geq p$ (see \cite[5.1.2]{hanche1984jordan} and \cite[Definition 29.12 and Theorem 29.13]{Maeda21970}).

A projection $p\in \mathcal{P}(M)$ is \emph{modular} if the projection lattice $[0, p]=\mathcal{P}(M_2(p))$ of $M_2(p)$ is modular. If the unit element of a JBW$^*$-algebra $M$ is modular, $M$ itself is called modular. The same definition applies to JBW-algebras (which are just the self-adjoint parts of JBW$^*$-algebras).

It turns out that modularity can be viewed as a stronger version of finiteness. More precisely, we have the following theorem.

\begin{thm}\label{T:modular algebras}
Let $M$ be a JBW$^*$-algebra. Then $M$ is modular if and only if it is triple-isomorphic to a JBW$^*$-algebra of the form
 $$V\oplus^{\ell_\infty} \left(\bigoplus^{\ell_\infty}_{j\in J}L^\infty(\mu_j,C_j)\right) \oplus^{\ell_\infty}H(W,\alpha),$$
    where
    \begin{itemize}
        \item $V$ is a finite von Neumann algebra;
        \item $J$ is a (possibly empty) set;
        \item $\mu_j$ is a probability measure for each $j\in J$;
        \item for each $j\in J$ the space $C_j$ is either a spin factor, or a finite-dimensional Cartan factor of type $3$ or $6$, or $C_j=B(H)_a$ for a Hilbert space $H$ of finite even dimension;
        \item $W$ is a (possibly trivial) continuous finite von Neumann algebra (i.e., a type II$_1$ von Neumann algebra) and $\alpha$ is a central linear involution on $W$ commuting with $*$.
    \end{itemize}
\end{thm}

This theorem will be proved at the end of this subsection using several properties and characterizations of modularity.

Let us first recall a stronger notion of finiteness, introduced by D.~Topping in \cite[\S 10]{Topping1965}, closely related to  modularity of projections. In \cite{Topping1965} the author deals with JBW-algebras, but we translate it to the language of JBW$^*$-algebras in order to be consistent with the rest of this paper. (Recall that JBW-algebras are exactly self-adjoint parts of JBW$^*$-algebras \cite{Wright1977}, so the reformulations make no harm.)

Let $N$ be a JBW$^*$-algebra. 
Two projections $p,q\in N$ are called equivalent (denoted by $p\stackrel{s}{\sim} q$) if there is a finite sequence $s_l, s_2 ,\ldots, s_n$ of self-adjoint symmetries (i.e. $s_j^*=s_j$ and $s_j^2=1$) such that $Q(s_1) \ldots Q(s_n) (p) =q$, where $Q(s_j) (x) =\{s_j,x,s_j\} = 2 (s_j \circ x) \circ s_j - s_j^2 \circ x$ for all $x\in N$ (cf. \cite[\S 10]{Topping1965}, \cite[5.1.4]{hanche1984jordan}, \cite[\S 3]{AlfsenShultzGeometry2003}). We shall write $p\stackrel{s1}{\sim} q$ (respectively, $p\stackrel{s2}{\sim} q$) when they are  exchanged by a symmetry (respectively, two symmetries) in the way we have seen before. Contrary to Murray-von Neumann equivalence, the relation $\stackrel{s}{\sim}$ 
has the property that $p \stackrel{s}{\sim} q$ in $N$ implies $1-p \stackrel{s}{\sim} 1-q$. 
If $p,q$ are two projections in a von Neumann algebra $M$, by  \cite[Proposition 6.56]{AlfsenShultzStateSpace2001} $p\stackrel{s}{\sim}q$ if and only if $p$ and $q$ are unitarily equivalent (i.e. there exists a unitary element $u\in M$ such that $u p u^* = q$). 
In particular,  $p\stackrel{s}{\sim} q$ implies $p\sim q$.

Topping considered in \cite[\S 13]{Topping1965} a property related to finiteness for the relation $\stackrel{s}{\sim}$. A projection $p$ in a JBW-algebra $N$ satisfies property $(F)$ if for each projection $q$ with $q\stackrel{s}{\sim} p$ and $q\leq p$ we have $p=q$. It is known that in a general JBW-algebra, a projection having property $(F)$ may have subprojections which violate $(F)$. We observe that the unit of $N$ always satisfies $(F)$ and that any finite projection in a von Neumann algebra satisfies $(F)$.

There is an equivalent description of the relation $\stackrel{s}{\sim}$. 
A result by S. Zanzinger (\cite[Theorem 4.1]{zanzinger1994}) proves that two projections $p, q$ in a JBW$^*$-algebra
are equivalent (i.e. $p\stackrel{s}{\sim} q$) if and only if they are \emph{perspective}, i.e., if they have common complement in the projection lattice, that is, there is a projection $r$ with $p\wedge r=q\wedge r=0$ and $p\vee r=q\vee r=1$.

The following proposition collects some known characterizations of modular JBW$^*$-algebras.

\begin{prop}\label{p characterization of modular JBW*-algebras} For a JBW$^*$-algebra $N$ the following are equivalent:\begin{enumerate}[$(i)$]\item $N$ is modular (i.e., $N$ has a modular projection lattice);
\item If $p$ and $q$ are projections in $N$ with $p\leq q$ and $p$ and $q$ are perspective, then $p=q$;
\item Every projection in $N$ satisfies property $(F)$, that is, if $q$ and $p$ are projections in $N$ with $q\stackrel{s}{\sim} p$ and $q\leq p$ we have $p=q$;
\item $N$ contains no copy of $B(H)_s$ for a separable infinite dimensional complex Hilbert space $H$;
\item Every orthogonal family of projections in $N$, any two of which are exchanged by a symmetry in $N$ (i.e., mutually $\stackrel{s1}{\sim}$-related), is finite.
\end{enumerate}
If $N$ is even a JW$^*$-algebra, then the previous assertions are equivalent also to the following one.
\begin{enumerate}[$(vi)$]
    \item Every orthogonal family of $\stackrel{s}{\sim}$-equivalent projections in $N$ is finite;
\end{enumerate}
\end{prop}

\begin{proof}
$(i)\Leftrightarrow(ii)$ This follows from \cite[Proposition 5.1.3]{hanche1984jordan} as the projection lattice is orthomodular.

$(ii)\Leftrightarrow(iii)$ This follows from the above-quoted result \cite[Theorem 4.1]{zanzinger1994}.

$(i)\Leftrightarrow(iv)\Leftrightarrow(v)$ This follows from \cite[Theorem 7.6.3]{hanche1984jordan}.

Finally, the equivalence of the additional assertion $(vi)$ for JW$^*$-algebras follows from \cite[Proposition 14]{Topping1965}.
\end{proof}

Now we are ready to proof the `if part' of Theorem~\ref{T:modular algebras}.

\begin{proof}[Proof of the `if part' of Theorem~\ref{T:modular algebras}]
Since modularity is obviously preserved by $\ell_\infty$-sums, it is enough to prove that any JBW$^*$-algebra triple-isomorphic to one of the individual summands is modular.

First, assume that $M$ is a JBW$^*$-algebra triple-isomorphic to a finite von Neumann algebra $V$.
It means that there is a unitary element $e\in V$ such that $M$ is Jordan-$*$-isomorphic to the JBW$^*$-algebra $V_2(e)$. Hence we may assume that $M=V_2(e)$. Let us define on $V$ a new product by setting
$$x\cdot_e y= xe^*y,\quad x,y\in V.$$
This product together with the involution $*_e$ (recall that it is defined by $x^{*_e}=ex^*e$) makes $V$ be a von Neumann algebra generating the JBW$^*$-algebra structure of $M=V_2(e)$. This von Neumann algebra is finite (by Lemma~\ref{L:finite tripotents}$(a,c)$).
Hence $M$ is modular by \cite[Theorem V.1.37]{Tak}. (Alternatively, we can use Proposition~\ref{p characterization of modular JBW*-algebras}. Indeed, the equivalence $\stackrel{s}{\sim}$ implies the Murray-von Neumann equivalence $\sim$, thus any finite projection in a von Neumann algebra satisfies  property $(F)$.)

Next assume that $M$ is triple-isomorphic to the summand $N=H(W,\alpha)$. It means that there is a unitary element $e\in N$ such that $M$ is Jordan-$*$-isomorphic to the JBW$^*$-algebra $N_2(e)$.
So, assume $M=N_2(e)$. Note that $e$ is also a unitary element of $W$. In the same way as in the previous paragraph we may introduce on $W$ a new structure of a von Neumann algebra such that $M$ is then a JBW$^*$-subalgebra.
Recall that we assume that $W$ is finite, hence $W_2(e)$ is modular as in the previous paragraph. Since the validity of condition $(v)$ from Propostion~\ref{p characterization of modular JBW*-algebras} is clearly inherited by JBW$^*$-subalgebras, we deduce that $M$ is modular.

Finally, assume that $M$ is triple-isomorphic to $N=L^\infty(\mu,C)$, where $C$ is either a finite-dimensional JBW$^*$-algebra or a spin factor.
Again, we assume $M=N_2(\uu)$ for a unitary element $\uu\in N$. Observe that $\uu(\omega)$ is a unitary element of $C$ for $\mu$-almost all $\omega$.

We will show that condition $(v)$ of Proposition~\ref{p characterization of modular JBW*-algebras} is fulfilled in $M$. We proceed by contradiction. Assume that there is an infinite orthogonal sequence $(\p_n)$ of projections in $M$ such that $\p_n\stackrel{s1}{\sim}\p_m$ in $M$ whenever $m,n\in\en$. It means that for $\mu$-almost all $\omega$ $(\p_n(\omega))_{n\in\en}$ is an orthogonal sequence of tripotents in $C$ below $\uu(\omega)$ and, moreover, there are $\vv_{m,n}\in M$ such that $\vv_{m,n}(\omega)$ is unitary (actually a symmetry) almost everywhere and $\p_n=\J{\vv_{m,n}}{\p_m}{\vv_{m,n}}$, i.e., 
$$\p_n(\omega)=\J{\vv_{m,n}(\omega)}{\p_m(\omega)}{\vv_{m,n}(\omega)}\qquad\mu\mbox{-a.e.}$$
Since we work only with a countable family of elements, we may without loss of generality assume that the mentioned equalities hold everywhere. In particular, if $\p_n(\omega)\ne0$ for some $\omega$, then $\p_m(\omega)\ne0$ for each $m\in\en$. It follows that there is an infinite sequence of pairwise orthogonal nonzero tripotents in  $C$, which is impossible. (Note that in a spin factor there do not exist three nonzero pairwise orthogonal tripotents by Proposition~\ref{P:relations in spin}$(3)$, and in a finite-dimensional JB$^*$-triple there cannot be an infinite family of pairwise orthogonal tripotents for trivial reasons.)
This completes the proof that $M$ is modular.
\end{proof}

To prove the converse implication we will use some consequences of a result due to S.A. Ayupov.  Recall that
a JW$^*$-algebra $N$ inside some $B(H)$ is said to be \emph{reversible} if $a_1 \ldots a_n + a_n \ldots a_1\in N$ whenever
$a_1,\ldots a_n \in N.$ Ayupov proved in \cite[Theorem 3]{ayupov1982} that a reversible JW$^*$-algebra is modular if and only if its
enveloping von Neumann algebra is finite.

The next result is a minor generalization o Ayupov's result.

\begin{lemma}\label{L:ayupov}
Let $V$ be a von Neumann algebra and let $M\subset V$ be a reversible JBW$^*$-subalgebra. Assume that the von Neumann algebra generated by $M$ coincides with the whole $V$ and that $M$ is triple-isomorphic to a modular JBW$^*$-algebra. Them $V$ is finite.
\end{lemma}

\begin{proof}
Observe that the assumptions yield that $1\in M$ and that there is a unitary element $e\in M$ such that the JBW$^*$-algebra $M_2(e)$ is modular. Then $M_2(e)$ is a JBW$^*$-subalgebra of the JBW$^*$-algebra $V_2(e)$. Moreover, $e$ is unitary in $V$ and $V_2(e)$ carries a structure of a von Neumann algebra with the product $\cdot_e$ defined by $x\cdot_e y=x e^* y$.

Let us prove that the von Neumann subalgebra of $V_2(e)$ generated by $M_2(e)$ coincides with $V_2(e)$. To this end denote the generated von Neumann algebra by $W$. It is enough to show that $W$ is stable under the original product and involution in $V$.

We know that $\{e,e^*,1\}\subset M\subset W$. Moreover, for any $x\in W$ we have
$$xe=x e^* e^2=x\cdot_e e^2\in W$$
as  $e^2\in M\subset W$. 
Hence, given $x,y\in W$ we have
$$xy=xee^*y=xe \cdot_e y\in W.$$
Finally, for any $x\in W$ we have
$$x^*=e^*ex^*e e^*=\J{e^*}{x^{*_e}}{e^*}\in W$$
as $W$ is a subtriple of $V$.

So, we deduce that $W=V=V_2(e)$.

Moreover, $M_2(e)$ is reversible in $V_2(e)$. Indeed, let $a_1,\dots,a_n\in M_2(e)$.  Since $M$ is a reversible JBW$^*$-subalgebra of $V$ we have 
\begin{multline*}
    a_1\cdot_e a_2\cdot_e\ldots\cdot_e a_n+ a_n\cdot_e\ldots\cdot_e a_2\cdot_e a_1\\=a_1e^*a_2e^*\dots e^*a_n+a_ne^*\dots e^*a_2e^*a_1\in\! M\!=\!M_2(e).\end{multline*}

Hence $V_2(e)$ is finite by Ayupov's theorem \cite[Theorem 3]{ayupov1982}. It follows that $V$ is finite by Lemma~\ref{L:finite tripotents}. 
\end{proof}

\begin{lemma}\label{L:modular is finite}
Any modular JBW$^*$-algebra is finite.
\end{lemma}

\begin{proof}
In view of Theorems~\ref{T:synthesis}
it is enough to prove that no properly infinite JBW$^*$-algebra is modular. Moreover, it is enough to prove this for the individual summands of $M_2$ from Theorem~\ref{T:synthesis}.

If $M$ is a JBW$^*$-algebra triple-isomorphic to a properly infinite von Neumann algebra, it cannot be modular by Lemma~\ref{L:ayupov}. 

It remains to analyze the triples of the form $M=A\overline{\otimes}B(H)_a$ where $A=L^\infty(\mu)$ and $H$ is an infinite-dimensional Hilbert space. Let us show how to apply Lemma~\ref{L:ayupov} in this case.

Recall that $H=\ell^2(\Gamma)$ for an infinite set (and the conjugation and transpose are considered with respect to the canonical basis). Since $\Gamma$ is infinite, there is a subset $\Delta\subset\Gamma$ and a bijection $\eta:\Delta\to\Gamma\setminus\Delta$. Let $u\in B(H)$ be the unitary operator acting on the canonical vectors by
$$e_\gamma\mapsto e_{\eta(\gamma)}\mbox{ and }e_{\eta(\gamma)}\mapsto -e_\gamma\mbox{ for }\gamma\in\Delta.$$
Then $u\in B(H)_a$. Further, $U=1\otimes u$ is a unitary element in $A\overline{\otimes}B(H)_a$, hence also in the von Neumann algebra $V=A\overline{\otimes}B(H)$. Let us equip $M$ with the structure of a  JBW$^*$-algebra given by $M_2(U)$. Then $M$ is a JBW$^*$-subalgebra of the JBW$^*$-algebra $V_2(U)$. Equip $V_2(U)$ with the structure of a von Neumann algebra as above.
We claim that $M$ is reversible in $V_2(U)$. Indeed, fix $x_1,\dots,x_n\in M$. Then
$$\begin{aligned}
(x_1\cdot_U\ldots\cdot_U x_n)^t&=(x_1 U^* \dots U^* x_n)^t=
x_n^t(U^*)^t\dots (U^*)^tx_1^t\\&=(-1)^{2n-1}x_nU^*\dots U^*x_1=
-x_n\cdot_U\ldots\cdot_U x_1,\end{aligned}$$
hence we deduce easily that $M$ is reversible.

Hence, to prove that no JBW$^*$-algebra triple-isomorphic to $M$ is modular, it is enough to show that the von Neumann algebra generated by $M$ is infinite and then use Lemma~\ref{L:ayupov}.

Let $(\gamma_n)$ be a one-to-one sequence in $\Delta$. Let $u_n$ be the partial isometry with $p_i(u_n)=p_f(u_n)=\span\{e_{\gamma_n},e_{\eta(\gamma_n)}\}$ defined by 
$$e_{\gamma_n}\mapsto e_{\eta({\gamma_n})}\mbox{ and } e_{\eta({\gamma_n})}\mapsto -e_{\gamma_n}.$$
Then $u_n\in B(H)_a$ and $u_n\le u$, hence $1\otimes u_n$ is a projection in the von Neumann algebra $V_2(U)$ contained in $M$. Moreover, these projections are mutually orthogonal.

We are going to show that these projections are mutually equivalent (in the Murray-von Neumann sense) in the envelopping von Neumann algebra. To this end fix two distinct numbers $m,n\in\en$. Let $v$ be the partial isometry with $p_i(v)=p_f(v)=\span\{e_{\gamma_n},e_{\gamma_m},e_{\eta(\gamma_n)},e_{\eta(\gamma_m)}\}$ defined by
$$e_{\gamma_n}\mapsto e_{\eta(\gamma_m)}, e_{\gamma_m}\mapsto e_{\eta(\gamma_n)}, e_{\eta(\gamma_n)}\mapsto -e_{\gamma_m} \mbox{ and } e_{\eta(\gamma_m)}\mapsto -e_{\gamma_n}.$$
Then $v\in B(H)_a$. Moreover, 
set $w=vu^*u_n$. Then $w$ is a partial isometry with $p_i(w)=p_i(u_n)$ and $p_f(w)=p_i(u_m)$. Since
$$1\otimes w=(1\otimes v)U^*(1\otimes u_n)=(1\otimes v)\cdot_U(1\otimes u_n),$$
$1\otimes w$ is a partial isometry which belongs to the envelopping von Neumann algebra of $M$ in $V_2(U)$. Moreover,
$$(1\otimes w)^{*_U} \cdot_U(1\otimes w)=U(1\otimes w^*)UU^*(1\otimes w)=U(1\otimes p_i(w))=1\otimes up_i(u_n)=1\otimes u_n$$
and
$$(1\otimes w) \cdot_U(1\otimes w)^{*_U}
=(1\otimes w)U^*U(1\otimes w^*)U=(1\otimes p_f(w))U=1\otimes p_f(u_m)u=1\otimes u_m.$$

Hence $1\otimes u_n\sim 1\otimes u_m$ in the envelopping von Neumann algebra of $M$ in $V_2(U)$. It follows from Lemma~\ref{L:finite projections basic facts}$(e)$ that the envelopping von Neumann algebra is infinite. This completes the proof.
\end{proof}

\begin{proof}[Proof of the `only if' part of Theorem~\ref{T:modular algebras}]
Assume that $M$ is a modular JBW$^*$-algebra. By Lemma~\ref{L:modular is finite} it is finite, so it is isomorphic to a JBW$^*$-algebra of the form $M_1$ from Theorem~\ref{T:synthesis}. It remains to prove that the von Neumann algebra $W$ from the summand $H(W,\alpha)$ is finite and, moreover, there is no summand of the form $A\overline{\otimes}B(H)_s$ for an infinite-dimensional Hilbert space $H$. 

In both cases we may apply Lemma~\ref{L:ayupov}. Indeed, $H(W,\alpha)$ is a reversible JBW$^*$-subalgebra of $W$ and, moreover, the envelopping von Neumann algebra contains each projection in $W$ (by Lemma~\ref{L:H(W,alpha) construction of tripotent}), hence it equals $W$. So, $W$ is finite by Lemma~\ref{L:ayupov}.

Similarly, $A\overline{\otimes}B(H)_s$ is a reversible JBW$^*$-subalgebra of the von Neumann algebra $A\overline{\otimes}B(H)$. Again, by Lemma~\ref{L:C3 construction of tripotent} the envelopping von Neumann algebra contains all projections, so it equals to $A\overline{\otimes}B(H)$. Since this von Neumann algebra is infinite, Lemma~\ref{L:ayupov} shows that
$A\overline{\otimes}B(H)_s$ is not triple-isomorphic to a modular JBW$^*$-algebra. This completes the proof.
\end{proof}

\begin{remark}\label{rem:modularity}
\begin{enumerate}[$(1)$]
\item It follows from Theorems~\ref{T:modular algebras} and~\ref{T:synthesis} that modularity is a stronger property than finiteness, i.e., any modular JBW$^*$-algebra is finite, but there are finite JBW$^*$-algebras which are not modular.
In particular, $A\overline{\otimes} B(H)_s$ where $H$ is an infinite-dimensional Hilbert space is such an example.
\item It follows from Theorem~\ref{T:modular algebras} that modularity of 
JBW$^*$-algebras is a triple-property, i.e., it is preserved by triple isomorphisms. In particular, if $M$ is a modular JBW$^*$-algebra and $u\in M$ is any tripotent, then $M_2(u)$ is again a modular JBW$^*$-algebra (cf. Lemma~\ref{L:modular is finite} and Lemma~\ref{L:finite tripotents}$(d)$).
\end{enumerate}
\end{remark}

The previous remark says, in particular, that the modularity can be consistently defined for JBW$^*$-triples. In accordance with the abstract setting of \cite[p. 141]{BattagliaOrder1991} we say that a tripotent $u$ in a JBW$^*$-triple $M$ is \emph{modular} if the JBW$^*$-algebra $M_2(u)$ is modular. Moreover, it is natural to call a JBW$^*$-triple {\em modular} if any of its tripotents is modular. (This does not exactly match the setting of \cite{BattagliaOrder1991}, but it is equivalent as explained below.) At the moment we just observe that a JBW$^*$-algebra is modular as a JBW$^*$-algebra if and only if it is modular as a JBW$^*$-triple (by Remark~\ref{rem:modularity}$(2)$) and that any modular JBW$^*$-triple is finite (by Lemma~\ref{L:modular is finite}).

\begin{prop}\label{P:modular triple}
Let $M$ be a JBW$^*$-triple. Then $M$ is modular if and only if it is triple-isomorphic to $M_1\oplus^{\ell_\infty}M_3$, where
\begin{itemize}
    \item $M_1$ is a modular JBW$^*$-algebra, hence it has the form from Theorem~\ref{T:modular algebras};
    \item $M_3$ has the form as in Theorem~\ref{T:synthesis}.
\end{itemize}
\end{prop}

\begin{proof}
Assume that $M$ is modular. Then it is finite, so by Theorem~\ref{T:synthesis} it can be decomposed as $M_1\oplus^{\ell_\infty}M_3$, where $M_1$ and $M_3$ have the meaning from Theorem~\ref{T:synthesis}. Moreover, since $M_1$ is necessarily modular, it must have the form from Theorem~\ref{T:modular algebras}.

To prove the converse we first recall that a modular JBW$^*$-algebra is a modular JBW$^*$-triple as explained above. 

Further, let $V$ be a a von Neumann algebra and let $p\in V$ be a finite projection. If $u\in pV$ is any tripotent, then $(pV)_2(u)$ is triple-isomorphic to a finite von Neumann algebra (see Proposition~\ref{P:Peirce2 in pV}), hence it is modular by Theorem~\ref{T:modular algebras}. It follows that $pV$ is modular.

Finally, if $D$ is a finite-dimensional JB$^*$-triple appearing in the representation of $M_3$, then $D$ is a subtriple of a finite-dimensional JB$^*$-algebra $C$ (if $D=B(H)_a$ with $\dim H$ finite odd, take $C=B(H)$; if $D=C_5$, take $C=H_3(\O)$). Since $L^\infty(\mu,C)$ is modular by Theorem~\ref{T:modular algebras}, we conclude that also the subtriple $L^\infty(\mu,D)$ is modular.
\end{proof}

We finish by the following proposition saying that our definition of modular JBW$^*$-triples is equivalent with the one used in \cite{BattagliaOrder1991}.

\begin{prop}
Let $M$ be a JBW$^*$-triple. Then $M$ is modular if and only if it admits a complete modular tripotent.
\end{prop}

\begin{proof}
The `only if part' is obvious. Let us prove the `if part'. Assume that $M$ admits a complete modular tripotent. By Lemma~\ref{L:modular is finite} this tripotent is finite, so $M$ is finite by Proposition~\ref{P:finite one complete}. So, by Theorem~\ref{T:synthesis} $M=M_1\oplus^\infty M_3$ (with the notation from the quoted theorem). By Proposition~\ref{P:modular triple} the summand $M_3$ is modular. Further, $M_1$ admits a complete modular tripotent. Since $M_1$ is a finite JBW$^*$-algebra, any complete tripotent is unitary (cf. Proposition~\ref{P:le0=le2}). 
So, $M_1$ admits a unitary modular tripotent, thus $M_1$ is modular by Theorem~\ref{T:modular algebras} (cf. Remark~\ref{rem:modularity}$(2)$).
This completes the proof.
\end{proof}

\section*{Acknowledgement}

We are grateful to the anonymous referees for several useful comments and suggestions.

\def\cprime{$'$} \def\cprime{$'$}


\end{document}